\documentclass[11pt,reqno]{amsart}
\usepackage{enumitem}
\usepackage{stmaryrd}
\usepackage[paperheight=11in,paperwidth=8.5in,left=1in,top=1in,right=1.25in,bottom=1.25in,]{geometry}
\usepackage{relsize}
\usepackage{rotating}
\usepackage[table]{xcolor}

\definecolor{darkred}{rgb}{0.55, 0.0, 0.0}

\usepackage[super]{nth}
\usepackage[none]{hyphenat}
\usepackage{hyperref}
\hypersetup{
	colorlinks,
	citecolor=black,
	filecolor=black,
	linkcolor=black,
	urlcolor=blue
}
\usepackage[capitalize, nameinlink, noabbrev, nosort]{cleveref}
\usepackage{graphicx, amssymb, amsfonts, amsmath,MnSymbol}
\usepackage{epstopdf} 
\usepackage{youngtab}
\usepackage[boxsize=.35 em]{ytableau}
\usepackage{comment}
\usepackage{setspace}
\usepackage{multicol}
\usepackage{dashbox}

\usepackage{tikz}   
\usetikzlibrary{arrows,calc,intersections,matrix,positioning,through,decorations.pathreplacing}
\usepackage{tikz-cd}
\usepackage{diagbox}
\usepackage{blkarray}
\tikzset{commutative diagrams/.cd,every label/.append style = {font = \normalsize}}

\usepgflibrary{shapes.geometric}
\usetikzlibrary{shapes.geometric}
\usetikzlibrary{backgrounds}

\newtheorem*{conjecture*}{Conjecture}
\newtheorem*{theorem*}{Theorem}
\newtheorem{theorem}{Theorem}[section]
\newtheorem{lemma}[theorem]{Lemma}
\newtheorem{proposition}[theorem]{Proposition}
\newtheorem{corollary}[theorem]{Corollary}
\newtheorem{conjecture}[theorem]{Conjecture}

\theoremstyle{definition}
\newtheorem{definition}[theorem]{Definition}
\newtheorem{example}[theorem]{Example}

\newtheorem{observation}[theorem]{Observation}
\newtheorem{remark}[theorem]{Remark}
\newtheorem{notation}[theorem]{Notation}

\newcommand{\lr}[1]{\langle #1 \rangle}

\setlist[itemize]{leftmargin=*}
\setlist[enumerate]{leftmargin=*}

\DeclareMathOperator{\Conf}{Conf}
\DeclareMathOperator{\Gr}{Gr}
\DeclareMathOperator{\arr}{arr}
\DeclareMathOperator{\wt}{wt}
\DeclareMathOperator{\Mat}{Mat}
\DeclareMathOperator{\colsp}{colsp}
\DeclareMathOperator{\rowsp}{rowsp}
\DeclareMathOperator{\Amp}{Amp}
\newcommand{\Grk}{\Gr_{k,n}^{\geq 0}}
\newcommand{\Pds}{\mathcal{P}^{\mathrm{ds}}}
\newcommand{\Ads}{\mathcal{A}^{\mathrm{ds}}}
\newcommand{\Abds}{\widehat{\mathcal{A}^{\mathrm{ds}}_m}}
\newcommand{\const}{\diamond}
\newcommand{\OO}{\mathcal{O}}
\newcommand{\F}{\mathbb{F}}

\newcommand{\PPP}{\mathbb{P}}
\newcommand{\PP}{\mathbb{M}}
\newcommand{\Z}{\mathbb{Z}}
\newcommand{\C}{\mathbb{C}}
\newcommand{\R}{\mathbb{R}}
\newcommand{\CC}{\mathbb{C}}
\newcommand{\A}{\mathcal{A}}
\newcommand{\Bb}{B_{bd}}
\newcommand{\Bi}{B_{int}}
\def\mVRC_#1{\mathcal{C}^m_{#1}}
\def\kVRC_#1{\mathcal{C}^k_{#1}}
\def\VRC_#1^#2{\mathcal{C}^{#2}_{#1}}
\def\tmVRC_#1{\tilde{\mathcal{C}}^m_{#1}}
\def\zVRC{[\bv, \bR]^{\partial = \bz}}
\def\mzVRC_#1{\mathcal{C}^{\partial= \bz}_{#1}}
\def\bbu{\mathbf{u}}
\def\bv{\mathbf{v}}
\def\bR{\mathbf{R}}
\def\bz{\mathbf{z}}
\def\bw{\mathbf{w}}
\def\bD{\mathbf{D}}
\def\bE{\mathbf{E}}
\def\bd{\mathbf{d}}
\def\root{r}
\DeclareMathOperator{\rt}{rt}
\def\ext{{\operatorname{ext}}}
\def\shuf{*}
\def\gProm{\psi}
\def\aProm{\Psi}
\DeclareMathOperator{\spn}{span}
\DeclareMathOperator{\sign}{sign}
\DeclareMathOperator{\GC}{GC}
\DeclareMathOperator{\GL}{GL}
\def\m{\Omega}
\def\comp{\Upsilon}
\def\aa{\mathrm{A}}
\def\bb{\mathrm{B}}
\def\cc{\mathrm{C}}
\def\dd{\mathrm{D}}
\def\ee{\mathrm{E}}

\newcommand{\br}{\,|\,}
\def\mcy{\mathcal{Y}}
\def\mcv{\mathcal{V}}
\def\mcu{\mathcal{U}}
\def\mcb{\mathcal{B}}
\def\mint{\operatorname{IN}_m}
\def\pt{\operatorname{pt}}

\def\Xcal{\mathcal{X}}
\def\Acal{\mathcal{A}}

\def\Fcal{\mathcal{F}}

\newcommand{\txx}{\widetilde{\mathbf{x}}}
\def\pos{\mathcal{M}}
\def\bsig{\boldsymbol{\sigma}}

\def\tZ{\tilde{Z}}

\title{Plabic tangles and cluster promotion maps}
\author[C.~Even-Zohar, M.~Parisi, M.~Sherman-Bennett, R.~Tessler, L.~Williams]{Chaim Even-Zohar, Matteo Parisi, Melissa Sherman-Bennett, Ran Tessler and Lauren Williams}

\begin{document}

\begin{abstract}
Inspired by the BCFW recurrence for tilings of the amplituhedron, we introduce the general framework of \emph{plabic tangles} that utilizes plabic graphs to define rational maps between products of Grassmannians called \emph{promotions}. The central conjecture of the paper is that promotion maps are quasi-cluster homomorphisms, which we prove for several classes of promotions. In order to define promotion maps, we utilize $m$-vector-relation configurations ($m$-VRCs) on plabic graphs. We relate $m$-VRCs to the degree (a.k.a `intersection number') of the amplituhedron map on positroid varieties and characterize all plabic trees with intersection number one and their VRCs. Finally, we show that promotion maps admit an operad structure and, supported by the class of \emph{$4$-mass box} promotions, we point at new positivity properties for non-rational maps beyond cluster algebras. Promotion maps have important connections to the geometry and cluster structure of the amplituhedron and singularities of scattering amplitudes in planar $\mathcal{N}=4$ super Yang--Mills theory.
\end{abstract}
\maketitle

\setcounter{tocdepth}{1}
\tableofcontents

\section{Introduction}

The \emph{positive Grassmannian} $\Gr_{k,n}^{\scriptscriptstyle\geq 0}$ is the subset of the real Grassmannian in which all Pl\"ucker coordinates are nonnegative.  It has a decomposition into 
\emph{positroid cells} \cite{postnikov} indexed by \emph{plabic graphs} of type $(k,n)$, planar bicolored graphs which are embedded in a disk. 
The (tree) \emph{amplituhedron} $\mathcal{A}_{n,k,m}(Z)$ is the image of the
positive Grassmannian $\Gr_{k,n}^{\scriptscriptstyle\geq 0}$ under the
\emph{amplituhedron map} ${\tilde{Z}: \Gr_{k,n}^{\scriptscriptstyle\geq 0} \to \Gr_{k,k+m}}$.
It was introduced by
Arkani-Hamed and Trnka \cite{arkani-hamed_trnka} in order to give a
geometric interpretation of the \emph{BCFW recurrence} \cite{BCFW} for
\emph{scattering amplitudes} in $\mathcal{N}=4$ super Yang Mills theory (SYM).
In particular, each way of iterating the BCFW recurrence gives rise
to a collection of positroid cells of $\Gr_{k,n}^{\scriptscriptstyle\geq 0}$. Arkani-Hamed and Trnka conjectured that the amplituhedron map $\tZ$ is injective on these cells, that the $\tZ$-images of these cells
`tile' the $m=4$ amplituhedron $\A_{n,k,4}(Z)$, and that one can compute 
scattering amplitudes  by
summing the `volumes' of the tiles.

In our previous work \cite{even2023cluster},  which was joint with  Lakrec,
we proved the above \emph{BCFW tiling conjecture} of Arkani-Hamed and Trnka.  We also proved 
the  \emph{cluster adjacency conjecture}, which says that
facets of tiles are cut out by collections of \emph{compatible cluster variables}.  To prove these conjectures, we used a graphical recurrence \cite[(3.3)]{Bai:2014cna}  on plabic graphs that involves inserting two smaller plabic graphs into faces of a fixed ``core'' plabic graph, see \cref{fig:tangle}.  In particular, we showed that this graphical recurrence gives rise to a \emph{quasi-cluster homomorphism}
$$\CC(\widehat{\Gr}_{4,N_L}) \otimes \CC(\widehat{\Gr}_{4,N_R}) \to 
\CC(\widehat{\Gr}_{4,n}),$$ that is, a map that respects the \emph{cluster algebras structure} on the Grassmannian, and in particular takes cluster variables to cluster variables.\footnote{up to a Laurent monomial in frozen variables.}

The goal of this paper is to illustrate that the particular graphical recurrence and quasi-cluster homomorphism that we used to prove the BCFW tiling and cluster adjacency conjectures is part of a much vaster framework.
In particular, in \cref{sec:promotion}, we introduce the notion of  \emph{plabic tangle}, which is the data of a plabic graph ``core" $G$ drawn inside an \emph{outer disk}, together with $\ell$ \emph{inner disks}, each of which lies in a face of $G$.  The definition of plabic tangle is inspired by the notion of \emph{planar tangle} \cite{Jones}, and we show that similarly to the case of planar tangles, there is an \emph{operad structure} on plabic tangles, see \cref{sec:categorical_pov}.  The central construction of this paper
is the association of a map between Grassmannians to each (sufficiently nice)
plabic tangle.
We note that while plabic graphs have been previously used to define subsets of the Grassmannian 
\cite{postnikov}, 
or to define functions on Grassmannians or rings of invariants \cite{Kuperberg, FP}, this is to our knowledge the first time that plabic graphs have been used to define \emph{maps} between (products of) Grassmannians or between their coordinate rings.

In order to associate maps to plabic tangles, we start with the framework of \emph{vector relation configurations} (see \cref{sec:promotion-plabic-graphs}), inspired by \cite{AGPR}\footnote{our definitions and results are rather different than those of \cite{AGPR}; see \cref{rem:AGPRcontrast} for a comparison.}.
  An \emph{$m$-vector-relation configuration} ($m$-VRC) on a bipartite plabic graph $G$ (with $n$ boundary vertices, colored black) is an assignment of a vector $v_b \in \CC^m$ to each black vertex and a scalar $r_e \in \CC^*$ to each edge such that the \emph{boundary vectors} $v_1, \dots, v_n$ on vertices $1, \dots n$ span $\CC^m$, and for each white vertex $w$, we have the linear relation
  \[\sum_{e=\{b,w\}} r_e v_b =0.\]
  Suppose $G$ is \emph{($m$-generically) solvable}, that is, a generic configuration of $n$ vectors in $\CC^m$ on the outer boundary of $G$ can be extended to a unique VRC (modulo gauge) on $G$. If we have a plabic tangle with core $G$, where each inner disk $D$ is connected to some black vertices of $G$, and we choose a scaling for our internal vectors $v_b$, we can use the VRC to define a rational map from the Grassmannian $\Gr_{m,n}$ 
associated to the outer boundary of $G$, to the product of Grassmannians $\Gr_{m,D}$ 
associated to the inner disks of $G$.  The pullback of this map is a map on the level of coordinate rings; we refer to these two maps as \emph{geometric} and \emph{algebraic promotion}. 
We will be particularly interested in when geometric promotion is a dominant map\footnote{i.e. its image is Zariski-dense in the codomain.}, in which case we call the tangle \emph{dominant}; see \cref{prop:dominant_solvable} for a combinatorial characterization.

We believe that dominant solvable plabic tangles are the appropriate setting for generalizing our results on the BCFW story.  In particular, 
we believe that they give rise to quasi-cluster homomorphisms.
\begin{conjecture*} [\cref{conj:cluster1}] Let $(G, \bD)$ be a dominant solvable plabic tangle. Then there exists a normalization of the vectors in the $m$-VRCs on $G$ such that the following hold.
\begin{enumerate}
\item \label{item-1} Geometric promotion by $(G, \bD)$ 
\[\gProm: \Gr_{m, n} \;\dashrightarrow\;  \Gr_{m, D^{(1)}} \times \dots \times \Gr_{m, D^{(\ell)}}\]
sends totally positive elements to totally positive elements.
\item Algebraic promotion by $(G, \bD)$ 
\[\aProm= \gProm^*:\CC({\Gr}_{m, D^{(1)}}) \otimes \dots \otimes 
    \CC({\Gr}_{m, D^{(\ell)}})
    \;\to\; \CC({\Gr}_{m, n})\]
    is a quasi-cluster homomorphism, 
    after we freeze some variables on the right-hand side.  
\end{enumerate}
\end{conjecture*}

\cref{conj:cluster1} generalizes the BCFW quasi-cluster homomorphism (reviewed in \cref{sec_bcfw}), and asserts a deep connection between $m$-VRCs on solvable plabic graphs and the cluster structure on the Grassmannian.  Moreover,
we prove it for several interesting infinite families of promotion maps.

\begin{theorem*}
\cref{conj:cluster1} is true in the cases of: \emph{star promotion} ($k=1$, any $m,n$); \emph{spurion promotion} ($k=2$, $m=4$, any $n$); \emph{chain-tree promotion} ($k=3$, $m=4$, any $n$); and \emph{forest promotion} ($k=2$, $m=3$, any $n$). For each promotion, the core is of type $(k,n)$, see \cref{sec:promotion}.
\end{theorem*}

The second part of our paper shows that $m$-VRCs on solvable graphs $G$ are also closely related to inverting the amplituhedron map. Let $\Pi_G \subset \Gr_{k,n}$ be the \emph{positroid variety} of $G$, that is, the Zariski closure of the 
positroid cell $S_G$. We say that 
$G$ and 
$\Pi_G$ have \emph{$m$-intersection number $d$}
if for generic $Z\in \Mat_{n,k+m}$,
the (rational) amplituhedron map $\tilde{Z}:\Gr_{k,n} \dashedrightarrow \Gr_{k,k+m}$ is generically $d$-to-1 on $\Pi_G.$ We prove that $G$ is $m$-generically solvable if and only if $\Pi_G$ has $m$-intersection number 1 and $\dim \Pi_G = km$ (see \cref{cor:int1}). In this situation, the $m$-VRCs on $G$ can be used to generically invert the amplituhedron map on $\Pi_G$ (see \cref{rem:inverting-amplituhedron-map}). More generally, we show that the  intersection number can be computed in terms of $m$-VRCs (see \cref{cor:int-num=num-VRC}). 

In the third part of our paper, we turn our attention to \emph{plabic trees}. This special class already demonstrates the range of promotion maps in our conjecture and the power of the $m$-VRC formulation of intersection number. We characterize the plabic trees which have $m$-intersection number 1, or equivalently, the $m$-generically solvable plabic trees. 

We say that a bipartite plabic tree $G$ is \emph{$m$-balanced} if 
for each edge $e$ of $G$, if we write 
$G\setminus \{e\} = G_1 \sqcup G_2$ (giving ``half" the edge $e$ to each $G_i$), then for each $i=1,2$, we have $$m(k_{G_i}-1) < \dim \Pi_{G_i} \leq mk_{G_i},$$ where $k_{G_i}$ is the $k$-statistic of the tree $G_i$. 
We have the following result, proved using $m$-VRCs.
\begin{theorem*} [\cref{prop:tree1}]
Let $G$ be a bipartite plabic tree of type $(k,km+1)$.  Then
$G$ has $m$-intersection number 1 
if and only if $G$ is {$m$-balanced}. If $G$ is not $m$-balanced, then $G$ has $m$-intersection number $0$.
\end{theorem*}
We note that using this characterization, one can produce for each $m$ various infinite families of solvable plabic trees, and hence infinite families of promotion maps.

Finally, we explore to what extent our constructions, such as the operad framework and the notion of promotion, make sense when we work with a plabic graph $G$ of higher intersection number.  In \cref{sec:4mass} we study  an intersection number $2$ cell called the 
\emph{$4$-mass box}, which has previously arisen in the physics literature\footnote{It associated to a Feynman diagram with one cycle (`loop') and $4$ vertices - hence a `box'. At each vertex a `massive' momentum flows - hence `$4$-mass' \cite{THOOFT}. The \emph{leading singularity} associated to the diagram is encoded by the positroid cell labelled by the plabic graph $G$ \cite{Arkani-Hamed:2012zlh}.}.   
We obtain two ``promotion'' maps from this cell, which cannot be quasi-cluster homomorphisms, because they involve a square root.  On the other hand, these maps still possess an intriguing positivity property, namely 
(1) of \cref{conj:cluster1} above.  This could point to a new algebraic structure beyond the framework of cluster algebras.

\begin{theorem*}[\cref{th:4mb_pos}] 
The $4$-mass box promotion maps $\aProm_{+}, \aProm_{-}$ preserve positivity of cluster variables for $\Gr_{4,N'}$, i.e. if $x$ is a cluster variable for $\Gr_{4,N'}$, then $\aProm_{+}(x), \aProm_{-}(x)$ are positive on $\Gr^{>0}_{4,n}$.
\end{theorem*}

On the physics side, we expect singularities of \emph{Yangian invariants} - building blocks of scattering amplitudes in $\mathcal{N}=4$ SYM - to be images of algebraic promotion maps, generalizing what we have seen with BCFW promotions. Our framework should thus shed light on the conjectural cluster structures and positivity pheonomena that arise in the study of scattering amplitudes, see \cref{rk:physics_1,rk:physics_scattering_2}. Promotions maps should also play an important role in describing the geometry and the cluster structure of amplituhedron tiles and, more generally, images of positroid cells under the amplituhedron map. This will be the subject of a future work.

\noindent{\bf Acknowledgements:~}
We would like to thank Nima Arkani-Hamed and Tsviqa Lakrec for multiple inspiring conversations.
MP is supported by the CMSA at Harvard University.
MSB is supported by the National Science Foundation under Award No.~DMS-2444020.
RT was supported by the ISF grant No.~1729/23.
LW was supported by the National Science Foundation under Award No.
DMS-2152991 until May 12, 2025, when the grant was terminated; she would also like to thank the Radcliffe Institute for Advanced Study, where some of this work was carried out. 
Any opinions, findings, and conclusions or recommendations expressed in this material are
those of the author(s) and do not necessarily reflect the views of the National Science
Foundation.

\section{Background}

\subsection{The Grassmannian and the configuration space}

We start by defining the Grass- mannian and the configuration space.
In this paper we will usually be working with these spaces defined over the 
complex numbers $\C$.

\begin{definition}\label{def:Grass}
The \emph{Grassmannian} $\Gr_{k,n}=\Gr_{k,n}(\F):=\{V \ \vert \ V \subseteq \F^n, \ \dim V = k\}$
is the space of all $k$-dimensional subspaces of
an $n$-dimensional vector space $\F^n$.
We let $\Gr^{\circ}_{k,n}$ denote the subset of the Grassmannian where no Pl\"ucker coordinates vanish, where Pl\"ucker coordinates
are defined below. 
\end{definition}

Let $[n]$ denote $\{1,\dots,n\}$, and $\binom{[n]}{k}$ denote the set of all $k$-element 
subsets of $[n]$.
We can use full rank $k \times n$ matrices $C$ to represent points of $\Gr_{k,n}$.
Recall that the \emph{Pl\"ucker coordinates} give an embedding of the 
Grassmannian into projective space.  More specifically, for $I=\{i_1 < \dots < i_k\} \in \binom{[n]}{k}$, we let $\lr{I}_V=\lr{i_1\,i_2\,\dots\,i_k}_V$ be the $k\times k$ minor of $C$ using the columns $I$. 
The $\lr{I}_V$ 
are called the {\itshape Pl\"{u}cker coordinates} of $V$, and are independent of the choice of matrix
representative $C$ (up to common rescaling). The \emph{Pl\"ucker embedding} 
$V \mapsto \{\lr{I}_V\}_{I\in \binom{[n]}{k}}$
embeds  $\Gr_{k,n}$ into
projective space. When it does not cause confusion, we will identify $C$ with its row-span and drop the subscript $V$ on Pl\"ucker coordinates. 
If $C$ has columns $v_1, \dots, v_n$, we may also identify $\lr{i_1\,i_2\,\dots\,i_k}_V$ with the element $v_{i_1} \wedge v_{i_2} \wedge \dots \wedge v_{i_k}$, hence the Pl\"ucker coordinates are \emph{alternating} in the indices, e.g. $\lr{i_1\,i_2\,\dots\,i_k}=- \lr{i_2\,i_1\,\dots\,i_k}$. 

The torus $(\CC^*)^n$ acts on $\Gr_{k,n}$ by scaling the $i$th column of 
a representative matrix $C$.  Because the element $(t,\dots,t)\in (\CC^*)^n$
scales all Pl\"ucker coordinates by $t^k$, the action of $(\CC^*)^n$ factors
through the torus $T:=(\CC^*)^n/\CC^* \cong (\CC^*)^{n-1}$.

\begin{definition}
The \emph{configuration space} $\Conf_{k,n}:=\GL_k \setminus \{(v_1,\dots,v_n) \ \vert \ v_i \in \PPP^{k-1}\}$ parameterizes the space of $n$ ordered points in $\PPP^{k-1}$, considered up to the action of $\GL_k$. 
We let $\Conf^{\circ}_{k,n}$ denote the subset of the configuration space where any $k$ of the points $v_1,\dots,v_n$ affinely
span a subspace of $\PPP^{k-1}$ of dimension $k-1$. 
\end{definition}

Note that there is an isomorphism $\Gr^{\circ}_{k,n}/T \cong \Conf^{\circ}_{k,n}$.

If $D$ is an index set with a total order, with $|D|=n$, we will also use
$\Gr_{k,D}$ and $\Conf_{k,D}$ to denote
the Grassmannian and configuration spaces $\Gr_{k,n}$ and $\Conf_{k,n}$,
where the columns of a representing matrix are indexed by $D$
(and similarly for $\Gr^{\circ}_{k,D}$ and $\Conf^{\circ}_{k,D}$).

\begin{definition}[Positive Grassmannian and positroids]\label{def:positroid}\cite{lusztig, postnikov}
We say that $V\in \Gr_{k,n}(\R)$ is \emph{totally nonnegative}
     if (up to a global change of sign)
       $\lr{I}_V \geq 0$ for all $I \in \binom{[n]}{k}$.
Similarly, $V$ is \emph{totally positive} if $\lr{I}_V >0$ for all $I
      \in \binom{[n]}{k}$.
We let $\Grk$ and $\Gr_{k,n}^{>0}$ denote the set of
totally nonnegative and totally positive elements of $\Gr_{k,n}(\R)$, respectively.  
$\Grk$ is called the \emph{totally nonnegative}  \emph{Grassmannian}, or
       sometimes just the \emph{positive Grassmannian}.

If we partition $\Grk$ into strata based on which Pl\"ucker coordinates are strictly
positive and which are $0$, we obtain a cell decomposition of $\Grk$
into \emph{positroid cells} $S_G$, which can be indexed by 
(equivalence classes of) plabic graphs $G$ \cite{postnikov}, 
see \cref{app:plabic} for background.
Each positroid cell $S_G$ gives rise to a matroid $\pos_G$, whose bases are precisely
the $k$-element subsets $I$ such that the Pl\"ucker coordinate
$\lr{I}$ does not vanish on $S_G$; this matroid 
is called a \emph{positroid}. The bases of $\pos_G$ are exactly the source sets of \emph{perfect orientations} of $G$, see \cref{def:orientation}.
\end{definition}

The Zariski closure of $S_G$ in $\Gr_{k,n}(\C)$ 
is called 
the \emph{positroid variety} $\Pi_G$. The positroid variety is cut out by setting to $0$ all Pl\"ucker coordinates $\lr{I}$ that vanish on $S_G$, or equivalently, all $\lr{I}$ such that $I \notin \pos_G$ \cite{KLS}. There is a map $\mathbb{B}_G$ called the \emph{boundary measurement map} which is an isomorphism from $(\CC^*)^{\dim \Pi_G}$ to a subset of $\Pi_G$; see \cref{def:bdry-meas-map}. We use $T_G$ to denote the image of $\mathbb{B}_G$, and call $T_G$ the \emph{boundary measurement torus}.

The following lemma will be useful to us. We use the notation $V^\perp$ for the orthogonal complement of a subspace $V$.

\begin{lemma}\label{lem:perp-takes-Pi-G-to-G-op} Let $G$ be a $(k,n)$-plabic graph and $G^{op}$ the plabic graph obtained by switching the colors of all vertices. Then 
\begin{align*} \Pi_G &\to \Pi_{G^{op}}\\
V &\mapsto V^{\perp}
\end{align*}
is an involutive isomorphism of varieties.
\end{lemma}
\begin{proof}
For $I \in \binom{[n]}{k}$, and $I^c = [n] \setminus I$ its complement, we have
\[\lr{I}_V = \pm \lr{I^c}_{V^\perp}\]
where the sign is determined by $I$ (see the discussion around Lemma 1.11 in \cite{Karp}). Thus, $\lr{I}_V =0$ if and only if $\lr{I^c}_{V^\perp} =0$. The matroids $\pos_G$ and $\pos_{G^{op}}$ are duals of each other, meaning that $I \in \pos_G$ if and only if $I^c \in \pos_{G^{op}}$. This shows that the map is well-defined. It is easy to see that it is regular, with regular inverse.
\end{proof}

\subsection{Cluster algebras}

\label{sec:cluster}

Cluster algebras were introduced by Fomin and Zelevinsky in \cite{FZ1}; see \cite{FWZ} for 
an introduction.
We give a quick definition of cluster algebras from quivers here.

A \emph{quiver} is a finite directed graph. A \emph{loop} of a quiver is an edge whose
source and target coincide.

\begin{definition}
[\emph{Quiver Mutation}]
Let $Q$ be a quiver without loops or oriented $2$-cycles.
Let $k$ be a vertex of $Q$.
Following \cite{FZ1}, we define the
\emph{mutated quiver} $\mu_k(Q)$ as follows:
it has the same set of vertices as $Q$,
and its set of arrows is obtained by the following procedure:
\begin{enumerate}
\item for each subquiver $i \to k \to j$, add a new arrow $i \to j$;
\item reverse all arrows with source or target $k$;
\item remove the arrows in a maximal set of pairwise
disjoint oriented $2$-cycles.
\end{enumerate}
\end{definition}

It is not hard to check that mutation is an involution, that is,
$\mu_k^2(Q) = Q$ for each vertex $k$.

\begin{definition}
[\emph{Seeds}]
\label{def:seed0}
Choose $s\geq r$ positive integers.
Let $\Fcal$ be an \emph{ambient field}
of rational functions
in $r$ independent
variables
over
$\CC(x_{r+1},\dots,x_s)$.
A \emph{seed} in~$\Fcal$ is
a pair $(\txx, Q)$, where
\begin{itemize}
\item
$\txx = (x_1, \dots, x_s)$ forms a free generating
set for
$\Fcal$,
and
\item
$Q$ is a quiver on vertices
$1, 2, \dots,r, r+1, \dots, s$,
whose vertices $1,2, \dots, r$ are called
\emph{mutable}, and whose vertices $r+1,\dots, s$ are called \emph{frozen}.
\end{itemize}
We call~$\txx$ the
\emph{cluster} of a seed $(\txx, Q)$, and its elements are called \emph{cluster variables}.
The variables $c=\{x_{r+1},\dots,x_s\}$ are called
	\emph{frozen} (or \emph{coefficient variables}).
We let $\PP$ denote the group of Laurent monomials in the frozen variables, which we call the \emph{frozen group}.
\end{definition}

\begin{definition}
[\emph{Seed mutations}]
\label{def:seed-mutation0}
Let $(\txx, Q)$ be a seed in $\Fcal$,
and let $k$ be a mutable vertex of $Q$.
The \emph{seed mutation} $\mu_k$ in direction~$k$ transforms
$(\txx, Q)$ into the seed
$\mu_k(\txx,  Q)=(\txx', \mu_k(Q))$, where the cluster
$\txx'=(x'_1,\dots,x'_s)$ is defined as follows:
$x_j'=x_j$ for $j\neq k$,
and $x'_k \in \Fcal$ is determined
by the \emph{exchange relation}
\begin{equation}
\label{exchange relation0}
x'_k\ x_k =
 \ \prod_{\substack{i \to k}} x_i
+ \ \prod_{\substack{k \to i}} x_i \, .
\end{equation}
\end{definition}

Note that arrows between two frozen vertices of a quiver do not
affect seed mutation; therefore we often omit
arrows between two frozen vertices. 

\begin{definition}
[\emph{Cluster algebra}]
\label{def:cluster-algebra0}
Given an initial labeled seed $(\txx, Q)$, we let $\Xcal$ denote the union of all cluster variables obtained by performing all possible mutation sequences from the initial seed.
Let $\CC[c^{\pm 1}]$ be the \emph{ground ring} consisting
of Laurent polynomials in the frozen variables. The
\emph{cluster algebra} $\Acal= \Acal(\txx,  Q)$ associated with a
given pattern is the $\CC[c^{\pm 1}]$-subalgebra of the ambient field $\Fcal$
generated by all cluster variables, 
with coefficients which are Laurent polynomials
in the frozen variables: $\Acal = \CC[c^{\pm 1}] [\Xcal]$.
We say that $\Acal$ has \emph{rank $r$} because each cluster contains
$r$ cluster variables. Cluster (or frozen) variables that belong to a common cluster are said to be \emph{compatible}.
\end{definition}

\begin{remark}\label{rmk:dif-ground-ring}
Another common convention is to choose the ring 
	$\CC[c]$ of polynomials in the frozen variables as the ground ring,
and define the cluster algebra 
as $\overline{\Acal} := \CC[c] [\Xcal]$.
\end{remark}

The Grassmannian $\Gr_{m,n}(\C)$ has a natural cluster algebra structure, generated by the Pl\"ucker coordinates. A~standard seed for this cluster structure is the \emph{rectangles seed} $\Sigma_{m,n}$, with $m(n-m)+1$ variables, $n$~of which are frozen, as demonstrated in \cref{rectangles} below. The cluster algebra $\overline{\A}(\Sigma_{m,n})$ is the homogeneous coordinate ring $\C[\Gr_{m,n}]$ of the complex Grassmannian \cite{scott, FW6}. We refer to any seed for this cluster algebra as a \emph{seed for} $\Gr_{m,n}$.

\subsection{Quasi-cluster homomorphisms of cluster algebras}

In this section we define quasi- cluster homomorphisms of cluster
algebras, following \cite{Fraser} and \cite{Fraser2}.

\begin{definition}[Exchange ratios] 
\label{def:exchage-ratios}
Given a seed $\Sigma=((x_1,\dots,x_s), Q)$ for a cluster algebra of rank $r\leq s$, and a mutable variable $x_i$ (so that $1 \leq i \leq r$), the \emph{exchange ratio} of $x_i$ (with respect to $\Sigma$) is
$$ \hat{y}_{\Sigma}(x_i) \;=\; 
\frac{\prod_{j: i\to j} x_j^{\# \arr(i\to j)}}{\prod_{j: j\to i} x_j^{\# \arr(j\to i)}} $$
where $\arr(i\to j)$ denotes the number of arrows from $i$ to $j$ in the quiver $Q$.
\end{definition}

Given a cluster algebra $\A$, we let $\PP$ denote its frozen group, that is the group of Laurent monomials
in the frozen variables.  For elements $x,y\in \A$, we say that
\emph{$x$ is proportional to $y$}, writing $x \propto y$, if $x=My$ for some Laurent monomial
$M\in \PP$. We then refer to $M$ as a \emph{frozen factor}.

\begin{definition}[Quasi-cluster homomorphism]\label{def:quasi}\cite[Definition 3.1 and Proposition 3.2]{Fraser}
Let $\A$ and $\overline{\A}$ be two cluster algebras of the same rank $r$,
and with respective frozen groups $\PP$ and $\overline{\PP}$.
        Then an algebra homomorphism $f:\A \to \overline{\A}$ that satisfies
        $f(\PP)\subseteq \overline{\PP}$ is called a \emph{quasi-cluster homomorphism}  
        from $\A$ to $\overline{\A}$
        if there are seeds
$\Sigma=((x_1,\dots,x_s),Q)$ and $\overline{\Sigma}=((\bar{x}_{\bar{1}},\dots,\bar{x}_{\bar{s}}), \bar{Q})$
        for $\A$ and $\overline{\A}$, such that
        \begin{enumerate}
                \item $f(x_i) \propto \bar{x}_{\bar{i}}$ for $1\leq i \leq r$
                \item $f(\hat{y}_{\Sigma}(x_i)) = \hat{y}_{\overline{\Sigma}}(\bar{x}_{\bar{i}})$ for $1\leq i \leq r$.
        \end{enumerate}
        If conditions (1) and (2) hold,
        we write $f(\Sigma) \propto \overline{\Sigma}$.
        \end{definition}
Note that conditions (1) and (2) imply that the map $i \mapsto \bar{i}$ of mutable nodes in $Q$ and $\bar{Q}$ extends to an isomorphism of the corresponding induced subquivers. If we want to show $f(\Sigma) \propto \overline{\Sigma}$, we often start by finding an isomorphism between the mutable nodes of $Q$ and $\bar{Q}$ and then check that (1) and (2) are satisfied.

\begin{proposition}
\cite[Proposition 3.2]{Fraser} \label{prop:similar}
If
$\Sigma \propto \overline{\Sigma}$, then
$\mu_k(\Sigma)$ is similar to $\mu_k(\overline{\Sigma})$. 
\end{proposition}

\subsection{Background on the Grassmann-Cayley algebra}\label{subsec:GC-alg} 
The \emph{Grassmann-Cayley algebra}, the exterior algebra endowed with an additional \emph{shuffle product}, 
is the natural setting for many of our computations. We recall the definitions and basic facts here. Let 
$$ S_n^r:= \{w \in S_n: w(1)< \cdots < w(r), w(r+1)< \cdots < w(n)\}.$$ 
For vectors $v_1, \dots, v_m \in \CC^m$, we use the notation $\lr{v_1 \cdots v_m}:= \det[v_1 \dots v_m]$.

\begin{definition}
Let $A = a_1 \wedge \dots \wedge a_p \in \bigwedge^p \CC^m$ and $B= b_1 \wedge \dots \wedge b_q \in \bigwedge^q \CC^m$ and set $s:=p+q-m$. The \emph{shuffle} $A \shuf B$ of $A$ and $B$ is $0$ if $s<0$. If $s \ge 0$, then we define 
\begin{align*}A \shuf B :=&
\sum_{w \in S_p^{m-q}} \sign(w)~ \lr{a_{w(1)} \cdots a_{w(m-q)} b_1 \cdots b_q}~ a_{w(m-q+1)} \wedge \dots \wedge a_{w(p)}\\
=&\sum_{w \in S_q^{s}}\sign(w)~ \lr{a_1 \cdots a_p b_{w(s+1)} \cdots b_{w(q)}}~ b_{w(1)} \wedge \dots \wedge b_{w(s)}.
\end{align*}

In the degenerate case where $s\ge 0$ and $p=0$, so $A$ is the scalar $a$, we have $q=m$ and define $A \shuf B := a \lr{b_1 \cdots b_q}$; analogously, if $s \ge 0$ and $q=0$ so $B$ is the scalar $b$, $A \shuf B := b \lr{a_1 \cdots a_p}$.
Note that if $s=0$, then 
$A\shuf B$ is the determinant $\lr{a_1 \cdots a_p b_1 \cdots b_q}$.
\end{definition}

We extend the operation $\shuf$ to arbitrary elements of the exterior algebra $\bigwedge \CC^m$  by linearity. In compound expressions, we adopt the order of operations precedence rule that the $\wedge$ product is performed before the $\shuf$ product.

\begin{lemma}
The operation $\shuf$ has the following properties.
\begin{itemize}
\item If $A \in \bigwedge^p \CC^m$ and $B \in \bigwedge^q \CC^m$, then $A \shuf B = (-1)^{(m-p)(m-q)} B \shuf A$.
\item The operation $\shuf$ is associative: $(A \shuf B) \shuf C = A \shuf (B \shuf C)$.
\end{itemize}
\end{lemma}

\begin{definition}
The \emph{Grassmann-Cayley algebra} $\GC(m)$ is the exterior algebra $\bigwedge \CC^m$ endowed with the additional \emph{shuffle product} $\shuf: \bigwedge \CC^m \times \bigwedge \CC^m \to \bigwedge \CC^m$, which is associative.
\end{definition}

We summarize some additional basic properties of $\wedge$ and $\shuf$, and their relation to intersections and sums of subspaces. For $A = a_1 \wedge \dots \wedge a_p \in \bigwedge \CC^m$, write $\overline{A}$ for $\spn({a_1, \dots, a_p})$. We say $A$ \emph{represents}~$\overline{A}$. 
 
\begin{lemma}\label{lem:basic-prop-of-shuf}
    Let $A= a_1 \wedge \dots \wedge a_p$ and $B=b_1 \wedge \dots \wedge b_q$ be decomposable elements of $\GC(m)$.
    \begin{itemize}
        \item We have $A \wedge B \neq 0$ if and only if $\overline{A} \cap \overline{B} = \{0\}$, and then $A \wedge B$ represents $\overline{A} + \overline{B} = \overline{A} \oplus \overline{B}.$
\item If $\overline{A} + \overline{B} \neq \CC^m$, then $A \shuf B =0$.
\item If $\overline{A} + \overline{B} = \CC^m$ and $\overline{A} \cap \overline{B} = \{0\}$, then $A \shuf B =\lr{a_1, \dots, a_p, b_1, \dots, b_q}$ is a scalar.
\item If $\overline{A} + \overline{B} = \CC^m$ and $\overline{A} \cap \overline{B} \neq \{0\}$, then $A \shuf B$ is again decomposable and represents $\overline{A} \cap \overline{B}$.
    \end{itemize}
\end{lemma}

Using these properties, it is straightforward to see that, under some constraints on dimension, wedge and shuffle products represent sums and intersections, respectively, of subspaces.

\begin{lemma}\label{lem:shuf-wedge-vs-cap-sum}
Suppose $A_1, \dots, A_p \in \bigwedge \CC^m$ are decomposable. 
\begin{itemize}
\item
${A_1 \wedge \dots \wedge A_p}$ represents $\overline{A_1} + \dots + \overline{A_p}$ if
\[\dim (\overline{A_1} + \dots + \overline{A_p}) = \sum_i \dim \overline{A_i}\]
and otherwise ${A_1 \wedge \dots \wedge A_p}=0$.
\item
${A_1 \shuf \dots \shuf A_p}$ represents $\bigcap_i \overline{A_i}$ if
\[\dim\left(\bigcap_{i} \overline{A_i}\right)= \left(\sum_i \dim \overline{A_i}\right) - (p-1)m \ge 1\]
and otherwise ${A_1 \shuf \dots \shuf A_p}$ is a scalar.
\end{itemize}
\end{lemma}

The following relation, which is easy to verify, will also be useful to us.

\begin{lemma}\label{lem:GC-easy-reln} 
Let $A_1 \in \bigwedge^{d_1} \CC^m, \dots, A_p \in \bigwedge^{d_p} \CC^m$ be decomposable, and suppose $d_1 + \dots + d_p = m+1$. Then 
\[\sum_{i=1}^p (-1)^{d_i(d_i+d_{i+1}+\dots+d_p)} ~ (A_1 \wedge \dots \wedge \widehat{A_i} \wedge \dots \wedge A_p) \shuf A_i = 0 \]
\end{lemma}

We will often use the Grassmann-Cayley algebra $\GC(m)$ to define regular functions on the Grassmannian $\Gr_{m,n}$, or functions from one Grassmannian to another. If $V \in \Gr_{m,n}$ is represented by the matrix $[v_1 \cdots v_n]$, we often write $v_i$ as just $i$. We also often omit the symbol $\wedge$, writing $i_1 \dots i_j$ for $v_{i_1}\wedge \cdots \wedge v_{i_j}$. 

\begin{notation}\label{not:brackets} If $F \in \GC(m)$ is a scalar, we sometimes write $\lr{F}$ for $F$, to emphasize that it is a scalar. As is consistent with our previous convention with Pl\"ucker coordinates, if $F= v_{1} \wedge \cdots \wedge v_m \in \bigwedge^m \CC^m$, we write $\lr{F}:=\det[v_1 \dots v_n]$. 
\end{notation}

\begin{definition}\label{def:chain}  Consider $\GC(m)$ and $V \in \Gr_{m,n}$ represented by $[v_1 \cdots v_n]$, and let $A = \{a_1,\dots,a_i\}$, $B = \{b_1,\dots,b_j\}$ and $C = \{c_1,\dots,c_k\}$ be ordered sets of indices in $[n]$ such that $i+j \ge m$, $i+k \geq m$, $j+k \geq m$, and $i+j+k=2m$.  Then
\[A \shuf B \shuf C \;=\; ({a_1}  \cdots  {a_i}) \shuf ({b_1}  \cdots  {b_j}) \shuf ({c_1}  \cdots  {c_k}) \;=\;(v_{a_1} \wedge \cdots \wedge v_{a_i}) \shuf (v_{b_1} \wedge \cdots \wedge v_{b_j}) \shuf (v_{c_1} \wedge \cdots \wedge v_{c_k})  \]
is a scalar, and we use the notation $\lr{A \shuf B \shuf C}:=A \shuf B \shuf C$.
We call $\lr{A \shuf B \shuf C}$ a \emph{chain polynomial}. 
\end{definition}

\begin{remark}
For $m=4$, \cite[Definition 2.5]{even2023cluster} defined a ``chain polynomial" using the notation $\lr{abc \br de \br fgh}$. We have $\lr{abc \shuf de \shuf fgh} = - \lr{abc \br de \br fgh}.$ 
\end{remark}

\begin{example}\label{ex:chain}
On the Grassmannian $\Gr_{m,n}$ the chain polynomials of \cref{def:chain} are quadratic polynomials in the Pl\"ucker coordinates.
For example, let $m=3$, and $i=j=k=2$.  
Let $A=\{a_1,a_2\}$, $B=\{b_1,b_2\}$, and $C=\{c_1,c_2\}$
be ordered sets of indices in $[n]$.
Then 
\begin{align*}
\lr{A \shuf B \shuf C} &= (A\shuf B) \shuf C = A \shuf (B \shuf C)\\ 
&= \lr{{a_1} {b_1} {b_2}} \lr{{a_2} {c_1} {c_2}} - \lr{{a_2} {b_1} {b_2}} \lr{{a_1} {c_1} {c_2}} \\
&=\lr{{a_1} {a_2} {b_2}}\lr{{b_1} {c_1} {c_2}} - 
\lr{{a_1} {a_2} {b_1}}\lr{{b_2} {c_1} {c_2}}.
\end{align*}
\end{example}

\begin{example}
\label{ex:intersections}
Let $A=\{a_1,a_2,a_3\}, B=\{b_1,b_2,b_3\}, C=\{c_1,c_2,c_3\}$ be ordered sets of indices in~$[n]$.
For $V \in \Gr_{4,n}$ represented by $[v_1 \cdots v_n]$, consider the element 
\[a_3  (B \shuf C)\;:=\;v_{a_3} \wedge ((v_{b_1} \wedge v_{b_2} \wedge v_{b_3}) \shuf (v_{c_1} \wedge v_{c_2} \wedge v_{c_3}) ) \;\in\; \GC(m).\]
\begin{itemize}
\item
By definition, we have
\begin{align*}a_3 (B \shuf C) &= a_3(\lr{b_1b_2b_3c_3} c_1c_2 - \lr{b_1b_2b_3c_2} c_1c_3 + \lr{b_1b_2b_3c_1} c_2c_3)\\
&= \lr{b_1b_2b_3c_3} a_3c_1c_2 - \lr{b_1b_2b_3c_2} a_3c_1c_3 + \lr{b_1b_2b_3c_1} a_3c_2c_3\\
&= \lr{b_1 c_1 c_2 c_3} a_3 b_2 b_3 -\lr{b_2 c_1 c_2 c_3} a_3 b_1 b_3 +\lr{b_3 c_1 c_2 c_3} a_3 b_1 b_2.\end{align*}
Provided that all vectors involved are sufficiently generic, $a_3 (B \shuf C)$ represents the 3-plane spanned by $v_{a_3}$ and the 2-plane
$\spn(v_{b_1}, v_{b_2}, v_{b_3}) \cap \spn(v_{c_1}, v_{c_2}, v_{c_3})$.
\item
We also have 
\begin{align*}
(a_1 a_2) \shuf (a_3 
(B \shuf C)) \;&=\;   \lr{a_1 a_3\shuf B \shuf C} a_2 -  \lr{a_2 a_3\shuf B \shuf C }a_1 \\
\;&=\; \lr{a_1 a_2\shuf B \shuf C}a_3 - \lr{A \shuf B\shuf c_2 c_3}c_1  +   \lr{ A \shuf B \shuf c_1 c_3}c_2 -   \lr{A\shuf B \shuf c_1 c_2}c_3
\end{align*}
Generically, this element represents the intersection of the 2-plane $\spn(v_{a_1}, v_{a_2})$ with the 3-plane represented by~$a_3 (B \shuf C)$.

\item
Similarly,
\begin{align*}
A\shuf B \shuf C \;&=\; \lr{a_1 a_2 \shuf B \shuf C}a_3-  \lr{a_1 a_3 \shuf B \shuf C}a_2 +   \lr{a_2 a_3 \shuf B \shuf C}a_1 \\
\;&=\; \lr{A \shuf B \shuf c_1c_2}c_3- \lr{A \shuf B \shuf c_1c_3}c_2 + \lr{A \shuf B \shuf c_2c_3}c_1
\end{align*}
represents the line obtained by intersecting three 3-planes: $\spn(v_{a_1}, v_{a_2}, v_{a_3})$, $\spn(v_{b_1}, v_{b_2}, v_{b_3})$, and $\spn(v_{c_1}, v_{c_2}, v_{c_3})$, provided the vectors involved are sufficiently generic.
\end{itemize}
\end{example}

\subsection{The amplituhedron map}

Motivated by the desire to give a geometric understanding of the \emph{BCFW recurrence} for computing scattering amplitudes in $\mathcal{N}=4$ super Yang Mills theory, and building on the Grassmannian formulation of \cite{abcgpt},
Arkani-Hamed and Trnka
\cite{arkani-hamed_trnka}
introduced
the \emph{(tree) amplituhedron}, which they defined as 
the image of the positive Grassmannian under a positive linear map.  While the amplituhedron is not a central focus of this paper, the desire to understand it was a hidden motivation of this work. The amplituhedron map, defined below, and its degree on positroid varieties, is very closely related to our constructions.

Let $\Mat_{n,p}^{\circ}$ denote the set of full-rank $n\times p$ matrices, and $\Mat_{n,p}^{>0} \subset \Mat_{n,p}^{\circ}$ be the subset of matrices whose maximal minors are positive.

\begin{definition}[The amplituhedron map]\label{defn_amplituhedron}
Let $Z\in \Mat_{n,k+m}^{\circ}$, where $k+m \leq n$. 
The \emph{amplituhedron map}
$\tilde{Z}:\Gr_{k,n} \dashedrightarrow \Gr_{k,k+m}$
        is the rational map defined by
        $\tilde{Z}(C):=CZ$,
    where
 $C$ is a $k \times n$ matrix representing an element of
        $\Gr_{k,n}$,  and  $CZ$ is a $k \times (k+m)$ matrix representing an
         element of $\Gr_{k,k+m}$.
         The \emph{exceptional locus} of $\tZ$, denoted by $E_Z$, is the subset of $\Gr_{k,n}$ where $\tZ$ is not defined.
\end{definition}
         
If $Z \in \Mat_{n,k+m}^{>0}$, $\tilde{Z}$ is well-defined when
restricted to the positive Grassmannian 
$\Gr_{k,n}^{\geq 0}$, and the \emph{tree amplituhedron} $\mathcal{A}_{n,k,m}(Z) \subset \Gr_{k,k+m}$ is the image
$\tilde{Z}(\Gr_{k,n}^{\ge 0})$.

It will be convenient for us to compose the amplituhedron map with an embedding into $\Gr_{m,n}$, which also depends on $Z$.

\begin{definition}\label{def:twistormap}
Let $Z \in \Mat_{n, k+m}^\circ$. The \emph{twistor map} is the map 
\begin{align*}
\upsilon_Z:\Gr_{k, k+m} &\hookrightarrow \Gr_{m,n}\\
Y &\mapsto Y^\perp Z^T
\end{align*}
where $Y^\perp \subset \CC^{k+m}$ denotes the subspace orthogonal to $Y$.
\end{definition}
 The twistor map $\upsilon_Z$ is injective, and if we set $W= \colsp Z$, the image of $\upsilon_Z$ is $\Gr_m(W):= \{X \in \Gr_{m,n}: X \subset W\}$, which is closed in $\Gr_{m,n}$. In terms of functions, the pullback $\upsilon_Z^*\lr{I}= \lr{I}_{\upsilon_Z (Y)}$ of a Pl\"ucker coordinate is the \emph{twistor coordinate}
 \[ \det\begin{bmatrix}-Y-\\
 -Z_{I}-
 \end{bmatrix}\]
  where $Z_I$ is the submatrix of $Z$ using rows $I$.

Let $\comp_Z:= \upsilon_Z \circ \tZ$.  We use the notation $\mcy_G:=\overline{\comp_Z(\Pi_G \setminus E_Z)}$ for the (Zariski) closure of the image of $\Pi_G$ under $\comp_Z$. We define  
\begin{equation}\label{eq:Upsilon}
\comp_{Z,G}:= \comp_Z |_{\Pi_G} : \Pi_G \dashrightarrow \mcy_G
\end{equation}
to be the restriction of $\comp_{Z}$ to $\Pi_G$.

\begin{definition} \label{def:intnumber} 
Let $\Pi_G \subset \Gr_{k,n}$.
The \emph{$m$-intersection number} of $\Pi_G$, denoted $\mint(G)$, is defined as 
\[\mint(G):= \begin{cases} 0 & \text{ if for generic }Z \in \Mat_{n, k+m}^\circ,~ \dim \mcy_G < \dim \Pi_G\\
d & \text{ if for generic }Z \in \Mat_{n, k+m}^\circ, ~\comp_{Z,G} \text{ is generically $d$-to-1.}
\end{cases}\]

If $\Pi_G \subset \Gr_{k,n}$ has dimension $km$ and has $m$-intersection number 1, then $\mcy_G=\Gr_m(W)$ and we call $\Pi_G$ an \emph{algebraic pre-tile}, or just a \emph{pre-tile}.
\end{definition}

\begin{remark}\label{rem:tile-vs-pretile}
The amplituhedron literature often considers a semi-algebraic version of the objects above. For a positroid cell $S_G = \Pi_G \cap \Grk$, one can consider $Z_G:= \overline{\tZ(S_G)}$, the closure of the image of the positroid cell. We call $Z_G$ a \emph{tile} if $\tZ$ is injective on $S_G$ and $\dim \Pi_G = km$. Note that $\Pi_G$ being a pre-tile \emph{does not} imply that $Z_G$ is a tile, and $Z_G$ being a tile \emph{does not} imply that $\Pi_G$ is a pre-tile. We record the following observations regarding tiles and pre-tiles:
\begin{itemize}
\item For all known tiles $Z_G$, $\Pi_G$ is a pre-tile.
\item For $m=1$: there are pre-tiles $\Pi_G$ such that $Z_G$ is not a tile. This behavior is expected to persist for odd $m$. 
\item For $m=2$: $\Pi_G$ is a pre-tile if and only if $Z_G$ is a tile.
\end{itemize}
 
\end{remark}
The following result of Lam gives a $Z$-independent characterization of $m$-intersection number and shows that it is always finite.

\begin{proposition}[{\cite[Proposition 4.8]{LamAmpCells}}]\label{prop:cohom-characterization-int-num} Suppose $\Pi_G \subset \Gr_{k,n}$ has dimension $km$. Then $\mint(G)$ is the coefficient of the Schur polynomial $s_{(n-k-m)^k}$ in the cohomology class $[\Pi_G] \in H^*(\Gr_{k,n})$.

Equivalently, fix $X \in \Gr_{n-m, n}$ and let $\Gr_k(X):= \{V \in \Gr_{k,n}: V \subset X\}.$ Then in $H^*(\Gr_{k,n})$, 
\[[\Gr_k(X)] \cap [\Pi_G] = \mint(G) \cdot [\pt].\]
\end{proposition}

\begin{remark}\label{rk:physics_1}
Scattering amplitudes in planar $\mathcal{N}=4$ super Yang-Mills (SYM) have fundamental building blocks called \emph{Yangian invariants}, which are functions on $\Gr_{4,n}$ of external momenta of scattering particles and encode the \emph{leading singularities} of the theory. Each Yangian invariant $\mathcal{Y}_G$ corresponds to a $4k$-dimensional positroid cell $S_{G} \subseteq \Gr^{\geq 0}_{k,n}$ \cite{Arkani-Hamed:2012zlh}.
If $\mbox{IN}_4(G)=1$, $\mathcal{Y}_G$ is a rational function whose poles correspond to facets of $\tilde{Z}(S_{G})$. If $\mbox{IN}_4(G)>1$, $\mathcal{Y}_G$ can be expressed as a sum of $d$ functions which have \emph{algebraic} singularities.
\end{remark}

\section{Vector-relation configurations}
\label{sec:promotion-plabic-graphs}
 In this section we introduce \emph{vector relation configurations}, essentially inspired by \cite{AGPR}.
We note that related frameworks had previously appeared in \cite{Arkani-Hamed:2012zlh}, \cite{postnikovicm} and \cite{lam2015totally}.

Throughout this section we will work with bipartite
plabic graphs as in \cref{not:bipartite}.  For background, see \cref{app:plabic}, more specifically \cref{def:plabic} and \cref{rem:bipartite}. 
\begin{notation}\label{not:bipartite}
Let $G=((B,W), E)$
be a bipartite plabic graph, where 
$B$ and $W$ denote the sets of black and white
vertices and $E$ denotes the set of edges.  We require
that all boundary vertices are black:
we write  
$B=\Bi \sqcup \Bb$, where $\Bi$ is the set of
black interior vertices, and $\Bb=\{1,2,\dots,n\}$
is the set of boundary vertices, labeled in increasing order clockwise. We use the notation $G^{op}$ for the graph obtained from $G$ by changing the color of every vertex (including the boundary vertices). 
We will assume that all our plabic graphs are leafless, as in \cref{def:plabic}.  Unless otherwise mentioned, $G$ is assumed to be reduced. 
\end{notation}

\begin{definition}\label{def:VRC}
Let $G$ be a bipartite plabic graph as in \cref{not:bipartite}.
    Let $m \geq 1$ be a positive integer. An \emph{$m$-vector-relation configuration} ($m$-VRC) on $G$ is an assignment of vectors $v_b \in \CC^m$ 
    to each black vertex of $G$ and edge weights $r_e \in \CC^*$ to each edge of $G$, such that 
    \begin{enumerate}
    \item the \emph{boundary vectors} $v_1, \dots, v_n$ on vertices $1, \dots n$ span $\CC^m$;
    \item
    for each white vertex $w$, the following relation holds
    \[\sum_{\{w,b\}=e \in E} r_e v_b =0.\]
   \end{enumerate} 
    We write $(\bv, \bR):=(\{v_b\}_{b \in B}, \{r_e\}_{e \in E})$ for an $m$-VRC. The \emph{boundary} of $(\bv, \bR)$ is the $m \times n$ matrix of boundary vectors
    \[\partial(\bv, \bR):= \begin{bmatrix}
        v_1 \cdots v_n
    \end{bmatrix}.\]

    An $m$-VRC is \emph{non-degenerate} if all vectors on internal vertices are nonzero, and otherwise is \emph{degenerate}.
\end{definition}

\begin{remark}
We note that our conventions are somewhat different from those of \cite{AGPR}.  We swap the roles of black and white relative to \cite{AGPR}; and in contrast to \cite{AGPR}, we allow our vectors $v_b$ to be zero, and we require our edge weights to be nonzero. 
\end{remark}

We note that $GL_m(\CC)$ acts on $(\bv, \bR)$ by left multiplication on all vectors. We often consider $(\bv, \bR)$ up to this action, in which case the boundary is a well-defined element of $\Gr_{m,n}$. We also have the following action of $(\CC^*)^{|W|+|\Bi|}$ on $m$-VRCs, via gauge transformations.

\begin{definition}
Let $(\bv, \bR)$ be an $m$-VRC on a plabic graph $G$. For an internal vertex $x$ of $G$ and $t \in \CC^*$, \emph{gauge transformation by $t$ at $x$} changes $r_e \mapsto t r_e$ for all edges $e$ adjacent to $x$ and, if $x=b$ is black, additionally changes $v_b \mapsto (1/t) v_b$. We write 
    $\mVRC_G$ for $m$-VRCs on $G$ modulo gauge transformations and $GL_m$-action. If $(\bv, \bR)$ is an $m$-VRC on $G$, we write $[\bv, \bR]$ for the corresponding element of $\mVRC_G$ and $\partial[\bv, \bR] \in \Gr_{m,n}$ for its boundary, which is well-defined. We write $\mzVRC_G$ for the subset of $\mVRC_G$ with boundary $\bz \in \Gr_{m,n}$.
\end{definition}
Note that gauge transformations commute with $GL_m$-action. It will sometimes also be useful to allow gauge transformations at the boundary vertices of $G$. With this enlarged gauge group, $\partial[\bv, \bR]$ is an element of $\Conf_{m,n}$.

\begin{remark}\label{rem:AGPRcontrast} 
One of our main results about vector relation configurations is
\cref{cor:int1}, which gives a characterization of when a collection of generic boundary vectors can be uniquely extended to a vector relation configuration.  This should be compared to \cite[Theorem 1.1]{AGPR}.
Let $G$ be a reduced plabic graph of type $(k,n)$.  Then we have the following results:
\begin{itemize}
\item (\cref{cor:int1}): A configuration of $n$ generic boundary vectors in $\C^m$ can be extended to a vector relation configuration on $G$ that is unique up to gauge, if and only if the positroid variety  $\Pi_G$ has \emph{$m$-intersection number} $1$ and $\Pi_G$ has dimension $km$. 

\item (\cite[Theorem 1.1]{AGPR}) A configuration of $n$ boundary vectors in $\C^k$ that represent a point of $\Pi_G$ can generically be extended to a vector relation configuration on $G$ that is unique up to gauge. 
\end{itemize}  \end{remark}

The next lemma gives a formula for the vectors in a VRC in terms of certain boundary vectors. 
Recall from \cref{def:orientation} that a \emph{reverse perfect orientation} of $G$ is an orientation of the edges such that every internal white vertex has a unique outgoing edge, and every internal black vertex has a unique incoming edge. Every reduced plabic graph has an acyclic reverse perfect orientation (cf. \cref{acycliclemma}).

\begin{lemma}\label{lem:v_b-from-paths}
    Suppose $G$ is a (leafless)  plabic graph with an acyclic reverse perfect orientation~$\OO$.   
    Let $I$ be the source set of $\OO$. Let $(\bv, \bR)=(\{v_b\}, \{r_e\})$ be an $m$-VRC on $G$.

    Then for any black vertex $b$, 
    \begin{equation}\label{eq:v_b-from-paths}
    v_b = \sum_{i \in I} \left( \sum_{\substack{ P: i \to b \\ \text{ path in } \OO}} \frac{\prod_e r_e}{\prod_{e'} (-r_{e'})}\right) v_i  \end{equation}
    where the product in the numerator is over the edges $e \in P$ which are oriented black-to-white in $\OO$ and the product in the denominator is over edges $e'\in P$ oriented white-to-black.

\end{lemma}
\begin{proof}
We induct on $q_b= \max_{i \in I} (\text{length of longest path }i \to b \text{ in }\OO)$. Since $\OO$ is acyclic and $G$ is leafless, every internal vertex can be reached from a source: every internal vertex has degree at least two, so has some incoming edge. If we walk backwards along the incoming edges, we must eventually reach a source, as $\OO$ has no oriented cycles. So $q_b$ is well-defined and nonnegative for all black vertices $b$.

If $q_b=0$, then $b=i$ for some $i \in I$.
In that case, there are no paths $j \to b$ unless $j=b$, so all terms on the right of \eqref{eq:v_b-from-paths} are zero except the one indexed by $b$. There is a single path $b \to b$ which consists of no edges. So the right-hand side of \eqref{eq:v_b-from-paths} is $v_i$.

Now, assume $q_b>0$ and that \eqref{eq:v_b-from-paths} holds for $b'$ with $q_{b'}<q_b$. This means $b$ is either internal or a sink of $\OO$. There is a unique edge which is directed towards $b$, say it is $f=(w,b)$. Say $b_1, \dots, b_p$ are the other neighbors of $w$; the edges $e_j=(b_j, w)$ are all directed towards $w$. We have $v_b = -r_{f}^{-1}\sum_{j} r_{e_j} v_{b_j}$. Any path from $i \in I$ to $b$ must go through some $b_j$ and any path $i \to b_j$ can be extended to a path $i \to b$, so we have $q_{b_j} < q_{b}$. Using the inductive hypothesis, we have 
\[ v_b =  \sum_{j} \frac{r_{e_j}}{-r_f} \sum_{i \in I} v_i \sum_{\substack{ P: i \to b_j \\ \text{ path in } \OO}} \frac{\prod_e r_e}{\prod_{e'} (-r_{e'})}= -\sum_{i \in I} v_i  \sum_{j} \sum_{\substack{ P: i \to b_j \\ \text{ path in } \OO}} \frac{r_{e_j}}{-r_f} \frac{\prod_e r_e}{\prod_{e'} (-r_{e'})}= \sum_{i \in I} v_i  \sum_{\substack{ P: i \to b \\ \text{ path in } \OO}}  \frac{\prod_e r_e}{\prod_{e'} (-r_{e'})} \]
where the two products are over edges $e$ and $e'$ in $P$ oriented black-to-white and white-to-black respectively. This shows \eqref{eq:v_b-from-paths} holds for $v_b$.
\end{proof}

\begin{corollary}\label{cor:VRC-det-by-bdry-and-R}
Suppose $G$ is a (leafless) plabic graph which admits an acyclic reverse perfect orientation. Then for any $m$-VRC $(\bv, \bR)$ on $G$, $\bv$ is determined by the boundary $\partial(\bv, \bR)$ and the coefficients $\bR$.
\end{corollary}
 
Moves on plabic graphs induce maps on $\mVRC_G$. See \cite[Section 2]{AGPR} for proofs. 

\begin{lemma}\label{lem:moves}
    Suppose $G'$ is obtained from $G$ by a black (resp. white) expand move. Then the map shown in \cref{fig:expand-contract}, top (resp. bottom), induces a bijection $\mVRC_G \to \mVRC_{G'}$. 

    Suppose $G, G'$ are related by a square move as in \cref{fig:sq-move}. The map shown in \cref{fig:sq-move} induces a bijection between $\{[\bv, \bR] \in \mVRC_G: ce-af \neq 0\}$ and $\{[\bv, \bR] \in \mVRC_{G'}: c'e'-a'f' \neq 0\}$.
\end{lemma}

\begin{figure}[h]
    \centering
    \includegraphics[width=0.7\linewidth]{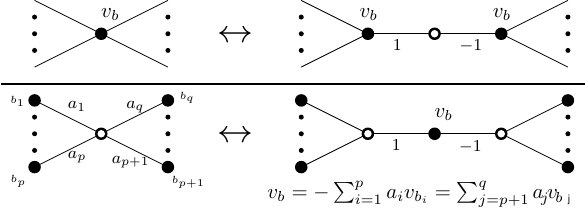}
    \caption{Top: reading left to right, a black-expand move and its effect on $m$-VRCs. Reading right to left, a black-contract move and its effect on $m$-VRCs. Bottom: white-expand and white-contract moves and their effect on $m$-VRCs.}
    \label{fig:expand-contract}
\end{figure}

\begin{figure}[h]
    \centering
    \includegraphics[width=0.7\linewidth]{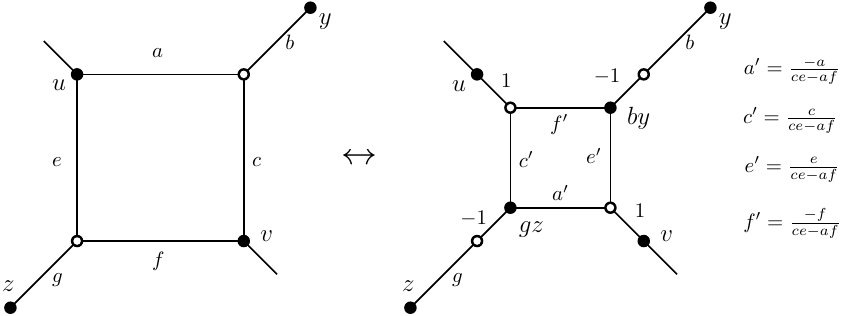}
    \caption{A square move and its effect on $m$-VRCs when $ce-af \neq 0$. Note that if $ce-af =0$, there is no relation among $u, y, z$ (resp. $v, y, z$) unless $u=0$ (resp. $v=0$).}
    \label{fig:sq-move}
\end{figure}

\section{Plabic tangles and promotion maps} \label{sec:promotion}

In this section we use $m$-VRCs on plabic graphs to give rational maps between Grassmannians. To record the image of the map, we need to first enhance plabic graphs slightly to \emph{plabic tangles}.  We note that 
the definition of plabic tangle is inspired by the notion of \emph{planar tangle} \cite{Jones}, and we will later show that similarly to the case of planar tangles, there is an \emph{operad structure} on plabic tangles, see \cref{sec:categorical_pov}.  

\subsection{Tangles and maps on configuration spaces}

\begin{definition}\label{def:tangle}
Let $\ell$ be a positive integer.
A \emph{plabic tangle} $(G, \bD)$ is the data of a plabic graph 
$G=((B,W),E)$, drawn inside a disk called an \emph{outer disk} (with boundary vertices $\Bb=\{1,\dots,n\}$), together with $\ell$ \emph{inner disks},
each of which lies in a 
face of $G$. 
Each inner disk has boundary vertices $D^{(i)}$ (for $i \in [\ell]$), and each vertex $u \in D^{(i)}$ is connected to one element $b_u$ of $B$ via a segment such that the resulting graph is planar.
Each inner and outer disk contains one \emph{$\star$-marked interval}, and the boundary vertices of each disk are labeled clockwise in increasing order starting just after the $\star$.
We use the notation $\bD:=\{D^{(i)}\}_{i \in [\ell]}$ for the boundary vertices of the inner disks. We often refer to $G$ as the \emph{core} 
and the 
inner disks as \emph{blobs}, see \cref{fig:tangle}.
\end{definition}

\begin{figure}[h]
    \centering
    \includegraphics[width=0.7\linewidth]{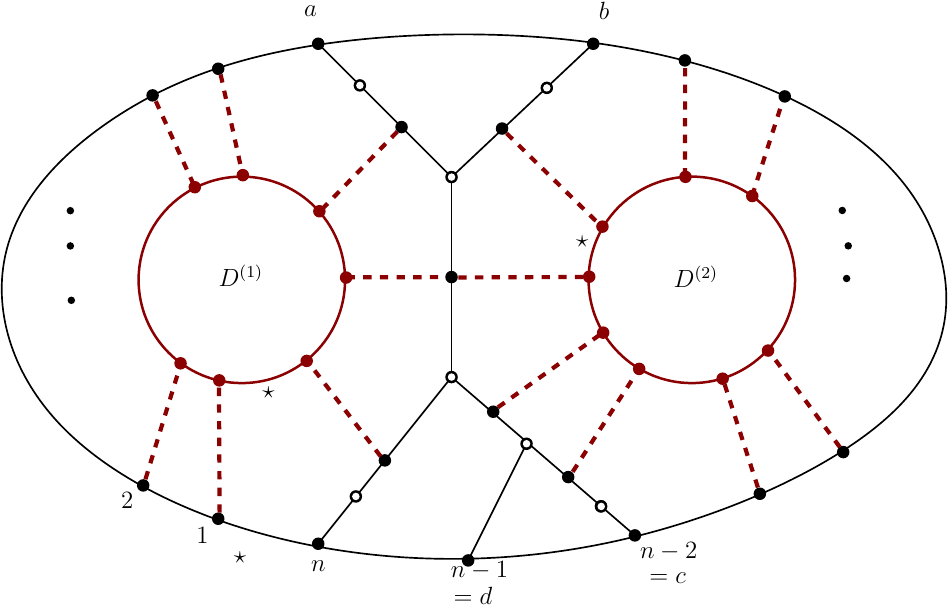}
    \caption{A plabic tangle with two blobs; the core is shown in black, while the blobs (inner disks) with their attaching segments are shown in red.}
    \label{fig:tangle}
\end{figure}

We first use VRCs on the core $G$ to define maps on configuration spaces. We will restrict our attention to those $G$ for which the boundary generically determines the VRC. Here and throughout, we use the Zariski topology on $\Gr_{m,n}$ and $\Conf^{\circ}_{m,n}$.

\begin{definition}
 A plabic graph $G$ is \emph{($m$-generically) solvable} if for all $\bz$ in an open dense subset of $\Conf^{\circ}_{m,n}$, or in an open dense subset of $\Gr^{\circ}_{m,n}$, 
 there is a unique $m$-VRC $[\bv, \bR] \in \mVRC_G$ with boundary $\bz$. We denote this $m$-VRC by $\zVRC.$ We say a plabic tangle is \emph{($m$-generically) solvable} if its core $G$ is and each blob $D \in \bD$ has size at least $m$.
\end{definition}

\begin{remark}
\cref{prop:solvable-only-one-m-possible} shows that if $G$ is reduced of type $(k,n)$, there is only one $m$ for which $G$ can be $m$-generically solvable. Namely, $m=(1/k)\dim \Pi_G = (\#(\text{faces of }G)-1)/k$. For this reason, we will drop ``$m$-generically" from now on. Being reduced and solvable is a property of the move-equivalence class of $G$, as is implied by \cref{cor:int1}. If $G$ and $G'$ are solvable and related by an expand-contract or square move, then the unique VRCs on $G, G'$ with boundary $\bz$ are related by the maps in \cref{lem:moves}, as is implied by \cref{cor:vrc-move-invariant}.
\end{remark}

We now use $\zVRC$ to define a rational map on configuration spaces.

\begin{definition}\label{def0:promotion}
Let $(G,\bD)$
be a solvable 
plabic tangle. We define a map
\begin{equation}\label{eq:promotion_tangle_conf}
       \gProm: \Conf^{\circ}_{m, n} \dashrightarrow \prod_{D \in \bD} \Conf_{m, D}\end{equation}
by first mapping generic $\bz\in \Conf^{\circ}_{m,n}$ to $\zVRC \in \mVRC_G$, the unique $m$-VRC with boundary $\bz$; since $\zVRC$ is a gauge-equivalence class of VRCs,
this associates a line $\CC v_b$ to each $b\in B$. Recall that for $D \in \bD$, each vertex $u$ 
in $D$ is adjacent to a unique element $b_u$ of $B$. We associate line $\CC v_{b_u}$ to $u$; by reading the vertices of $D$ in clockwise order starting just after the $\star$, we obtain the \emph{blob's boundary lines}, a tuple of $|D|$ lines which we identify with an element of $\Conf_{m,D}$. 
\end{definition}

It will frequently be useful for us to make the assumption that the blob's boundary lines usually give an element of $\Conf^{\circ}_{m,D}$. This is formalized in the following definition.

\begin{definition}\label{def:rank_reg_and_dom}
Let $(G,\bD)$ be a plabic tangle. An $m$-VRC $[\bv,\bR]\in\mVRC_G$ is \emph{rank-$m$ regular} for $(G,\bD)$ if for each blob $D \in \bD$,
the boundary lines $v_{b_u}$ of the blob
span $\CC^m$. If $(G, \bD)$ is solvable, we say that it is \emph{(rank-$m$) regular} if for generic $\bz \in \Conf^{\circ}_{m,n}$, $\zVRC$ is rank-$m$ regular. 
More generally, $(G, \bD)$ is \emph{rank-$m$ regular} if in the set
\[\{(\bz,[\bv,\bR]) \ \colon \ \bz\in\Conf^\circ_{m,n},~\partial[\bv,\bR]= \bz, ~[\bv, \bR]~\text{is nondegenerate}\},\]
the subset of points where $[\bv, \bR]$ is also rank-$m$ regular is dense.
\end{definition}
 \subsection{Pinnings, brushings, and maps on Grassmannians}

We would like to upgrade the map on configuration spaces
given in \cref{def0:promotion} to a rational 
map on Grassmannians instead. In order to do so, we need to choose a specific representative for $\zVRC$, in such a way that the coefficients $r_e$ are rational functions of $\bz$. We formalize this with the notion of a \emph{pinning}.

 \begin{definition}\label{def:pinning} Let $G$ be a solvable plabic graph.
 A collection of rational functions $\{r_e\}_{e \in E(G)} \subset \CC(\Gr_{m,n})$ is called a \emph{pinning} of $G$ if for generic $\bz$, $\zVRC$ has a representative of the form $(\{v_b\}, \{r_e(\bz)\})$. Abusing terminology, we may call the representative $(\{v_b\}, \{r_e(\bz)\})$ a pinning, and may write $(\{v_b(\bz)\}, \{r_e(\bz)\})$, since the vectors $v_b$ are determined by $\bz$ and $r_e(\bz)$ via \eqref{eq:v_b-from-paths}.
 \end{definition}

 We will show in \cref{prop:int-num-1-ratl-coeffs} that many pinnings exist; one may take $r_e(\bz)=f_e(\bz)$ as provided there, and then act by gauge by any rational function of $\bz$ at any vertex.

\begin{definition}\label{def:promotion} 
A \emph{pinned plabic tangle} $(G,\bD)$ is a solvable, rank-$m$ regular plabic tangle 
 together with a pinning $(\{v_b(\bz)\}, \{r_e(\bz)\})$  of $G$. Using this pinning, we define the following maps:  
\begin{itemize}
\item
A rational \emph{geometric promotion} map
    \begin{equation}\label{eq:promotion_tangle_geo}\gProm: \Gr_{m, n} \dashrightarrow \prod_{D \in \bD} \Gr_{m, D} \end{equation}
    is defined as follows. Fix generic $\bz \in \Gr_{m,n}$ and choose a blob $D \in \bD$. 
  Since each vertex $u$ 
in $D$ is adjacent to a unique element $b_u$ of $B$, we associate vector $v_{b_u}(\bz)$ to $u$, which we call the \emph{boundary vector} of $u$. By reading the vertices of $D$ in clockwise order starting just after the~$\star$, we obtain the \emph{blob's  boundary vectors}, which we assemble into a full-rank $m \times |D|$ matrix $M$, and identify with an element of the Grassmannian $\Gr_{m,D}$.  

\item
    \emph{(Algebraic) promotion} by the pinned plabic tangle is the pullback $\aProm:=\gProm^*$ of geometric promotion $\gProm$; that is, it is the map
    \begin{equation}\label{eq:promotion_tangle_alg} \aProm:\CC(\widehat{\Gr}_{m, D^{(1)}}) \otimes \dots \otimes 
    \CC(\widehat{\Gr}_{m, D^{(\ell)}})
    \to \CC(\widehat{\Gr}_{m, n}),\end{equation}
    which acts on Pl\"ucker coordinates by substituting 
    the $u$th column vector with the associated vector $v_{b_u}(\bz)$
    as above. 
\end{itemize}
\end{definition}

We note that by \cref{lem:v_b-from-paths}, the vectors $v_b(\bz)$ can be expressed as linear combinations of the columns of $\bz$ with coefficients that are rational functions of $\bz$, so $\aProm$ is indeed well-defined. 

One particularly nice way to obtain a pinned tangle is to \emph{brush} the tangle using collections of vertex-disjoint paths. 
\begin{definition}\label{def:brushed}
A \emph{brushed plabic tangle} is a plabic tangle $(G,\bD)$, together with a choice for each $D \in \bD$ of:
\begin{itemize}
\item a reverse acyclic perfect orientation $\OO^D$ of $G$ (see \cref{def:orientation})
\item a collection of oriented vertex-disjoint paths $\{P_u\}_{u \in D}$ in $\OO^D$ such that $P_u$ goes from some boundary vertex 
$i_u \in \Bb$
to the vertex $b_u$ adjacent to $u\in D$.
\end{itemize}
We write $\mcb:=\{\OO^D, \{P_u\}_{u \in D}\}_{D \in \bD}$ for the data of the brushing. We will also need to choose a sign $\sigma_{b_u} \in \{\pm 1\}$ for each $b_u \in \Bi$ adjacent to a boundary vertex of a blob. We denote the brushing together with this choice of signs by $\mcb^{\bsig}.$
\end{definition}
See \cref{fig:brushing} for an example of a brushing.

\begin{figure}
\includegraphics[width=\textwidth]{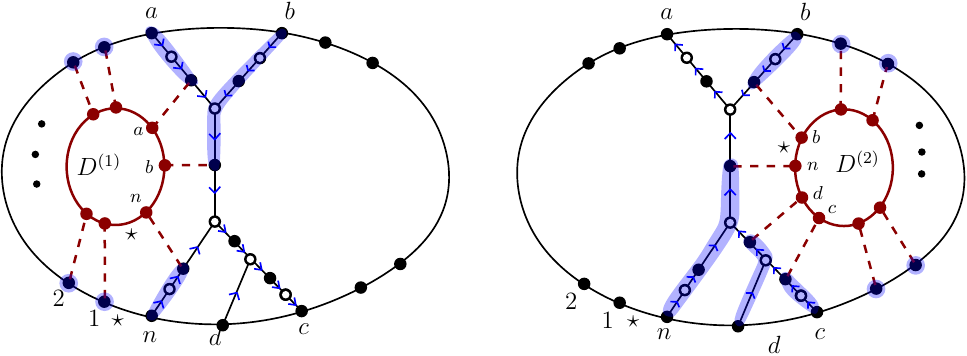}
\caption{A brushing for the tangle in \cref{fig:tangle}. On the left, the reverse acyclic perfect orientation and paths for $D^{(1)}$, and on the right, those for $D^{(2)}$. The paths are highlighted in blue. Note that some paths in the brushing are trivial and consist of a single boundary vertex.}
\label{fig:brushing}
\end{figure}

A brushed plabic tangle has a distinguished pinning, which can be used to define promotion.
\begin{definition}
\label{def:brushed-pinning}
Suppose $(G, \bD)$ is solvable plabic tangle and let $\mcb^{\bsig}$ be a brushing of $(G, \bD)$. 
Let $(\{v'_b(\bz)\}, \{r'_e(\bz)\})$ be an arbitrary pinning. For $D \in \bD$ and $u \in D$, set
\[\wt'(P_u):=  \frac{\displaystyle \prod_{e \in P_u} r'_e(\bz)}{ \displaystyle \prod_{e' \in P_u} (-r'_{e'}(\bz))}\]
where the product in the numerator is over edges $e \in P_u$ which are oriented black-to-white in~$\OO^D$, and the product in the denominator is over edges $e' \in P_u$ which are oriented white-to-black in~$\OO^D$. The \emph{$\mcb^{\bsig}$-pinning} is the pinning $(\{v_b(\bz)\}, \{r_e(\bz)\})$ obtained by applying the gauge transformations 
\[v'_{b_u}(\bz) \mapsto v_{b_u}(\bz):= \sigma_{b_u} \frac{v'_{b_u}(\bz)}{\wt'(P_u)} \]
at all boundary vectors of all blobs. By \cref{lem:b-pinning-well-defined} below, the $\mcb^{\bsig}$-pinning does not depend on the choice of pinning $(\{v'_b(\bz)\}, \{r'_e(\bz)\})$.

If $(G, \bD, \mcb^{\bsig})$ is a brushed tangle which is solvable and rank-$m$ regular, then \emph{algebraic (resp. geometric) promotion by $(G, \bD, \mcb^{\bsig})$} is algebraic (resp. geometric) promotion by $(G, \bD)$ using the $\mcb^{\bsig}$-pinning. That is, algebraic promotion is the map which acts on Pl\"ucker coordinates of $\Gr_{m, D}$ by substituting the $u$th column vector with $v_{b_u}(\bz)$.

\end{definition}

See \cref{sec_bcfw} for the promotion maps corresponding to the brushed tangle in \cref{fig:brushing}, with all signs chosen to be $1$.

\begin{lemma}\label{lem:b-pinning-well-defined} Let $\mcb$ be a brushing of a solvable plabic tangle $(G, \bD)$, and let $(\{v'_b(\bz)\}, \{r'_e(\bz)\})$, $(\{v''_b(\bz)\}, \{r''_e(\bz)\})$ be two pinnings of $G$. For any $D \in \bD$ and boundary vertex $u \in D$, we have 
\[\frac{v'_{b_u}(\bz)}{\wt'(P_u)} =\frac{v''_{b_u}(\bz)}{\wt''(P_u)}\]
where $\wt'(P_u), \wt''(P_u)$ are as in \cref{def:brushed-pinning}.
\end{lemma}
\begin{proof}
As $G$ is solvable, the two pinnings are related to each other by gauge transformations. Fix $u \in D$ for some $D \in \bD$. Note that $\wt'(P_u)$ is invariant under gauge transformations at $x \neq b_u$. So we may assume that the pinnings differ only by a gauge transformation at $b_u$. Say $v'_{b_u}(\bz) = t(\bz) \cdot v''_{b_u}(\bz)$, and for any edge $e$ adjacent to $b_u$, $r'_e(\bz) = r''_e(\bz)/t(\bz)$, where $t(\bz)$ is some rational function of $\bz$. Then we see that $\wt'(P_u)=t(\bz) \wt''(P_u)$, and so the ratio $v'_{b_u}(\bz)/\wt'(P_u)$ is equal to $v''_{b_u}(\bz)/\wt''(P_u)$.
\end{proof}

\begin{remark} Brushed tangles whose cores are related by expand-contract or certain square moves give rise to the same promotion maps. See \cref{rem:brushed,fig:brushed-square-moves} for more details.
\end{remark}

Not every plabic tangle admits a brushing. In fact, as we will show, having a brushing is related to the promotion map being dominant.

Given a plabic tangle $(G, \bD)$, let $\pi^D: \prod_{D' \in \bD} \Conf_{m,D'} \to \Conf_{m, D} $ denote the projection map and let $\gProm$ be the map from \eqref{eq:promotion_tangle_conf} which assigns to each blob its collection of boundary lines.

\begin{definition}\label{def:dominant}
A solvable plabic tangle $(G, \bD)$ is \emph{dominant} if it is rank-$m$ regular and, 
for each blob $D \in \bD$, the image of the map $\pi^D \circ \gProm$
contains a Zariski dense subset of $\Conf^{\circ}_{m, D}$.
\end{definition}

\begin{theorem}\label{prop:dominant_solvable}
A solvable plabic tangle $(G,\bD)$ is dominant if and only if it admits a brushing. 
\end{theorem}
The proof of this proposition is delayed until \cref{ssec:proof-dominant-solvable}.

Note that we can compose tangles by inserting the outer disk of one into an inner disk of another, provided that the number of boundary vertices and the positions of the $\star$ lines up. In \cref{sec:categorical_pov} we will show that with some care we can also compose brushed tangles. For now, we mention an application which allows us to produce many solvable graphs.

\begin{corollary}\label{cor:gluing} Let $(G, \bD)$ be a dominant plabic tangle, or equivalently a generically $m$-solvable plabic tangle which admits a brushing. For each $D \in \bD$, choose a generically $m$-solvable plabic graph $G^D$ with $|D|$ boundary vertices. Let $\tilde{G}$ denote the bipartite plabic graph obtained by identifying the boundary disk of each $G^D$ with the corresponding inner disk $D$, then ignoring $D$ and contracting the attaching segments. Then $\tilde{G}$ is also generically $m$-solvable.
\end{corollary}

\begin{proof}
Let $\psi$ be the map from \eqref{eq:promotion_tangle_conf} and for $D \in \bD$ let $\pi^D$ denote projection onto $\Conf_{m, D}$. Let $\mcu_G \subset \Conf^{\circ}_{m, n}$ be an open dense subset such that for $\bz \in \mcu_G$, there is a unique $[\bv, \bR] \in \mVRC_G$ with boundary $\bz$. Define $\mcu_D \subset \Conf^{\circ}_{m, D}$ similarly, using $G^D$ rather than $G$.

Since $(G, \bD)$ is dominant, we have $\pi^D \circ \gProm(\mcu_G)$ is dense in $\Conf^{\circ}_{m, D}$ and so has nonempty intersection with $\mcu_D$, which is open and dense. Thus the preimage of $\mcu_D$ under $\pi^D \circ \gProm$ is a nonempty open set $\mcu'_D$ in $\Conf^{\circ}_{m, n}$. 

Let $\mcu:= \mcu_G \cap \bigcap_{D \in \bD} \mcu'_D$. For $\bz \in \mcu$, $\bz$ uniquely determines the lines and relations for each vertex and edge of $G$. In particular, the boundary lines $\CC v_{b_u}$ of each blob $D$ are uniquely determined. By construction, the boundary lines give an element of $\mcu_D$ and so they in turn uniquely determine the lines and relations for each vertex and edge of $G^D$. This shows that for $\bz \in \mcu$, there is a unique $[\bv, \bR] \in \mVRC_{\tilde{G}}$ with boundary $\bz$, as desired.
\end{proof}

\begin{remark} \cref{cor:gluing} applied repeatedly to the BCFW core in \cref{fig:brushing}, together with its reflection and its cyclic shifts, shows that all $G$ arising in the BCFW recurrence are generically $4$-solvable. That is, $\Pi_G$ is an algebraic pre-tile for all $G$ arising in the BCFW recurrence and $\tZ$ is generically injective on $\Pi_G$. As mentioned in \cref{rem:tile-vs-pretile}, this does not immediately imply that $Z_G$ is a tile, which was proved using different methods in \cite{even2023cluster}.
\end{remark}

\subsection{Promotions from brushed tangles are conjecturally quasi-cluster}

We make the following conjectures regarding promotions from brushed solvable plabic tangles. In what follows, let $\Sigma_{m, n}$ denote a seed in the standard cluster structure on $\CC[\widehat{\Gr}_{m,n}]$.

\begin{conjecture}\label{conj:cluster1} Let $(G, \bD)$ be a dominant plabic tangle. Then there exists a brushing $\mcb$ and a choice of signs $\bsig$ such that the following hold.
\begin{enumerate}
\item Geometric promotion by $(G, \bD, \mcb^{\bsig})$
sends totally positive elements to totally positive elements.
\item Algebraic promotion by $(G, \bD, \mcb^{\bsig})$ 
\[\aProm:\CC(\widehat{\Gr}_{m, D^{(1)}}) \otimes \dots \otimes 
    \CC(\widehat{\Gr}_{m, D^{(\ell)}})
    \to \CC(\widehat{\Gr}_{m, n})\]
     restricts to a quasi-cluster homomorphism from 
    \[\mathcal{A}(\Sigma_{m, D^{(1)}} \sqcup \cdots \sqcup \Sigma_{m, D^{(\ell)}}) \to \mathcal{A}(\overline{\Sigma})\]
    where $\overline{\Sigma}$ is obtained from a seed $\Sigma_{m,n}$ for $\widehat{\Gr}_{m,n}$ by freezing some variables.   
\end{enumerate}
\end{conjecture}

\begin{remark}\label{conj:cluster-tree} Let $(G, \bD)$ be a dominant plabic tangle where $G$ is a plabic tree (cf. \cref{sec:trees}). Then we expect that \cref{conj:cluster1} holds for \emph{all} brushings $\mcb$; that is, for all brushings, there is a choice of signs so that (1) and (2) of \cref{conj:cluster1} hold.
\end{remark}

\section{Examples of promotions from tangles}\label{sec:promotion-examples}

In this section, we illustrate how to 
obtain promotion maps from various solvable plabic tangles, for cores with $k=1, 2, 3$.
We will show in \cref{sec:proofscluster} that all of these maps
are quasi-cluster homomorphisms.
For these tangles, \cref{def:vec-in-VRC-from-GC,def:coeffs-in-VRC-from-GC} gives formulas for the vectors and coefficients in $\zVRC$; we use these without justification. We also choose the signs $\bsig$ to be 1 on all vertices.

\subsection{Star promotion ($m \geq 3$, $k=1$)}
\label{sec:upper}

In this section we give a first example of how to use a vector relation configuration to get a map on Grassmannians which is a quasi-cluster homomorphism.  In particular, we introduce the notion of \emph{(unary) star promotion}, which uses the framework of \cref{def:tangle} and \cref{def:promotion}, to define 
a homomorphism from 
$\C(\widehat{\Gr}_{m,n-1})$
to $\C(\widehat{\Gr}_{m,n})$. We use a plabic tangle with a single inner disk (hence the word ``unary''),
whose core graph is a tree on $m+1$ leaves with only white vertices, see \cref{fig:upperMis3}.  This core is move-equivalent to 
a single white vertex with $m+1$ legs emanating from it, that is, a ``star graph.''

\begin{figure}[h]
\begin{tikzpicture}[scale=1,  
every node/.style={line width=2pt, circle, minimum size=2mm, draw, inner sep=0pt},
boundary/.style={line width=2pt, circle, minimum size=1mm, draw, inner sep=0pt, fill=black},
core/.style={line width=1.5pt},
blob/.style={line width=1.5pt,dash pattern=on 2pt off 1pt},
]
\node[boundary, label=above right:1] (b1) at (60:2) {};
\node[boundary, label=above right:2] (b2) at (20:2) {};
\node[boundary, label=below right:3] (b3) at (-20:2) {};
\node[boundary, label=below right:4] (b4) at (-60:2) {};
\node[boundary, label=below left:5] (b5) at (-100:2) {};
\node[boundary] (b11) at (-140:2) {};
\node[boundary] (b12) at (180:2) {};
\node[boundary] (b13) at (140:2) {};
\node[boundary, label=above left:$n$] (bn) at (100:2) {};
\node[draw=none,label=above right:$\star$] (star) at (80:2) {};
\node (blob) at (180:0.5) {};
\node[fill=white] (i1) at (1,1.2) {};
\node[fill=black] (i2) at (1,0.8) {};
\node[fill=white] (i3) at (1,0.4) {};
\node[fill=black] (i4) at (1,0) {};
\node[fill=white] (i5) at (1,-0.4) {};
\node[fill=black] (i6) at (1,-0.8) {};
\node[fill=white] (i7) at (1,-1.2) {};
\draw[core] (b1)--(i1)--(i2)--(i3)--(i4)--(i5)--(i6)--(i7)--(b4) (i3)--(b2) (i5)--(b3);
\draw[darkred,blob] (i2)--(blob) (i4)--(blob) (i6)--(blob) (b5)--(blob) (b11)--(blob) (b12)--(blob) (b13)--(blob) (bn)--(blob);
\draw[darkred,fill=white] (blob) circle[radius=0.9];
\draw[black] (0,0) circle[radius=2];
\node[draw=none] (star2) at ($(blob) + (60:1)$) {$\star$};
\node[draw=none] at ($(blob) + (40:0.7)$) {$1$};
\node[draw=none] at ($(blob) + (0:0.7)$) {$3$};
\node[draw=none] at ($(blob) + (-40:0.7)$) {$4$};
\end{tikzpicture}
\;\;\;\;\;\;\;\;\;\;
\begin{tikzpicture}[scale=1,  
every node/.style={line width=2pt, circle, minimum size=2mm, draw, inner sep=0pt},
boundary/.style={line width=2pt, circle, minimum size=1mm, draw, inner sep=0pt, fill=black},
core/.style={line width=1.5pt},
blob/.style={line width=1.5pt,dash pattern=on 2pt off 1pt},
]
\node[boundary, label=above:1] (b1) at (90:2) {};
\node[boundary, label=above right:2] (b2) at (45:2) {};
\node[boundary, label=above right:3] (b3) at (15:2) {};
\node[boundary, label=below right:4] (b4) at (-15:2) {};
\node[boundary, label=below right:5] (b5) at (-45:2) {};
\node[boundary, label=below:6] (b6) at (-90:2) {};
\node[boundary, label=below left:7] (b7) at (-120:2) {};
\node[boundary] (b11) at (-150:2) {};
\node[boundary] (b12) at (180:2) {};
\node[boundary] (b13) at (150:2) {};
\node[boundary, label=above left:$n$] (bn) at (120:2) {};
\node[draw=none,label=$\star$] (star) at (105:2) {};
\node (blob) at (180:0.7) {};
\node[fill=white] (i1) at (0.25,1.75) {};
\node[fill=black] (i2) at (0.5,1.5) {};
\node[fill=white] (i3) at (0.8,1.2) {};
\node[fill=black] (i4) at (1,0.8) {};
\node[fill=white] (i5) at (1,0.4) {};
\node[fill=black] (i6) at (1,0) {};
\node[fill=white] (i7) at (1,-0.4) {};
\node[fill=black] (i8) at (1,-0.8) {};
\node[fill=white] (i9) at (0.8,-1.2) {};
\node[fill=black] (i10) at (0.5,-1.5) {};
\node[fill=white] (i11) at (0.25,-1.75) {};
\draw[core] (b1)--(i1)--(i2)--(i3)--(i4)--(i5)--(i6)--(i7)--(i8)--(i9)--(i10)--(i11)--(b6) (i3)--(b2) (i5)--(b3) (i7)--(b4) (i9)--(b5);
\draw[darkred,blob] (i2)--(blob) (i4)--(blob) (i6)--(blob) (i8)--(blob) (i10)--(blob) (b7)--(blob) (b11)--(blob) (b12)--(blob) (b13)--(blob) (bn)--(blob);
\draw[darkred,fill=white] (blob) circle[radius=1];
\draw[black] (0,0) circle[radius=2];
\node[draw=none] (star2) at ($(blob) + (75:1.1)$) {$\star$};
\node[draw=none] at ($(blob) + (60:0.8)$) {$1$};
\node[draw=none] at ($(blob) + (30:0.8)$) {$3$};
\node[draw=none] at ($(blob) + (0:0.8)$) {$4$};
\node[draw=none] at ($(blob) + (-30:0.8)$) {$5$};
\node[draw=none] at ($(blob) + (-60:0.8)$) {$6$};
\end{tikzpicture}
\caption{Brushed tangles for unary star promotion for $m=3$ and $m=5$
\label{fig:upperMis3}. The labels of the inner disk vertices indicate the brushing: the paths go from the boundary vertex $i \in \Bb$ to the (black vertex attached to the) inner disk vertex labeled $i$. These paths extend to a reverse perfect orientation.
}
\end{figure}
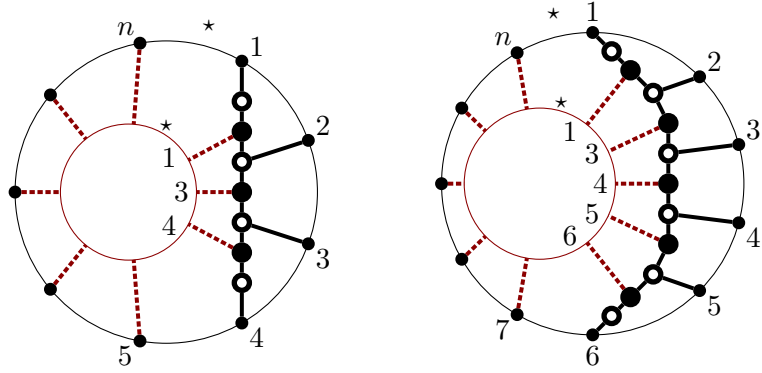

We will prove in \cref{sec:proofscluster} that unary star promotion is a
quasi-cluster homomorphism. Therefore if we apply unary star promotion to
a cluster variable (resp. subset of a cluster),
we obtain a cluster variable (resp. subset of a cluster)  of
$\Gr_{m,n}$, up to a Laurent monomial in
a certain set of distinguished Pl\"ucker coordinates. 

In accordance with the brushed plabic tangles in \cref{fig:upperMis3}, we let $N'=\{1,3,\dots,n\}$. We first apply \cref{def:promotion}
to the diagram at the left of \cref{fig:upperMis3} (with all signs equal to~$1$), and obtain the $m=3$
\emph{unary star promotion map} described in \cref{pro-upper:0}.

\begin{definition}[Unary star promotion for $m=3$]
\label{pro-upper:0}
\emph{Unary star promotion} (for $m=3$) is the homomorphism
$$\Psi_{3} : \C(\widehat{\Gr}_{3,N'}) \;\to\; \C(\widehat{\Gr}_{3,n})$$
induced
by the following substitution on vectors $1,3,4,\dots,n$ in $\CC^3$:
\begin{align*}
   3 &\;\mapsto\; 
\frac{12 \shuf 34}{\lr{124}} 
\;=\; 3-\frac{\lr{123}}{\lr{124}} \, 4 \\
   j &\; \mapsto\; j \hspace{6cm} \text{ for } j =1,4,5,\dots,n.
\end{align*}
\end{definition}

\begin{remark}
We can give an equivalent definition of $\Psi_3$ using Pl\"ucker coordinates. 
Recall the chain polynomial notation from \cref{def:chain}.  We have that 
\begin{align*}
\Psi_3\left(\lr{3ab}\right) &\;=\; \lr{3ab} - \frac{\lr{123}}{\lr{124}} \lr{4ab} \;=\; \frac{\lr{1\,2 \shuf 3\,4 \shuf a\,b}}{\lr{124}}\\
\Psi_3\left(\lr{abc}\right) &\;=\; \lr{a,b,c} \hspace{7cm} \text{ if } 3\notin \{abc\}.
\end{align*}
\end{remark}

We now give the generalization of \cref{pro-upper:0} for general $m \geq 3.$  This comes from star plabic tangles as in 
\cref{fig:upperMis3}, where the additional example $m=5$ is shown on the right.

\begin{definition}[Unary star promotion]
\label{pro-upper}
Fix a positive integer $m \ge 3$.
\emph{Unary star promotion} is the homomorphism
$$\Psi_{m} : \C(\widehat{\Gr}_{m,N'}) \;\to\; \C(\widehat{\Gr}_{m,n})$$
induced by the following substitution on vectors $1,3,4,\dots,n$ in $\CC^m$:
\begin{align*}
j &\;\mapsto\; \frac{(1 \wedge 2 \wedge \dots \wedge j{-}1) \shuf (j \wedge j{+}1 \wedge \dots \wedge m \wedge m{+}1))}{\lr{ 1,2,\dots,j{-}1,j{+}1,\dots,m,m{+}1}} 
&&\text{ for }3 \leq j \leq m \\
j &\; \mapsto\; j &&\text{ for }j =1 \text{ and } m+1 \leq j \leq n
\end{align*}
\end{definition}

\subsection{BCFW Promotion $(m=4, k=1)$}\label{sec_bcfw}
This section discusses the promotion map coming from the brushed tangle in \cref{fig:brushing}. This promotion could be called \emph{binary star promotion}.

\begin{definition}(BCFW Promotion)
\label{def:bcfw_promotion}
Let $N_L=\{1,2,\ldots,a,b,n\}$ and $N_R=\{b,b+1,\ldots,c,d,n\}$, 
 where $a,b$ and $c,d,n$ are consecutive. 
 Let $F_e$ denote the Pl\"ucker coordinate obtained by formally erasing the label $e$ from $\lr{abcdn}$, e.g.
$F_b=\lr{acdn}$ and $F_n=\lr{abcd}$. 
\emph{BCFW promotion} is the homomorphism
$$\Psi_{BCFW}: \CC(\Gr_{4,N_L}) \times \CC(\Gr_{4,N_R}) \;\rightarrow\; \CC(\Gr_{4,n})    
$$
induced by the substitutions:
\begin{align*}
 \mbox{on } \Gr_{4,N_L}:& \;\;\;\;\;\; b \mapsto \frac{ab \shuf cdn}{F_b}\\
 \mbox{on } \Gr_{4,N_R}:& \;\;\;\;\;\; d \mapsto \frac{cd \shuf abn}{F_d}
\;\;\;\;\;\;
n \mapsto \frac{cdn \shuf ab}{F_n} \;\;\;\;\;\; 
\end{align*}

and $j \mapsto j$ otherwise. 
\end{definition}

The authors together with T. Lakrec proved that $\Psi_{BCFW}$ is a quasi-cluster homomorphism in 
\cite[Theorem 4.7]{even2023cluster}.

\subsection{Spurion promotion $(m=4, k=2)$}
\label{sec:spurion}

The next example applies \cref{def:promotion} to obtain the promotion map $\Psi_{sp}$ by the plabic tangle $G_{sp}$ in \cref{fig:upper_spurion}. In this case, the core is the \emph{spurion}\footnote{See \cite{companion} for etymology and the appearance of this graph in tilings of the amplituhedron.}, with $k=2$, and it is attached with $5$ legs to one blob. The resulting promotion map is not obtainable by composing two star promotions.

\begin{figure}[htbp]
\centering
\begin{tikzpicture}[scale=1,  
every node/.style={line width=2pt, circle, minimum size=2mm, draw, inner sep=0pt},
boundary/.style={line width=2pt, circle, minimum size=1mm, draw, inner sep=0pt, fill=black},
core/.style={line width=1.5pt},
blob/.style={line width=1.5pt,dash pattern=on 2pt off 1pt},
]
\node[boundary, label=above:1] (b1) at (90:2) {};
\node[boundary, label=above right:2] (b2) at (60:2) {};
\node[boundary, label=above right:3] (b3) at (30:2) {};
\node[boundary, label=right:4] (b4) at (0:2) {};
\node[boundary, label=below right:5] (b5) at (-30:2) {};
\node[boundary, label=below right:6] (b6) at (-60:2) {};
\node[boundary, label=below:7] (b7) at (-90:2) {};
\node[boundary, label=below left:8] (b8) at (-120:2) {};
\node[boundary, label=below left:9] (b9) at (-150:2) {};
\node[boundary, label=left:10] (b10) at (-180:2) {};
\node[boundary] (b11) at (165:2) {};
\node[boundary] (b12) at (150:2) {};
\node[boundary] (b13) at (135:2) {};
\node[boundary, label=above left:$n$] (bn) at (120:2) {};
\node[draw=none,label=above left:$\star$] (star) at (100:2) {};
\node (blob) at (150:1) {};
\node[fill=black] (i1) at (80:1.6) {};
\node[fill=white] (i2) at (70:1.2) {};
\node[fill=black] (i3) at (60:0.8) {};
\node[fill=white] (i4) at (30:0.5) {};
\node[fill=black] (i5) at (-30:0.5) {};
\node[fill=white] (i6) at (-90:0.5) {};
\node[fill=black] (i7) at (-120:0.8) {};
\node[fill=white] (i8) at (-130:1.2) {};
\node[fill=black] (i9) at (-140:1.6) {};
\node[fill=white] (i0) at (-30:1.25) {};
\draw[core] (b1)--(i1)--(i2)--(i3)--(i4)--(i5)--(i6)--(i7)--(i8)--(i9)--(b9);
\draw[core] (b2)--(i2) (b3)--(i4) (b4)--(i0) (b5)--(i0) (b6)--(i0) (b7)--(i6) (b8)--(i8) (i0)--(i5);
\draw[darkred,blob] (i1)--(blob) (i3)--(blob) (i5)--(blob) (i7)--(blob) (i9)--(blob) (b10)--(blob) (b11)--(blob) (b12)--(blob) (b13)--(blob) (bn)--(blob);
\draw[darkred,fill=white] (blob) circle[radius=0.7];
\draw[black] (0,0) circle[radius=2];
\node[draw=none] (star2) at ($(blob) + (70:0.8)$) {$\star$};
\node[draw=none] at ($(blob) + (50:0.5)$) {$1$};
\node[draw=none] at ($(blob) + (10:0.5)$) {$2$};
\node[draw=none] at ($(blob) + (-30:0.5)$) {$7$};
\node[draw=none] at ($(blob) + (-70:0.5)$) {$8$};
\node[draw=none] at ($(blob) + (-110:0.5)$) {$9$};
\end{tikzpicture}
\caption{A plabic tangle $(G_{sp},\{1,2,7,8,9,\dots,n\})$ for unary spurion promotion. The labels of the inner disk indicate the brushing: the paths go from the respective boundary vertices $1,2,7,8,9$ to the (black vertex attached to the) blob vertex with the same label. These paths can be completed to a reverse perfect orientation.
\\
}
\label{fig:upper_spurion}
\end{figure}
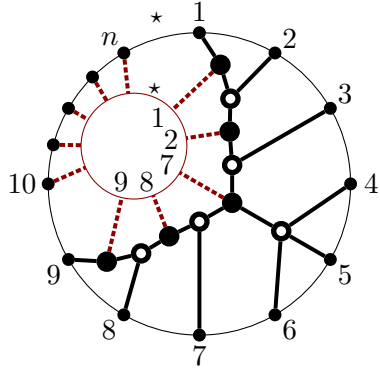

\begin{definition}(Unary spurion promotion)
\label{def:upper_spurion_promotion}
Let $N'=\{1,2, 7,8,\ldots,n\}$. 
 Let $F_i$ denote the expression obtained by formally erasing the label $i$ from the expression $\lr{123 \shuf 456 \shuf 789}$, e.g.
$F_2=\lr{13 \shuf 456 \shuf 789}$.
\emph{Unary spurion promotion} is the homomorphism
$$ \Psi_{sp}: \CC(\Gr_{4,N'}) \;\rightarrow\; \CC(\Gr_{4,n})    
$$
induced by the substitutions: 
$$
2 \mapsto \frac{12 \shuf 3 (456 \shuf 789)}{F_2} \;\;\;\;\;\;
7 \mapsto \frac{123 \shuf 456 \shuf 789}{F_7} \;\;\;\;\;\;
8 \mapsto \frac{(123 \shuf 456)7 \shuf 89}{F_8}
$$
and $j \mapsto j$ otherwise. 
\end{definition}

\begin{remark} \label{rk:intersections}
By the identities in \cref{ex:intersections}, the substitutions in unary spurion promotion can be written as follows:
\begin{align*} 
2 \;\;\mapsto\;\; & 2 - \frac{F_1}{F_2}1 \;\;=\;\; \frac{F_3 \cdot 3 - F_7 \cdot 7 + F_8 \cdot 8 - F_9 \cdot 9}{F_2} \\[5pt]
7 \;\;\mapsto\;\; & 7-\frac{F_8}{F_7}8+\frac{F_9}{F_7}9 \;\;=\;\; \frac{F_1 \cdot 1-F_2 \cdot 2+F_3 \cdot 3}{F_7} \\[5pt]
8 \;\;\mapsto\;\; & 8 - \frac{F_9}{F_8}9 \;\;=\;\; \frac{F_7 \cdot 7-F_1 \cdot 1+F_2 \cdot 2-F_3 \cdot 3}{F_9}
\end{align*}
We note that all the chain polynomials appearing in this formula have sets of size $3 = m-1$. Using \cref{cor:irreducible-quadratics} with appropriate cyclic shifts, these polynomials are cluster variables and so are irreducible in $\C[\widehat{\Gr}_{4,n}]$ by \cite[Theorem 1.3]{GLS}.
\end{remark}

\subsection{Chain-tree promotion ($m=4, k=3$)}\label{sec:chain}
An interesting class of examples are promotions induced by plabic tangles whose cores are \emph{chain trees}. These trees let us generalize the promotions by the $k=1$ star and the $k=2$ spurion to higher values of~$k$.

\begin{definition}
\label{def-chain-tree}
(Chain tree) A $(t_1,t_2,\dots,t_{2k-1})$-\emph{chain tree} is a plabic graph of type $(k,\sum_i t_i)$, which is a tree of the following form:
\begin{center}
\begin{tikzpicture}[scale=1,  
every node/.style={line width=2pt, circle, minimum size=2mm, draw, inner sep=0pt},
boundary/.style={line width=2pt, circle, minimum size=1mm, draw, inner sep=0pt, fill=black},
none/.style={inner sep=1pt, minimum size=0pt},
core/.style={line width=1.5pt},
blob/.style={line width=1.5pt,dash pattern=on 2pt off 1pt},
]
\foreach \x in {0,1,4,7,8,11,14,15,18,21,22,25,33,34,37,40,41,44} 
\node[boundary] (b\x) at (0.3*\x,0) {};
\foreach \x in {2.5,9.5,16.5,23.5,35.5,42.5} 
\node[draw=none] at (0.3*\x,0) {\color{gray}$\dots$};
\foreach \x in {2,16,42}
\node[fill=white] (w\x) at (0.3*\x,-1) {};
\foreach \x in {9,23,35}
\node[fill=black] (b\x) at (0.3*\x,-1.5) {};
\foreach \x in {9,23,35}
\node[fill=white] (w\x) at (0.3*\x,-0.75) {};
\node[draw=none] at (0.3*29,-0.5) {\color{black}$\dots$};
\draw[core] (w2) -- (b9) -- (w16) -- (b23) -- (0.3*26.5,-1.25)
(w2) -- (b0)
(w2) -- (b1)
(w2) -- (b4)
(w9) -- (b9) 
(w9) -- (b7) 
(w9) -- (b8) 
(w9) -- (b11) 
(w16) -- (b14)
(w16) -- (b15)
(w16) -- (b18)
(w23) -- (b21) 
(w23) -- (b22) 
(w23) -- (b25) 
(w23) -- (b23)
(w35) -- (b33) 
(w35) -- (b34) 
(w35) -- (b37) 
(w35) -- (b35)
(w42) -- (b40)
(w42) -- (b41)
(w42) -- (b44)
(0.3*31.5,-1.25) -- (b35) -- (w42);
\draw [decorate,decoration={brace,amplitude=6pt},thick] (0.3*-1,0.3) -- (0.3*5,0.3) node[above=9pt,midway,draw=none]{$t_1$};
\draw [decorate,decoration={brace,amplitude=6pt},thick] (0.3*6,0.3) -- (0.3*12,0.3) node[above=9pt,midway,draw=none]{$t_2$};
\draw [decorate,decoration={brace,amplitude=6pt},thick] (0.3*13,0.3) -- (0.3*19,0.3) node[above=9pt,midway,draw=none]{$t_3$};
\draw [decorate,decoration={brace,amplitude=6pt},thick] (0.3*20,0.3) -- (0.3*26,0.3) node[above=9pt,midway,draw=none]{$t_4$};
\draw [decorate,decoration={brace,amplitude=6pt},thick] (0.3*32,0.3) -- (0.3*38,0.3) node[above=3pt,midway,draw=none]{$t_{2k-2}$};
\draw [decorate,decoration={brace,amplitude=6pt},thick] (0.3*39,0.3) -- (0.3*45,0.3) node[above=3pt,midway,draw=none]{$t_{2k-1}$};
\end{tikzpicture}
\end{center}
\end{definition}

For example, the $(3,3,3)$-chain tree for $k=2$ is the spurion discussed above. Other chain-trees of interest are the $(3,3,1,3,3)$-chain tree for $k=3$, the $(3,3,1,3,1,3,3)$-chain tree for $k=4$, and in general the $(3,3,1,3,1,\dots,3,1,3,3)$-chain tree for every~$k$. This sequence of plabic graphs are $(k,4)$-amplitrees, cf \cref{def:mbalanced}, and can be used to define new promotion maps for~$m=4$. 

The following example uses the $(3,3,1,3,3)$-chain tree and one blob to define the $k=3$ chain-tree promotion map $\Psi_{ch}$, which is not obtainable by composing the previous promotions of $k=1$ and~$2$.

\begin{figure}[h]
\begin{tikzpicture}[scale=1,  
every node/.style={line width=2pt, circle, minimum size=2mm, draw, inner sep=0pt},
boundary/.style={line width=2pt, circle, minimum size=1mm, draw, inner sep=0pt, fill=black},
core/.style={line width=1.5pt},
blob/.style={line width=1.5pt,dash pattern=on 2pt off 1pt},
]
\node[boundary, label=above:1] (b1) at (90:3) {};
\node[boundary, label=above right:2] (b2) at (70:3) {};
\node[boundary, label=above right:3] (b3) at (50:3) {};
\node[boundary, label=right:4] (b4) at (30:3) {};
\node[boundary, label=right:5] (b5) at (10:3) {};
\node[boundary, label=below right:6] (b6) at (-10:3) {};
\node[boundary, label=below right:7] (b7) at (-30:3) {};
\node[boundary, label=below right:8] (b8) at (-50:3) {};
\node[boundary, label=below:9] (b9) at (-70:3) {};
\node[boundary, label=below:$\aa$] (bA) at (-90:3) {};
\node[boundary, label=below left:$\bb$] (bB) at (-110:3) {};
\node[boundary, label=below left:$\cc$] (bC) at (-130:3) {};
\node[boundary, label=below left:$\dd$] (bD) at (-150:3) {};
\node[boundary, label=left:$\ee$] (b10) at (-170:3) {};
\node[boundary] (b11) at (170:3) {};
\node[boundary] (b12) at (150:3) {};
\node[boundary] (b13) at (130:3) {};
\node[boundary, label=above left:$n$] (bn) at (110:3) {};
\node[draw=none,label=above:$\star$] (star) at (100:3) {};
\node (blob) at (150:1) {};
\node[fill=black] (j1) at (85:2.6) {};
\node[fill=white] (i2) at (70:2.1) {};
\node[fill=black] (j2) at (50:1.8) {};
\node[fill=white] (i3) at (30:1.6) {};
\node[fill=black] (j3) at (0:1.3) {};
\node[fill=white] (i7) at (-30:1.3) {};
\node[fill=black] (jB) at (-60:1.3) {};
\node[fill=white] (iB) at (-90:1.6) {};
\node[fill=black] (jC) at (-110:1.8) {};
\node[fill=white] (iC) at (-130:2.1) {};
\node[fill=black] (jD) at (-145:2.6) {};
\node[fill=white] (i456) at (10:2.2) {};
\node[fill=white] (i89A) at (-70:2.2) {};
\draw[core] (b1)--(j1)--(i2)--(j2)--(i3)--(j3)--(i7)--(jB)--(iB)--(jC)--(iC)--(jD)--(bD);
\draw[core] (i2)--(b2) (i3)--(b3) (iB)--(bB) (iC)--(bC) (i7)--(b7) (b4) -- (i456) -- (b5) (b6) -- (i456) -- (j3) (b8) -- (i89A) -- (b9) (bA) -- (i89A) -- (jB);
\draw[darkred,blob] (j1)--(blob) (j2)--(blob) (j3)--(blob) (jB)--(blob) (jC)--(blob) (jD)--(blob) (b10)--(blob) (b11)--(blob) (b12)--(blob) (b13)--(blob) (bn)--(blob);
\draw[darkred,fill=white] (blob) circle[radius=1];
\draw[black] (0,0) circle[radius=3];
\node[draw=none] at ($(blob) + (60:0.75)$) {$1$};
\node[draw=none] at ($(blob) + (25:0.75)$) {$2$};
\node[draw=none] at ($(blob) + (-10:0.75)$) {$3$};
\node[draw=none] at ($(blob) + (-50:0.75)$) {$\bb$};
\node[draw=none] at ($(blob) + (-85:0.75)$) {$\cc$};
\node[draw=none] at ($(blob) + (-120:0.75)$) {$\dd$};
\node[draw=none] (star2) at ($(blob) + (80:1.125)$) {$\star$};
\end{tikzpicture}
\caption{The plabic tangle for the $k=3$ chain-tree promotion. As in \cref{fig:upper_spurion}, the labels $1,2,3,\bb,\cc,\dd$ in the blob indicate the brushing.
\label{fig:chain}}
\end{figure}
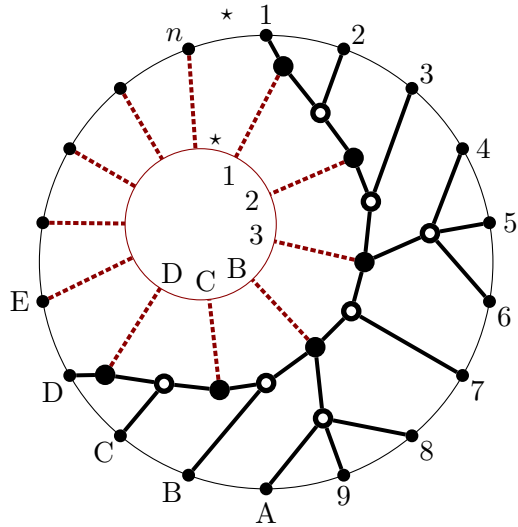
\begin{definition}(Chain-tree Promotion)
\label{def:chain_promotion}
Let $N'=\{1,2,3,\bb,\cc,\dd,\ldots,n\}$ where we extend the numerals by $\aa,\bb,\cc,\dd$ to abbreviate $10,11,12,13$. Let $F_i$ denote the function obtained by ``formally erasing the label $i$" from the expression $\lr{(123 \shuf 456) 7 (89\aa \shuf \bb\cc\dd)}$, e.g. 
$$ F_2=\lr{(13 \shuf 456) 7 (89\aa \shuf \bb\cc\dd)}, \;\; F_7=\lr{123 \shuf 456 \shuf 89\aa \shuf \bb\cc\dd}, \;\; F_8=\lr{(123 \shuf 456) 7(9\aa \shuf \bb\cc\dd)}.$$ 
\emph{Chain-tree promotion} is the homomorphism
$$ \Psi_{ch}: \CC(\Gr_{4,N'}) \;\rightarrow\; \CC(\Gr_{4,n})    
$$
induced by the substitutions:
\begin{align*}
& 2 \;\mapsto\; \frac{12 \shuf (3(456 \shuf (7(89\aa\shuf \bb\cc\dd)  )))}{F_2} &&
3 \;\mapsto\; \frac{123 \shuf 456 \shuf (7(89\aa \shuf \bb\cc\dd))}{F_3}
\\&
B \;\mapsto\; \frac{\bb\cc\dd \shuf 89\aa \shuf (7(123 \shuf 456))}{F_B}
&& C \;\mapsto\; \frac{\cc\dd \shuf (\bb(89\aa \shuf (7(456 \shuf 123))}{F_C} 
\end{align*}
and $j \mapsto j$ otherwise. 
\end{definition}

\subsection{Forest promotion ($m=3, k=2$)}\label{sec:forest}
An interesting class of examples are promotions where the core has several connected components. 
We illustrate this by an example where the core of the plabic tangle is a forest with two components, attached to the same blob.
 
\begin{definition}(Forest Promotion)
\label{def:forest_promotion}
Choose $4$ consecutive numbers, $5 \leq a <b <c <d \leq n$, and let $N'=\{1,2,4,5,\ldots,a-1,a,b,d,d+1,\ldots,n\}$. \emph{Forest promotion} is the homomorphism
$$ \Psi_{fo}: \CC(\Gr_{3,N'}) \;\rightarrow\; \CC(\Gr_{3,n})    
$$
induced by the substitutions: 
$$
2 \mapsto \frac{12 \shuf 34}{\lr{134}} \;\;\;\;\;\;
b \mapsto \frac{ab \shuf cd}{\lr{acd}}
$$
and $j \mapsto j$ otherwise. 
\end{definition}

\begin{figure}[h]
\includegraphics[width=0.7\textwidth]{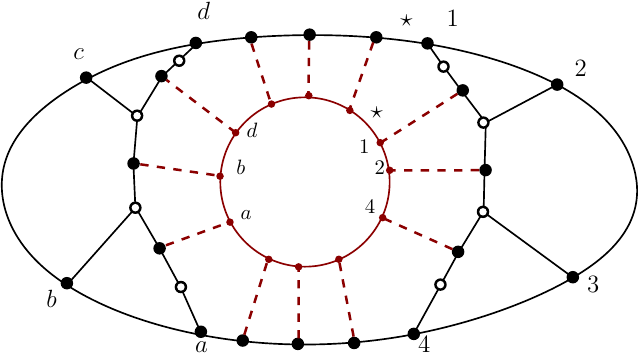}
\caption{The plabic tangle for forest promotion in \cref{def:forest_promotion}. As in Figures \ref{fig:upperMis3},
\ref{fig:upper_spurion} and \ref{fig:chain}, the labels in the blob indicate the brushing.}
\end{figure}

\section{Vector-relation configurations and intersection number}

In this section, we relate $m$-VRCs on $G$ and the $m$-intersection number of the positroid variety $\Pi_G \subset \Gr_{k,n}$ when $\dim \Pi_G = km$.  
In \cref{cor:int-num=num-VRC}, we show
the number of $m$-VRCs with fixed generic boundary vectors is equal to the intersection number of $\Pi_G$. In particular, $G$ is $m$-generically solvable if and only if 
$\Pi_G$ has $m$-intersection number $1$ and $\Pi_G$ has dimension $km$ (\cref{cor:int1}).
Throughout this section, we fix $k,m,n$ such that $k+m\leq n$. The words ``open", ``closed", and ``dense" all refer to the Zariski topology.

\subsection{Rephrasing intersection number}

In this section we will translate the notion of 
inter- section number (see 
\cref{def:intnumber}) in terms of 
generic $Z \in \Mat_{n, k+m}^\circ$, to a condition about 
generic $z\in \Gr_{m,n}$, see
\cref{prop:int-num-in-terms-of-Gr_mn}. In \cref{prop:int1criterion} we will also give a combinatorial criterion on the plabic graph $G$
for
verifying that a positroid variety $\Pi_G$ has
intersection number $0$.

Recall the definition of $G^{op}$ from \cref{not:bipartite}. Recall from \cref{lem:perp-takes-Pi-G-to-G-op} that $V \mapsto V^\perp$ is an involutive isomorphism from $\Pi_H$ to $\Pi_{H^{op}}$.

\begin{lemma}\label{lem:too-low-dim-maps-to-pos-codim}
    Fix $k,m$. Suppose $H$ is a plabic graph of type $(k,n)$ and set
    \[
    Q_H:=\{\bz \in \Gr_{m,n}:\bz \subset V \text{ for some } V \in \Pi_{H^{op}} \}=\{\bz \in \Gr_{m,n}:\bz^\perp \supset C \text{ for some }C \in \Pi_H\}.\]
    If $\dim \Pi_H < km$, then $Q_H$ is an irreducible subvariety of $\Gr_{m,n}$ of positive codimension.
\end{lemma}
\begin{proof}
    Consider the incidence variety
    \[\{(\bz, V) \in  \Gr_{m,n} \times \Pi_{H^{op}}: \bz \subset V\} \subset \Gr_{m,n} \times \Gr_{n-k, n}.\]
    This is a projective variety (cut out by the equations for $\Pi_{H^{op}}$ and the incidence-Pl\"ucker relations) and is irreducible as $\Pi_{H^{op}}$ is irreducible. The projection onto the first coordinate, which is a proper map, gives $Q_H$. This implies $Q_H$ is closed and irreducible.

    Now, we turn to dimension. There are $\dim \Pi_{H^{op}}= \dim \Pi_{H}$ degrees of freedom to choose $V\in \Pi_{H^{op}} \subset \Gr_{n-k,n}$. Once $V$ is chosen, $\bz$ can be any element of $\Gr_m(V) \cong \Gr_{m, n-k}$, which has dimension $m(n-k-m)$. So the incidence variety has dimension
    \[\dim \Pi_{H} + m(n-k-m) < km +m(n-k-m) = m(n-m) = \dim \Gr_{m,n}.\]
    Since $Q_H$ is a projection of the incidence variety, its dimension is at most the dimension of the incidence variety, hence it has positive codimension in $\Gr_{m,n}$.

\end{proof}

\begin{remark}\label{rem:more-maps-to-pos-codim}
 \cref{lem:too-low-dim-maps-to-pos-codim} holds by the same proof if we replace $\Pi_{H^{op}}$ with an irreducible subvariety $X \subset Gr_{n-k, n}$ satisfying $\dim X < km$ (and replace $\Pi_H$ with the image of $X$ under $V \mapsto V^{\perp}$). 
\end{remark}

In the next proposition, we use $T_G$ to denote the image of the boundary measurement map (see \cref{def:pathmatrix}), which is an algebraic torus by \cite[Theorem 7.1]{MullerSpeyerTwist}. 
The torus $T_G$ is a Zariski-open, dense subset of $\Pi_G$.

\begin{proposition}\label{prop:int-num-in-terms-of-Gr_mn}
Suppose that $G$ is a plabic graph of type $(k,n)$, and $\dim \Pi_G = km$. The following are equivalent:
\begin{enumerate}
\item $\Pi_G$ has $m$-intersection number $d$;
\item for an open dense set of $\bz \in \Gr_{m,n}$, the set $\{C \in \Pi_G: C \subset \bz^\perp\}$ has cardinality $d$;
\item for an open dense set of $\bz \in \Gr_{m,n}$, the set $\{C \in T_G: C \subset \bz^\perp\}$ has cardinality $d$.
\end{enumerate}
\end{proposition}
\begin{proof}

For $z\in\Gr_{m,n}$, let $P_z=\Gr_{k}(z^\perp)$ be the collection of $k$-planes $V \in \Gr_{k,n}$ which are contained in $z^\perp$. This is a Schubert variety of dimension $ k(n-m-k)$ which is isomorphic to $\Gr_{k, n-m}$. Notice that $P_z \cap \Pi_G$ and $P_z \cap T_G$ are the sets mentioned in (2) and (3), respectively.

$(1 \Leftrightarrow 2)$:
By \cref{prop:cohom-characterization-int-num}, in the cohomology ring $H^*(\Gr_{k,n})$, we have $[\Pi_G]\cap[P_z]= \mint(G) \cdot [\pt].$ This means that $\Pi_G \cap P_z$ will consist of exactly $\mint(G)$ points, provided that the intersection is transverse. We now show that this intersection is transverse for a generic set of $z$.  Since $GL_n$ acts transitively on $\Gr_{m,n}$,
and $(P_z) \cdot g$ is equal to $P_{zg}$, by Kleiman's theorem there's an open dense subset of $g\in GL_n$ such that 
$\Pi_G$ and $P_{zg}$ intersect transversely.  Therefore for generic $z\in \Gr_{m,n}$,
 $\mint(G) =|P_z \cap \Pi_G|,$ and hence $\mint(G)=|\{C \in \Pi_G: C \subset \bz^\perp\}|$.

$(2 \Leftrightarrow 3)$: Consider $ \Pi_G \setminus T_G$, which is a closed subvariety of $\Pi_G$ of dimension less than $km$. \cref{lem:too-low-dim-maps-to-pos-codim,rem:more-maps-to-pos-codim} imply that $\{\bz\in \Gr_{m,n}:\bz^\perp \supset C \text{ for some }C \in \Pi_G \setminus T_G\}$ is contained in a positive codimension subvariety $\mcv$ of $\Gr_{m,n}$. For $\bz$ in the open dense set $\Gr_{m,n} \setminus \mcv$, any $C \in \Pi_G$ contained in $\bz^\perp$ lies in $T_G$, and so 
\[|P_{\bz} \cap \Pi_G|=|P_{\bz} \cap T_G|. \]
Therefore for $\bz$ in an open dense set of $\Gr_{m,n}$, the two  sets $\{C \in \Pi_G: C \subset \bz^\perp\}$ and $\{C \in T_G: C \subset \bz^\perp\}$ have the same cardinality.
\end{proof}

\cref{prop:int1criterion} below gives a useful criterion for
verifying that a positroid variety $\Pi_G$ has
intersection number $0$.  We will use this in the next section in the proof of \cref{prop:tree1}.

\begin{lemma}\label{lem:geom-criterion-int0} Fix $\Pi_G$ of dimension $km.$ Suppose there is a variety $X \subset \Gr_{k', n}$ of dimension less than $k'm$ and for every point $V$ in an open dense subset $\mcu$ of $\Pi_G$, there exists $V' \in X$ such that $V' \subset V$. Then $\mint(G)=0$.
\end{lemma}

\begin{proof}
Assume for the sake of contradiction that $\Pi_G$
has intersection number $d>0$.  Repeating the proof of \cref{prop:int-num-in-terms-of-Gr_mn} using $\mcu$ rather than $T_G$, we obtain that there is an open dense set 
$S\subset \Gr_{m,n}$ such that for 
$\bz\in S$, 
\[\{V \in \mcu : V \subset \bz^\perp\}\]
consists of $d$ points and in particular is nonempty. By construction of $\mcu$, we also have that all $V \in \mcu$ contain an element of $X$. This means that 
\[Y=\{\bz \in \Gr_{m,n}: \text{for some }V' \in X,~ V' \subset \bz^{\perp}\}\]
contains $S$ and so in particular is full-dimensional. But this contradicts \cref{lem:too-low-dim-maps-to-pos-codim}, which states that $Y$ is of positive codimension.
\end{proof}

In the next proposition, we let $S(\OO)$ denote the sources of a perfect orientation.

\begin{proposition}\label{prop:int1criterion} 

Fix a positive integer $m$, and let 
$G$ be a plabic graph of type $(k,n)$ such that 
$\dim \Pi_G=km$. Suppose there is an acyclic perfect orientation $\OO$ of $G$ and a subgraph $G'$ of $G$ obtained by deleting some vertices and edges $\tilde{E}$ of $G$ such that:
\begin{enumerate}
\item \label{item1} $\OO$ restricts to a perfect orientation $\OO'$ of $G'$;
\item \label{item2} the sources and sinks of $\OO'$ are a subset of $\Bb$, and $S(\OO') = S(\OO) \cap \{a, a+1, \dots, b\}$ for some $a<b$; 
\item \label{item3} all vertices of $G'$ incident to edges of $\tilde{E}$ are black;
\item \label{item4} $G'$ has type $(k',n')$ with $1\leq k'<k$, and $\dim \Pi_{G'}<k'm$.
\end{enumerate}
Then $\Pi_G$ has $m$-intersection number 0.
\end{proposition}
\begin{proof}
We would like to use \cref{lem:geom-criterion-int0} on $\Pi_G$. Note that $\Pi_{G'}$ can be embedded into $\Gr_{k', n}$ by adding zero columns in positions $\Bb \setminus S$. We let $X$ denote the image of $\Pi_{G'}$ under this embedding. \cref{item4} implies that $\dim X <k'm$.

We now show that every element of $T_G$ contains an element of $X$. Let $\Bb'$ denote the sources and sinks of $\OO'$, which are the boundary vertices of $G'$.

\cref{item1,item3} imply that if $e \in \tilde{E}$ is incident to $v \in G'$, then $e$ is oriented towards $G'$. This means that if $s$ is a source of $\OO'$, then there are no paths in $\OO$ from $s$ to $\Bb \setminus \Bb'$. So in the path matrix $A(G, \OO)$, the row indexed by $s$ has zeros in columns indexed by $\Bb \setminus \Bb'$. Moreover, this together with \cref{item2} implies the entry in row $s$ and column $t \in \Bb'$ is the same as the entry of $A(G', \OO)$ in row $s$ and column $t$, up to a sign $\sigma_{s,t}$. Because $S(\OO')$ indexes a consecutive subset of the rows of $A(G, \OO)$, $\sigma_{s,t}$ depends only on the column $t$.

This means that $A(G, \OO)$ has a submatrix which, up to resigning columns, has the same span as $A(G', \OO)$. As $\Pi_{G'}$, and thus $X$, is preserved by resigning columns, this implies each element $V \in T_G$ contains some element $V' \in X$.

We have verified the assumptions of \cref{lem:geom-criterion-int0}, and conclude that $\mint(G)=0$.
\end{proof}

Based on computations for $m \leq 8$, we conjecture that the sufficient condition for $\mint(G)=0$ in \cref{prop:int1criterion} is also necessary.

\begin{conjecture}\label{conj:IN0}
Suppose $\Pi_G \subset \Gr_{k,n}$ has dimension $km$. The following are equivalent:
\begin{enumerate}
\item $\Pi_G$ has $m$-intersection number 0;
\item there is a plabic graph $G'$ of type $(k', n')$ with $1 \leq k' <k$ and $\dim \Pi_{G'} < k'm$ such that every $V \in T_G$ contains some $V' \in \Pi_{G'}$;
\item there is a graph move-equivalent to $G$ with a subgraph $G'$ and perfect orientation $\OO$ satisfying the assumptions of \cref{prop:int1criterion}.
\end{enumerate}

\end{conjecture}

\subsection{Relating VRCs to the intersection number}

In this section we start by recalling results of \cite{AGPR} which relate the boundary measurement torus $T_G$ to non-degenerate VRCs. We then show 
in \cref{prop:VRC-vs-Cz=0} that the set of $m$-VRCs on $G$ with boundary 
$\bz$ is in bijection with the set $\{C \in T_G: C \subset \bz^\perp\}$, and 
we show in \cref{cor:int-num=num-VRC} that 
the number of $m$-VRC's with boundary $\bz$ is equal to the 
 intersection number of $\Pi_G$.

Recall from \cref{not:bipartite} that $G^{op}$ is the graph obtained from $G$ by switching the colors of every vertex. Recall from \cref{def:VRC} that $\partial[\bv, \bR]$ denotes the boundary of the VRC $[\bv, \bR]$. 
Also recall the boundary measurement map 
$\mathbb{B}_G$ from \cref{def:pathmatrix}.  

\begin{theorem}[{\cite[Proposition 7.3, Theorem 7.8]{AGPR}}]\label{thm:bdry-meas-vs-bdry-restriction}
    Let $G$ be a reduced $(k,n)$-plabic graph. Choose any Kasteleyn signs $(\sigma_e)_{e \in E(G)}$ on $G$ (see \cite[Section 7.1]{AGPR}). Then we have isomorphisms
    \begin{center}
    \begin{tikzcd}[row sep=tiny]\{[\bv, \bR] \in \VRC_G^{n-k} \text{ non-degenerate}\} \arrow[r,"\sim"', "\partial"] & T_{G^{op}}  & \arrow[l,"\sim","\mathbb{B}_{G^{op}}"'](\CC^*)^{E(G)}/\text{gauge}\\
   {[\{v_b\}, \{r_e\}]}=[\bv, \bR] \arrow[r,maps to] &\partial[\bv, \bR] &\arrow[l, maps to] (\sigma_e r_e)_{e \in E(G)}.
    \end{tikzcd}
    \end{center}
\end{theorem}

\cref{thm:bdry-meas-vs-bdry-restriction} implies that for any choice of $\bR$ (modulo gauge), there is a unique non-degenerate $[\bv, \bR] \in \VRC_G^{n-k}$ and its boundary is in $T_{G^{op}}$.

In the next result, we use the notation $\mzVRC_G := \{[\bv, \bR] \in \mVRC_G: \partial[\bv, \bR] = \bz \}$, and the notation $V^\perp$ for the orthogonal complement of the vector space $V$.

\begin{theorem}\label{prop:VRC-vs-Cz=0}
    Let $G$ be a reduced plabic graph of type $(k,n)$.
    
    Fix $m < n-k$ and $\bz \in \Gr_{m,n}^{\circ}$. 
    Then there are well-defined
    maps $\alpha$ and $\perp$ which give bijections
    in the following diagram: 
    \begin{center}
    \begin{tikzcd}[row sep=tiny]
    \mzVRC_G\arrow[r,"\alpha", "\sim"'] \arrow[bend left=20]{rr}{\m}[swap]{\sim} & \{ V \in T_{G^{op}}: \bz \subset V\}\arrow[r, "\perp","\sim"'] & \{C \in T_G: C \subset \bz^\perp\}\\
    {[\bv, \bR]} \arrow[r, mapsto]& {\partial[\bv', \bR]} = V \arrow[r, maps to] & V^{\perp}
    \end{tikzcd}
    \end{center}
    In particular, $\alpha$ sends $[\bv, \bR]$ to the boundary of the unique non-degenerate $[\bv' ,\bR'] \in \VRC_G^{n-k}$ with $\bR'=\bR$.
\end{theorem}

\begin{proof}
We focus first on the map
\[\alpha: \mzVRC_G \to \{ V \in T_{G^{op}}: \bz \subset V\}. \]
Let $[\bv, \bR] \in \mzVRC_G$. By \cref{thm:bdry-meas-vs-bdry-restriction}, there is a unique nondegenerate $[\bv', \bR'] \in \VRC_G^{n-k}$ with $\bR'=\bR$. The map $\alpha$ sends $[\bv, \bR] \mapsto V=\partial[\bv', \bR]$. \cref{thm:bdry-meas-vs-bdry-restriction} shows $V \in T_{G^{op}}$, so to show $\alpha$ is well-defined, we need to show $\bz \subset V$.

Fix a representative $A$ of $\bz$ and fix a representative $(\bv, \bR)$ of $[\bv, \bR]$ with boundary $A$. Fix a representative $(\bv', \bR)$ of $[\bv', \bR]$ with exactly the same relations as $(\bv, \bR)$. Choose $I \in \binom{[n]}{n-k}$ such that $G$ has an acyclic reverse perfect orientation with sources $I$. By assumption, the $m \times (n-k)$ matrix $[z_{i_1} \dots z_{i_{n-k}}]$ is full rank, where $I=\{i_1, \dots, i_{n-k}\}$. So we may extend $[z_{i_1} \dots z_{i_{n-k}}]$ to a full-rank $(n-k) \times (n-k)$ matrix $M'$, by adding some rows to the bottom. Since $V \in T_{G^{op}}$, $\lr{I}_V \neq 0$; in other words, the vectors $v'_{i_1}, \dots, v'_{i_{n-k}}$ are a basis of $\CC^{n-k}$. Using the $GL_m$ action if necessary, we may assume that $M=[v'_{i_1} \dots v'_{i_{n-k}}]$.

Because $G$ is reduced and has an acyclic reverse perfect orientation $\OO$ with source set $I$,  
it satisfies the assumptions of \cref{lem:v_b-from-paths}. Applying \eqref{eq:v_b-from-paths} to both $(\bv', \bR)$ and $(\bv, \bR)$, we obtain that, for $j \notin I$, 
\begin{equation} \label{eq:path}
z_j = \sum_{i \in I} \left( \sum_{\substack{ P: i \to j \\ \text{ path in } \OO}} \frac{\prod_e r_e}{\prod_{e'} (-r_{e'})}\right) z_i \qquad \text{and} \qquad v_j' =\sum_{i \in I} \left( \sum_{\substack{ P: i \to j \\ \text{ path in } \OO}} \frac{\prod_e r_e}{\prod_{e'} (-r_{e'})}\right) v'_i.
\end{equation}
Projecting $v'_i$ onto the first $m$ coordinates gives $z_i$ for $i \in I$ by construction. So this formula shows that projecting $v_j'$ onto the first $m$ coordinates gives $z_j$. That is, (a representative for) $\bz$ is the first $m$ rows of (a representative for) $V$, and $\bz \subset V$ as desired.

The inverse map $\beta$ of $\alpha$ is defined as follows. For $V \in T_{G^{op}}$ with $\bz \subset V$, let $(\bv', \bR')$ be a representative of the unique $ [\bv', \bR']\in \VRC_G^{\partial =V}$. Use the $GL_m$-action to make the top $m$ rows of $\partial(\bv', \bR')$ into a representative of $\bz$; this is possible because by assumption, $\bz \subset V$. Then we define $(\bv, \bR)$ by setting $\bR:=\bR'$, and setting $v_b$ to be the projection of $v_b'$ to the first $m$ coordinates. We set $\beta(V) = [\bv, \bR] \in \mVRC_G$.

The argument that $\alpha$ is well-defined also shows that $\beta \circ \alpha$ is the identity. We note that $V$ and $\alpha \circ \beta (V)$ are boundaries of $(n-k)$-VRCs with identical relations $\bR'$. \cref{thm:bdry-meas-vs-bdry-restriction} implies that $\bR'$ determines the boundary, so $V=\alpha \circ \beta (V)$.

We now turn to the map
\begin{equation}\label{eq:orthogonal}
\{ V \in T_{G^{op}}: \bz \subset V\} \xrightarrow{\sim} \{C \in T_G: C \subset \bz^\perp\}
\end{equation}
which sends $V$ to its orthogonal complement $V^\perp=:C$. Clearly $\bz \subset V$ if and only if $C \subset \bz^\perp$, so as long as orthogonal complement takes $T_{G^{op}}$ to $T_G$ and vice-versa, the map is well-defined and bijective.

Let $N$ be the $n \times n$ diagonal matrix with diagonal $1,-1, \dots, (-1)^{n-1}$. The map $V \mapsto VN$ is an involution on $T_{G^{op}}$: given a VRC on $G$ with boundary $V$, one can produce a VRC with boundary $VN$ by changing the sign of $r_e$ for $e$ adjacent to boundary vertices $2, 4, 6,$ etc. 
\cref{lem:bdry-meas-G-vs-Gop} shows that the map $V \mapsto (VN)^{\perp}$ is a bijection between $T_{G^{op}}$ and $T_G$. 
 Composing these maps shows that $V \mapsto VN \mapsto V^\perp$ is a bijection from $T_{G^{op}}$ to $T_G$.
\end{proof}

\begin{remark}\label{rem:VRC-vz-CZ=0-weaker-assumption}
   The proof of \cref{prop:VRC-vs-Cz=0} uses only that for some $I \in \binom{[n]}{n-k}$ such that $G$ has an acyclic reverse perfect orientation with sources $I$, the submatrix of $\bz$ in columns $I$ is full rank. Thus, \cref{prop:VRC-vs-Cz=0} holds for $\bz$ in a slightly larger open set than $\Gr_{m,n}^\circ$.
\end{remark}

 \begin{corollary}\label{cor:int-num=num-VRC} Suppose $\Pi_G \subset \Gr_{k,n}$ has dimension $km$. Then for generic $\bz \in \Gr_{m,n}$,
 \[|\mzVRC_G| = \mint(G),\]
 that is, the number of $m$-VRC's with boundary $\bz$ is equal to the 
 intersection number of $\Pi_G$. 
 
 In particular, for generic $\bz$, $|\mzVRC_G|$ is finite, does not depend on $\bz$, and depends only on the move-equivalence class of $G$.
 \end{corollary}
 \begin{proof}
     Let $\mcu \subset \Gr_{m,n}$ be the (open dense) subset guaranteed by \cref{prop:int-num-in-terms-of-Gr_mn} where $\{C \in T_G: C \subset \bz^\perp\}$ has size $\mint(G)$. 
     
For $\bz \in \Gr_{m,n}^\circ \cap \mcu$, the assumptions of \cref{prop:VRC-vs-Cz=0} are satisfied. So \cref{prop:VRC-vs-Cz=0} and the choice of $\mcu$ imply that 
     \[|\mzVRC_G| = |\{C \in T_G: C \subset \bz^\perp\}| = \mint(G).\]

     For the second sentence, $|\mzVRC_G|$ for generic $\bz$ depends only on the move-equivalence class of $G$ because $\Pi_G$ depends only on the move-equivalence class of $G$, and $\mint(G)$ is defined in terms of $\Pi_G$. The other statements are clear.
 \end{proof}

Recall that $G$ is $m$-generically solvable if $|\mzVRC_G|=1$ for generic $\bz \in \Gr_{m,n}$. The next proposition shows that a graph $G$ is $m$-generically solvable for at most one $m$.

\begin{proposition}\label{prop:solvable-only-one-m-possible} Suppose $G$ is reduced and of type $(k,n)$. If for generic $\bz \in \Gr_{m,n}$, $\mzVRC_G$ is nonempty and finite, then $$m=(\#(\text{faces of }G) -1)/(|W|- |\Bi|) = (1/k) \dim \Pi_G.$$

In particular, if $G$ is $m$-generically solvable, then $m= \dim \Pi_G /k$.
\end{proposition}
\begin{proof}
The first statement implies the second, since $G$ is $m$-generically solvable if $|\mzVRC_G|=1$ for generic $\bz \in \Gr_{m,n}$.

We prove the contrapositive of the first statement: if $\dim \Pi_G \neq km$, then for generic $\bz$, $\mzVRC_G$ is either empty or infinite.

If $\dim \Pi_G <km$, then \cref{lem:too-low-dim-maps-to-pos-codim} implies that for generic $\bz \in \Gr_{m,n}$,
\[\{C \in T_G: C \subset \bz^\perp\} \subset \{C \in \Pi_G: C \subset \bz^\perp\} = \emptyset.\]
By \cref{prop:VRC-vs-Cz=0}, the leftmost set is also equal to $\mzVRC_G$.

Now suppose $\dim \Pi_G >km$. We may assume that for generic $\bz \in \Gr_{m,n}$,
\[\{C \in \Pi_G: C \subset \bz^\perp\} \neq \emptyset,\]
since otherwise we are done by the reasoning of the previous paragraph. This means that the projection map $\pi$ from the incidence variety
\[X= \{(C, \bz) \in \Pi_G \times \Gr_{m,n}: C \subset \bz^\perp \}\]
to $\Gr_{m,n}$ is dominant. For generic $\bz \in \Gr_{m,n}$, the dimension of the fiber $\pi^{-1}(\bz)$ is $\dim X-\dim \Gr_{m,n}$. Using the computation of $\dim X$ from \cref{lem:too-low-dim-maps-to-pos-codim} and the assumption that $\dim \Pi_G >km$, we conclude that generic fibers $\pi^{-1}(\bz)$ have positive dimension. Note that $T_G$ is obtained by removing some divisors from $\Pi_G$ and so is open and dense. Similarly, $\tilde{T}_G:=T_G \times \Gr_{m,n}$ is open and dense in $X$. By a standard transversality argument, this implies that for generic $\bz \in \Gr_{m,n}$, the intersection $\pi^{-1}(\bz) \cap \tilde{T}_G$ is of positive dimension.
 Now, the projection map from $\pi^{-1}(\bz) \cap \tilde{T}_G$ to $\Pi_G$ is a bijection, with image exactly 
\[\{C \in T_G: C \subset \bz^\perp\}\]
which for generic $\bz$ is in bijection with $\mzVRC_G$ by \cref{prop:VRC-vs-Cz=0}. So for generic $\bz$, there are infinitely many $m$-VRCs with boundary $\bz$.
\end{proof}

\cref{cor:int-num=num-VRC,prop:solvable-only-one-m-possible} together imply the following corollary.

\begin{corollary}\label{cor:int1}
A plabic graph $G$ of type $(k,n)$ is $m$-generically solvable if and only if 
$\Pi_G$ has dimension $km$ and $m$-intersection number $1$.
\end{corollary}

\cref{cor:int-num=num-VRC} shows that if $\Pi_G \subset \Grk$ has dimension $km$, then $m$-VRCs with generic boundary are quite well-behaved. One may also obtain the next two corollaries, which are along similar lines. They show that in this scenario, square moves generically induce a bijection on VRCs and most VRCs are nondegenerate. 

\begin{corollary}\label{cor:vrc-move-invariant} Suppose $G$ is reduced and $\Pi_G \subset \Gr_{k,n}$ has dimension $km$. If $G$ and $G'$ are related by a square move, then for generic $\bz \in \Gr_{m,n}$, the map in \cref{fig:sq-move} gives a bijection between $\mzVRC_G$ and $\mzVRC_{G'}$. 
\end{corollary}

 \begin{corollary}\label{cor:VRC-generically-nondegen} Suppose $G$ is a reduced $(k,n)$-plabic graph and $\Pi_G$ has dimension $km$. Then for generic $\bz \in \Gr_{m,n}$, the set $\mzVRC_G$ does not contain any degenerate VRCs.
 \end{corollary}

Recall from \cref{def:pinning} the notion of a \emph{pinning} of $G$. We conclude this section by showing pinnings exist for solvable $G$.

 \begin{proposition}\label{prop:int-num-1-ratl-coeffs}
     Suppose $G$ is generically $m$-solvable, or equivalently suppose $\Pi_G \subset \Gr_{k,n}$ has dimension $km$ and $\mint(G)=1$. Denote by $\zVRC$ the unique $m$-VRC with boundary $\bz \in \Gr_{m,n}$, where $\bz$ is generic. Then there are rational functions $\{f_e\}_{e \in E(G)} \subset \CC(\widehat{\Gr}_{m,n})$ such that for generic $\bz$, $\zVRC$ is represented by $(\{v_b(\bz)\}, \{f_e(\bz)\})$, where the vectors $v_b(\bz)$ are defined using \eqref{eq:v_b-from-paths}. In other words, a pinning of $G$ exists.
 \end{proposition}

 \begin{proof}
Recall that by \cref{prop:VRC-vs-Cz=0} and \cref{prop:int-num-in-terms-of-Gr_mn}, for $\bz$ in some open dense set $\mcu_1 \subseteq \Gr_{m,n}$,
\[|\mzVRC_G| = |\{V \in T_{G^{op}} : \bz \subset V\}| = |\{C \in T_G : C \subset \bz^\perp\}|=|\{C \in \Pi_G : C \subset \bz^\perp\}|=1. \]
 
Consider the incidence variety
 \[X=\{(\bz, V) \in \Gr_{m,n}\times\Pi_{G^{op}}: \bz \subset V\} \subset \Gr_{m,n} \times \Gr_{n-k, n}.\]
The projection $X \to \Gr_{m,n}$ onto the first factor is a regular map. By the previous paragraph, this projection is dominant and points in $\mcu_1$ have a unique preimage. This implies the projection map is birational, so there is a rational map $\hat{F}: \Gr_{m,n} \to X$ sending $\bz \mapsto (\bz, F(\bz))$ which is the inverse of the projection map on an open dense set $\mcu \subseteq \Gr_{m,n}$. We may assume $\mcu \subset \mcu_1$. The Pl\"ucker coordinates of $F(\bz)$ are rational functions in the Pl\"ucker coordinates of $\bz$. By the choice of $\mcu$, $F(\bz) \in T_{G^{op}}$.

For $\bz \in \mcu$, let $[\bv, \bR]$ be the unique element of $\mzVRC_G$. By \cref{prop:VRC-vs-Cz=0}, $F(\bz) = \alpha[\bv,\bR]$. That is, $F(\bz)$ is the boundary of the $(n-k)$-VRC $[\bv', \bR]$. \cite[Proposition 5.9 and Theorem 7.1]{MullerSpeyerTwist} together with \cref{thm:bdry-meas-vs-bdry-restriction} gives an explicit formula for a rational map $T_{G^{op}} \to (\CC^*)^{E(G)}$ sending $V=\partial[\bv'', \bR'']$ to $\widehat{\bR''}$, where $\widehat{\bR''}$ is a particular distinguished representative of the gauge-equivalence class of $\bR''$.

So the map $\Gr_{m,n} \to (\CC^*)^{E(G)}$ sending
\[\bz \mapsto F(\bz)=\partial[\bv', \bR] \mapsto \widehat{\bR}\]
is a composition of rational maps and thus is rational. This shows the desired functions $f_e$ exist, that is, that there is an $m$-VRC with boundary $\bz$ with $r_e = f_e(\bz)$. \cref{lem:v_b-from-paths} implies that for this $m$-VRC, the vectors are given by \eqref{eq:v_b-from-paths}.

\end{proof}

\begin{remark}\label{rem:inverting-amplituhedron-map}
    While \cref{prop:int-num-1-ratl-coeffs} is not constructive, in practice one may be able to compute the rational functions $f_e$ using e.g. the Grassmann-Cayley algebra, as we do in \cref{def:coeffs-in-VRC-from-GC} when $G$ is a tree.
    Once one has $\{f_e\}_{e \in E(G)}$, one can invert $\comp_{Z,G}$ as follows. First, choose Kasteleyn signs on $G$ (cf. \cite[Section 7.1]{AGPR}), and then change the sign on the boundary edges $2, 4, 6, $ etc. Denote these modified Kasteleyn signs by $\{\sigma_e\}_{e \in E(G)}$. Then for generic $\bz \in \Gr_m(W)$, the preimage $C=\comp_{Z,G}^{-1}(\bz)$ equals $\mathbb{B}_G(\{\sigma_e f_e(z)\})$, where $\mathbb{B}_G$ is the boundary measurement map. That is, $C$ is the rowspan of the path matrix $A(G, \OO)$ using edge weights $\{\sigma_e f_e(\bz)\}$ (see \cref{def:pathmatrix}), and its Pl\"ucker coordinates are given by \cref{prop:bdry-meas-pluckers}.
\end{remark}

\begin{remark} The results in this subsection, connecting intersection number and $m$-VRCs, were in part inspired by various examples in the physics literature of inverting the amplituhedron map by `solving $C \cdot \bz =0$'  \cite{Mago:2020kmp,He:2020uhb}.
For more on the relationship to the `$C \cdot \bz=0$' equations, see \cite[Remark 7.13]{even2023cluster}.
\end{remark}

\subsection{Proof that dominant and brushable are equivalent}\label{ssec:proof-dominant-solvable}
 In this section, we use the results relating VRCs and elements of $T_G$ to prove \cref{prop:dominant_solvable}, which states that a solvable tangle is dominant if and only if it admits a brushing. 
 
 We will need the following combinatorial lemmas. Recall from \cref{def:gauge-fix} the definition of a gauge-fix of a plabic graph, and the characterization of gauge-fixes from \cref{prop:gauge-fix-characterization}.

\begin{lemma}\label{lem:gauge_fix_NI_paths}
Let $G$ be a reduced plabic graph 
and $\{P_j\}_{j \in J}$ a collection of vertex-disjoint paths where $P_j$ starts at boundary vertex $j$ and ends at some vertex $v_j$. Then there is a gauge-fix $E(G)= E_1 \sqcup E_{\neq 1}$ such that
\begin{itemize}
\item the edges of the paths $\cup_j P_j$ are contained in $E_1$;
\item the only path from a boundary vertex to $v_j$ which avoids $E_{\neq 1}$ is $P_j$.
\end{itemize}
\end{lemma}

\begin{proof}
Let $G'$ be the graph $G$ together with edges $e_i:=(i,i+1)$ for $i=1, \dots, n-1$. Note that $G'$ is connected, since all connected components of $G$ contain a boundary vertex. There is a spanning tree $T$ of $G'$ which contains $\cup_{j \in J} P_j$ and $\cup_{i \in [n-1]} e_i$. Removing the edges $e_i$ from $T$, we obtain a spanning forest $T'$ of $G$ which contains $\cup_{j \in J} P_j$. Each connected component $C$ of $T'$ contains exactly one boundary vertex: if $C$ contained both $i$ and $j$, then $T$ would contain a cycle; if $C$ contains no boundary vertex, then this implies $T$ has more than one connected component. By \cref{prop:gauge-fix-characterization}, if we set $E_1$ to be the edges of $T'$ and $E_{\neq 1}:= E(G) \setminus E_1$, then $E_1 \sqcup E_{\neq 1}$ is a gauge-fix.

By construction $\cup_j P_j$ is contained in $E_1$. This implies $j$ and $v_j$ are in the same connected component of $T'$. Since $T'$ is a forest, any path from $j$ to $v_j$ which is not $P_j$ passes through an edge in $E_{\neq 1}$. Consider another boundary vertex $i$. This lies in another connected component of $T'$, so every path from $i$ to $v_j$ passes through an edge in $E_{\neq 1}$. This shows the constructed gauge-fix has the desired properties.
\end{proof}

\begin{proof}[Proof of \cref{prop:dominant_solvable}] 

\noindent$(\Rightarrow):$ We show that if $(G, \bD)$ admits a brushing, then it is dominant. Fix a brushing $\mcb$, and a blob $D \in \bD$. Let $\OO:=\OO^D$ be the corresponding acyclic reverse perfect orientation and $\{P_i\}_{i \in I}$ the collection of paths. We index $P_i$ by its starting point $i \in \Bb$ rather than its ending point $u_i \in D$ for convenience. Since 
 every boundary vertex of $D$ is an endpoint of a unique $P_i$, there is a natural bijection from $I$ to $D$. To ease notation, we will use $I$ as an index set throughout rather than $D$.

Choose a gauge-fix $E(G)= E_1 \cup E_{\neq 1}$ having the properties of \cref{lem:gauge_fix_NI_paths}. For a VRC $[\bv, \bR] \in \mVRC_G$, we will choose the unique representative $(\bv, \bR)$ where $r_e =1$ for $e \in E_1$ (uniqueness is because $E_1 \cup E_{\neq 1}$ is a gauge-fix).
To ease notation, we denote the vector on $b_{u_i}$ by $w_i$ and denote by $\mathbf{w}$ the element of $\Gr_{m,I}$ with columns $w_i$. 
 This element depends on $(\bv, \bR)$, though we omit this dependence from the notation. A priori, $\bw$ may not be full-rank; however, its maximal minors are rational functions in the Pl\"ucker coordinates of $\bz$, and the proof below shows that they are not identically zero, so for generic $\bz$, $\bw$ is indeed an element of $\Gr_{m, I}$.

It follows from \cref{prop:VRC-vs-Cz=0} if the set
\[\{\mathbf{w}: \exists V\in T_{G^{op}}~\text{with}~\partial[\bv, \bR]\subset V\} \]
is dense in $\Gr_{m,I}$, then $(G,\bD)$ is dominant\footnote{The definition of dominance involves density in $\Conf^{\circ}_{m,I}$, but this is implied by density in $\Gr_{m,I}$.}. 

Fix $\mathbf{p}=(p_i)_{i\in I}\in\Gr_{m,I}$. It is enough to find $V\in T_{G^{op}}$ and $[\bv, \bR] \in \mVRC_G$ such that $\partial[\bv, \bR]\subset V$ and $\mathbf{w}$ is arbitrarily close to $\mathbf{p}.$ 

We will show that we can find $V\in T_{G^{op}}$ and $\bz\subset V$ such that $z_{i}$ is arbitrarily close to $p_i$ for all $i \in I$. We then use \cref{prop:VRC-vs-Cz=0} and the gauge-fix to obtain a VRC $(\bv, \bR)$ with boundary $\bz$, and show that in this scenario, $z_{i}$ is arbitrarily close to $w_i$.

Let $G'$ be the plabic graph obtained from $G^{op}$ by erasing the edges of $E_{\neq 1}$. Write $E_{\neq 1}^+$ for the edges of $E_{\neq 1}$ oriented black-to-white, and $E_{\neq 1}^-$ for those oriented white-to-black. Note that $T_{G'}$ is in the boundary of $T_{G^{op}}$ and consists of a single point $V'$. Indeed, points in $T_{G^{op}}$ have a unique representative $A$ with the identity matrix in columns indexed by the source set $J$ of $\OO$, and are the boundary of a unique $(n-k)$-VRC on $G$ with $r_e=1$ for $e \in E_1$. Then $T_{G'}$ is obtained by letting $r_e\to 0$ for $e\in E_{\neq 1}^+$ and $r_e\to\infty$ for $e\in E_{\neq 1}^-$. Every path between boundary vertices in $\OO$ involves an edge of $E_{\neq 1}$, so the boundary of the resulting VRC is the matrix $M$ with the identity in columns $J$ and zero columns elsewhere. The point $V'$ is represented by $M$ and has a single nonzero Pl\"ucker coordinate, $\lr{J}$.

Let $\mathbf{p}' \in \Gr_{m,n}$ be obtained from $\mathbf{p}$ by adding zero columns in indices $[n] \setminus I$. Since $I \subset J$, $V'$ also has a representative matrix $M'=gM$ whose top $m$ rows are (a representative for) $\mathbf{p}'$.

Define $B=gA$. As $r_e\to 0$ for $e\in E_{\neq 1}^+$ and $r_e\to\infty$ for $e\in E_{\neq 1}^-$, $A \to M$ and so $B \to M'$. If we take the top $m$ rows of $B$, we obtain an element $\bz \in \Gr_{m,n}$ which is contained in $B$ and which approaches $\mathbf{p}'$. In particular, $z_i \to p_i$, as desired. 

Let $(\bv, \bR)$ be the $m$-VRC with boundary $\bz$ and coefficients $r_e$ equal to those of the VRC with boundary $B$ (in particular, with $r_e=1$ for $e \in E_1$). This VRC exists by \cref{prop:VRC-vs-Cz=0}. If we express $w_i$ using \eqref{eq:v_b-from-paths}, we have
\[w_u = z_{i} + \sum_{\substack{{P: j \to u}\\{P \neq P_i}}} \wt(P) z_{j}\]
where the sum is over paths $P \neq P_i$ in $\OO^D$ from a boundary vertex to $u$. By \cref{lem:gauge_fix_NI_paths}, for all terms in the right sum, $\wt(P) \neq 1$. In particular, $\wt(P)$ has in the numerator $r_e$ for some $e \in E_{\neq 1}^+$, or has in the denominator $r_e$ for some $e \in E_{\neq 1}^+$. So as $r_e$ become arbitrarily small (or arbitrarily large, as appropriate), $w_i$ approaches $z_{i}$, which in turn approaches $p_i$, as desired.

\noindent$(\Leftarrow):$ We now show that if $(G, \bD)$ is solvable and does not admit a brushing, then it is not dominant. 

Suppose for $D \in \bD$, there is no acyclic reverse perfect orientation with vertex-disjoint paths from boundary vertices to $D$. Let $\OO$ be any acyclic reverse perfect orientation. A max-flow/min-cut argument shows that there is a number $l<|D|$ such that
\begin{itemize}
\item there is a vertex-disjoint path collection from $\Bb$ to $D$ consisting of  $l$ paths. Say the set of sources in these paths is $I \subset \Bb$.
\item there is a set of $l$ vertices $X$ such that all vertex-disjoint path collections from $\Bb$ to $D$ pass through $X$. We may assume $X=\{b_1, \dots, b_l\}$ consists of black vertices.
\end{itemize}

Let $(G',D)$ be the subtangle obtained by restricting to vertices and edges that can be reached from $X$ via a directed path in $\OO$; note this includes all boundary vectors $b_u$ of the blob $D$. Let $\Bb'$ denote the boundary vertices of $G'$. The orientation $\OO$ restricts to an acyclic reverse perfect orientation of $G'$, with source set exactly $X.$

The restriction of a VRC of $G$ to $G'$ is a VRC for $G'$, by definition. 
Moreover, the restriction of the VRC $(\bv, \bR)$ to $G'$ depends only on the vectors $\{v_b\}_{b \in \Bb'}$, again by definition of the VRC. 
Alternately, using \cref{lem:v_b-from-paths}, the restriction to $G'$ depends only on $\bR|_{G'}$ and $\{v_b\}_{b \in X}$.

Additionally, for $\bz$ coming from a dense open set, the vectors $(v_b(\bz))_{b \in \Bb'}$ in the VRC $\zVRC$ are the boundary of a unique VRC on $G'$. Indeed, otherwise we could change the VRC on $G'$ to another without changing $(v_b(\bz))_{b \in \Bb'}$, and obtain another VRC on $G$ with boundary $\bz$, contradicting solvability of $G$. So the tangle $(G', D)$ gives a map $f: \Conf^\circ_{m,\Bb(G')} \dashrightarrow \Conf_{m,D}$.

Fix now a pinning of $(G, \bD)$. There is some gauge fix such that vertex-disjoint paths from $I$ to $X$ are in $E_1$; we may assume that the coefficients on these edges are 1. Note that by the argument of the `if' direction above, the set of obtainable $\bv_X:=(v_b(\bz))_{b \in X}$ is dense in $\Gr_{m,l}.$ Thus, we can write the blob's boundary vectors $\{v_{b_u}(\bz)\}_{u \in D}$ as functions on a dense subset of $\Gr_{m,l}.$ 
 Write $\mathbf{w}(\bz)=(w_u)_{u \in D}$ for the blob vectors as functions of $\bz.$ Then $\mathbf{w}(\bz)$ is a rational function on $\bz,$ by \cref{prop:int-num-1-ratl-coeffs}. On the other hand, by the above discussion, we can also write $\mathbf{w}$ as a function of $\bv_X,$ that is, we can write
\[F(\bv_X(\bz))=\mathbf{w}(\bz),\]
where $F$ is defined on a dense subset of $\Gr_{m,l}$, and also $\bv_X(\bz)$ is a rational function of $\bz.$ Note that $F$ is a lift of $f$ to the Grassmannian.

It suffices to show that $F$ is a rational map from $\Gr_{m,l}$ to $\Gr_{m, D}$, and hence $f$ is a rational map from $\Conf^\circ_{m,l}$ to $\Conf_{m,D}$. Since $l<|D|$, $f$ cannot be dominant. This implies that the blob boundary lines $\{\CC\mathbf{w}(\bz): \bz \in \Gr_{m,n}\}$ are not dense in $\Conf^\circ_{m,D}.$

We now show that $F$ is rational. Let $\mcu\subseteq\Gr_{m,n}$ be an open dense set on which there is a unique and non-degenerate VRC with boundary $\bz$, which exists by \cref{prop:int-num-1-ratl-coeffs} and \cref{cor:VRC-generically-nondegen}. Let $S$ be 
 the (Zariski dense) image of the map $\mcu \to \Gr_{m,l}$ which sends $\bz \mapsto \bv_X(\bz)$. Note that when $\bz\in \mcu,$ $\mathbf{w}(\bz)$ is well-defined, hence for $\mathbf{x}\in S$, $F(\mathbf{x})$ is well-defined. For $\mathbf{x}\in S$, write $\mathbf{x}'\in\Gr_{m,n}$ as the vector space obtained from $\mathbf{x}$ by adding zero columns in columns $N \setminus I$. 
As in the `if' direction, we can take the limit $\bz\to\mathbf{x}'$ by letting $r_e\to 0$ ($r_e\to\infty$) for black-to-white (white-to-black) edges of $G\setminus G'$ with $r_e\nequiv 1$. Then
\[\lim_{\mcu \ni\bz\to\mathbf{x}'}\bv_X (\bz)=\mathbf{x}.\] Also we have that
$\lim_{\bz\to\mathbf{x}'}\mathbf{w}(\bz)=F(\mathbf{x})$. Indeed, $\mathbf{x}\in S,$ hence $F(\mathbf{x})$ is well-defined and equals $\mathbf{w}(\bz)$ for every $\bz\in \mcu$ for which $\bv_X(\bz)=\mathbf{x}.$
Thus, the rational function $\mathbf{w}(-)$ extends to $\mathbf{x}'$ and we can write
\[F(\mathbf{x})=\mathbf{w}(\mathbf{x}').\]
Since $\mathbf{w}$ is rational, this shows that $F$ is a rational mapping from $S,$ hence from $\Gr_{m,l}$, to $\Gr_{m,D}.$
\end{proof}

\section{Characterizing intersection number one trees}\label{sec:trees}

In this section we focus on plabic trees.  In particular,
we will show in \cref{prop:tree1} that if $G$ is a plabic tree, it is solvable (equivalently, $\dim \Pi_G = km$ and $\mint(G)=1$) if and only if it is \emph{$m$-balanced}, see \cref{def:mbalanced}.
We also provide an explicit $m$-VRC on amplitrees with fixed boundary vectors via the Grassmann-Cayley algebra (\cref{prop:amplitree-VRC-rep-up-to-sign}).

We will assume throughout this section that our plabic trees are 
bipartite, i.e. $G=((B,W), E)$, with $B=B_{int} \sqcup B_{bd}$
as in \cref{not:bipartite}.

\subsection{Plabic trees and the {$m$}-balanced property}

In this section we introduce the notion of \emph{$m$-balanced} plabic trees. We will show in 
\cref{cor:int0} that if a plabic tree $G$ of type $(k,km+1)$
fails to be $m$-balanced, then 
$\mint(G)=0$.

For convenience, we record in \cref{lem:facts} some well-known (and easy to prove) facts about plabic trees that we will use later.
\begin{lemma}\label{lem:facts}
Let $G$ be a plabic tree of type $(k,n)$.
   \begin{enumerate}
       \item Then $G$ is reduced,  and $\dim \Pi_G = n-1$.
       \item We have 
       $$k=1+\sum_{b\in \Bi} (\deg b-2).$$
       \item We can apply moves to $G$ to make $G$ bipartite with trivalent internal black vertices.
       \item We have that 
       $k=|W|-|\Bi|$. 
       \item If $G$ is bipartite with black vertices trivalent,  then $k=|\Bi|+1$.
   \end{enumerate} 
\end{lemma}

\begin{definition}\label{def:mbalanced}
Let $G$ be a bipartite plabic tree of type $(k,km+1)$.
We say that $G$ is \emph{$m$-balanced} if 
for each edge $e$ of $G$, if we write 
$G\setminus \{e\} = G_1 \sqcup G_2$ (giving ``half" the edge $e$ to each $G_i$), then for each $i=1,2$, we have
$$m(k_i-1) < \dim \Pi_{G_i} \leq mk_i,$$ where $k_i$ is the $k$-statistic of the tree $G_i$. 
A \emph{$(k,m)$-amplitree} is a bipartite plabic tree $G$ 
of type $(k,km+1)$ which is $m$-balanced.
\end{definition}

\begin{remark}
We will show in \cref{prop:amplitree-unique-VRC} that `amplitrees' have intersection number one and, by \cref{rk:physics_1}, they correspond to a new infinite class of rational Yangian invariants for scattering amplitudes in $\mathcal{N}=4$ SYM (hence the use of `ampli' in `amplitree').
\end{remark}

We start by proving a characterization
of the $m$-balanced property plus some helpful lemmas.

\begin{lemma}\label{prop:dimsig}
    Let $G$ be a bipartite plabic tree of type $(k,km+1)$. Then the $m$-balanced property says that 
for each $i=1,2$, we have
\begin{equation}\label{eq:mbalanced}
    1 \leq |\Bb \cap G_i| 
     - m \sum_{b\in G_i \cap \Bi} (\deg b-2) \leq m.
     \end{equation}
\end{lemma}
\begin{proof}
    Let us rewrite \eqref{eq:mbalanced} as 
    \begin{equation*} 
     1+ m \sum_{b\in G_i \cap \Bi} (\deg b-2) \leq 
     |\Bb \cap G_i| \leq 
     m+ m \sum_{b\in G_i \cap \Bi} (\deg b-2).
    \end{equation*}
    Using \cref{lem:facts} (1) and (2), we then get 
        \begin{equation*} 
     1+ m (k_i-1) \leq 
     \dim  \Pi_{G_i} \leq 
     m k_i.
    \end{equation*}
    The result follows.
\end{proof}

The following observations are easy to verify.
\begin{lemma}
Use the notation of \cref{def:mbalanced}.
\begin{itemize}
\item The $m$-balanced property is preserved under bipartite expand and contract moves. 
\item If \eqref{eq:mbalanced} holds for $G_1$,
then it also holds for $G_2.$
\end{itemize}
\end{lemma}

The next proposition shows the first connection between the $m$-balanced property and intersection number.

\begin{proposition}\label{prop:fail}
Let $G$ be a bipartite plabic tree of type 
$(k,km+1)$.  Suppose that the $m$-balanced property fails.
Then we can find a subtree $G'$ of $G$ and 
a perfect orientation $\OO$ of $G$ 
which satisfy 
the conditions of \cref{prop:int1criterion}.
\end{proposition}

\begin{proof} 
By \cref{lem:facts} we can assume without loss of generality
that the internal black vertices of $G$ are trivalent.
Since the $m$-balanced property fails, we know that 
there is an edge $e$ such that $G\setminus \{e\} = G_1 \sqcup G_2$
such that without loss of generality,
$\dim \Pi_{G_1} \leq m(k_1-1)$ (and correspondingly
$\dim \Pi_{G_2} > mk_2$). Let $v_i$ denote the vertex of $e$ in $G_i$.

Choose a perfect orientation $\OO$ of $G$.
We first claim that without loss of generality, we can 
assume that $e$ is directed from $G_2$ to $G_1$.  To see this,
note that each edge $e$ is part of (at least one) directed
path from a boundary vertex of $G$ to another boundary vertex.
(This is because no internal vertex of $G$ is a source or 
sink, so we will never get ``stuck'' when we try to produce such
a path.)  Thus, if $e$ is not directed from $G_2$ to $G_1$
in $\OO$, we can choose a directed path $p$ passing through $e$
which goes from one boundary vertex to another, and reverse all 
arrows along $p$, obtaining a new perfect orientation in which 
$e$ is oriented from $G_2$ to $G_1$.  

Note that $G_1$ must contain at least two boundary vertices from $\Bb$.  Otherwise $\dim \Pi_{G_1}=1$, $k_1=1$,
and it's not true that $\dim \Pi_{G_1} \leq m(k_1-1).$
Let $S$ be the set of sources of $\OO$ which lie in 
$\Bb \cap G_1$.  Let $k'=|S|$.  Note
that $k_1=1+|S| = 1+k'$ (since $G_1$ also has a source at $v_2$).

Suppose vertex $v_1$ is black.
It is trivalent, so it must have another edge $e'$ directed
towards it, and thus there must be a path from some 
vertex of $\Bb \cap G_1$ to $v_1$.
Thus $1\leq |S| <k$. 
Let $G'$ be the graph $G_1$ with the edge $e$ removed.
Clearly $\OO$ restricts to a perfect orientation of $G'$ with sources $S$, so 
$G'$ has type $(k', n')$ for some $n'$. 
Also  $$\dim \Pi_{G'} < \dim \Pi_{G_1} \leq m(k_1-1) = mk'.$$
Thus 
the conditions of \cref{prop:int1criterion} are satisfied.

Now suppose vertex $v_1$ is white.
First we claim that 
there's at least one source in $G_1 \cap \Bb$.
If not, $G_1$ has $k_1=1$, and so it's not true 
that $\dim \Pi_{G_1} \leq m(k_1-1).$
Let $e_1,\dots, e_{\ell}$ be the edges besides $e$ which are incident to $v_1$. 

Let $G'$ be $G_1$ with edges $e, e_1,\dots, e_{\ell}$ and the vertex $v_1$ removed. 
Clearly $\OO$ restricts to a perfect orientation of $G'$
(this would not have been true if we only removed edge $e$) with sources $S$, so $G'$ has type $(k',n')$ for some $n'$. 
We have  $\dim \Pi_{G'} < \dim \Pi_{G_1} \leq m(k_1-1) = mk'.$
Thus we found a perfect orientation $\OO$ and a  subtree $G'$
such that
the conditions of \cref{prop:int1criterion} are satisfied.
 \end{proof}

By \cref{prop:fail} and \cref{prop:int1criterion}, we obtain the following result.
\begin{corollary}\label{cor:int0}
Let $G$ be a bipartite plabic tree of type 
$(k,km+1)$, such that the $m$-balanced property fails.
Then $\Pi_G$ has $m$-intersection number 0.
\end{corollary}

\subsection{Amplitrees have unique $m$-VRCs and intersection number $1$} 

In this subsection, we show in \cref{prop:amplitree-unique-VRC} that amplitrees are solvable. 
We will also show in \cref{prop:tree1} that a bipartite plabic tree 
is solvable if and only if it is an 
amplitree.

To show solvability, we need to construct VRCs on plabic trees. To help us do this, we will decorate the vertices with subspaces of various dimensions.

\begin{definition}\label{def:subspaces-on-vertices}
    Let $G$ be a bipartite plabic tree with $\Bb=[n]$. Let $\root$ be an internal vertex, and fix $\bz \in \Mat_{m,n}$ with columns $z_1, \dots, z_n$. Orient the edges of $G$ towards $\root$. Going from leaves to root, we associate to each vertex $x$ of $G$ a subspace $V_x^{\rt=\root}=V_x$ of $\CC^m$ as follows:
    \begin{itemize}
        \item If $x=i$ is a boundary vertex, $V_i=\spn({z_i})$.
        \item If $x$ is an internal white vertex with children $b_1, \dots, b_p$, then $V_{x}=V_{b_1} + \cdots + V_{b_p}$.
        \item If $x$ is an internal black vertex with children $w_1, \dots, w_p$, then $V_x= V_{w_1} \cap \cdots \cap V_{w_p}$.
    \end{itemize}
\end{definition}
We often omit the superscript $\rt=\root$ if the root is clear from context. The subspace $V_x^{\rt=\root}$ clearly depends on $\bz$, though we omit this dependence from the notation. Note that if we replace $\bz$ with $g \bz$, for $g \in GL_m$, the subspace $V_x$ becomes $g V_x:= \{g v: v \in V_x\}.$ The following lemma shows the relationship between these vector spaces and VRCs.

\begin{lemma}\label{lem:vrc-vec-in-subspaces}
    Suppose $(\bv, \bR)$ is an $m$-VRC on $G$ with boundary $\bz \in \Mat_{m,n}$. Let $\root$ be an internal vertex of $G$. Then for every black vertex $b$ of $G$, the vector $v_b$ is in the subspace $V_b^{\rt=\root}$.
\end{lemma}
\begin{proof}
    We prove this going from leaves to root. It is true for the leaves by construction. Choose a non-leaf black vertex $b$, and suppose the lemma is true for all grandchildren of $b$. The children of $b$ are white vertices $w_1, \dots, w_p$. Say the children of $w_i$ are black vertices $b_{1}, \dots, b_{j}$. By assumption $v_{b_{\ell}} \in V_{b_{\ell}}$. The definition of $m$-VRCs and the presence of the white vertex $w_i$ implies that vectors $v_b, v_{b_1}, \dots, v_{b_j}$ satisfy a relation with all coefficients nonzero. In particular, $v_b$ is in the span of $v_{b_1}, \dots, v_{b_j}$, which is contained in $V_{b_1}+ \dots+ V_{b_j}=V_{w_i}$. So $v_b$ is in $V_{w_1} \cap \dots \cap V_{w_p}= V_b$.
\end{proof}

We will now restrict our attention to those $\bz$ where the dimension of the subspaces $V_x$ are predictable. This will turn out to be a dense open subset of $\Gr_{m,n}$.

The next definition will use the following notion. If $G$ is a plabic graph, we denote by $G^{\ext}$ the plabic graph obtained by adding a bivalent white vertex and bivalent black vertex next to each boundary vertex as in \cref{fig:extension}.
\begin{figure}
    \centering
    \includegraphics[width=0.3\linewidth]{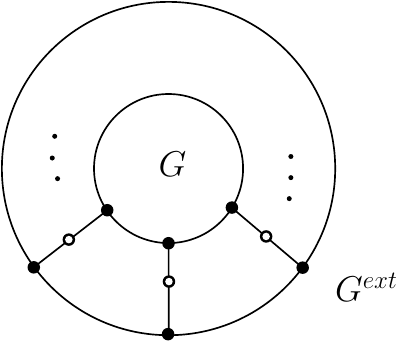}
    \caption{Adding bivalent white and black vertices to $G$ to obtain $G^{\ext}$}
    \label{fig:extension}
\end{figure}

\begin{definition}\label{def:generic-for-tree}
    Let $G, \root, \bz$ and the subspaces $V_x^{\rt=\root}=V_x$ be as in \cref{def:subspaces-on-vertices}. We call $\bz$ \emph{generic for $(G, \root)$} if for all vertices $x$ of $G$,
    \begin{itemize}
    \item if $x$ is a boundary vertex $i$, then $\dim V_x= \dim \spn(z_i) = 1$
        \item if $x$ is white with children $b_1, \dots, b_p$, then $\dim V_x = \min(\dim V_{b_1} + \dots + \dim V_{b_p}, m)$.
        \item if $x$ is black with children $w_1, \dots, w_p$, then $\dim V_x = \max(\dim V_{w_1} + \dots + \dim V_{w_p} -(p-1)m, 0)$.
    \end{itemize}
    We call the right hand sides of the above equations the \emph{expected dimension of $V_x$}, and denote it by $d_{x}$. 
    Genericity is preserved under left-multiplication by an element of $GL_m$, so if $\bz$ is generic for $(G, \root)$, we also say $\rowsp \bz \in \Gr_{m,n}$ is generic for $(G,\root)$.

    We say $\bz \in \Mat_{m,n}$ and $\rowsp 
    \bz \in \Gr_{m,n}$ are \emph{generic for $G$} if they are generic for $(G^\ext, \root)$ for all choices of internal vertex $\root$ of $G^\ext$.
\end{definition}

We note that if $\bz$ is generic for $G$, it is also generic for $(G, \root)$ for all choices of internal vertex $\root$; the extra vertices in $G^\ext$ do not change the dimension conditions when $\root$ is a vertex of $G$. They impose additional dimension conditions when the root is chosen to be one of the added vertices, which will be useful in \cref{lem:bdry-edge-labels}.

The use of the term ``generic" above is justified by the following proposition and corollary.

\begin{proposition}\label{prop:generic-for-tree-root-open-dense}
     Let $G$ be a bipartite plabic tree with $\Bb=[n]$ and let $\root$ be an internal vertex. Then there is an open dense
     subset $\mcu_{G, \root}$ of $\Gr_{m,n}$ such that $\bz \in \mcu_{G, \root}$ is generic for $(G, \root)$.
\end{proposition}

\begin{proof}
We will show, going from leaves to root, that for every vertex $x$ there is a Zariski-open dense subset $\mcu_{x}\subseteq\Gr_{m,n}$ such that $V_y$ has the expected dimension $ d_{y}$ for all $y$ which are descendants of $x$ (including $x$ itself). Then $\mcu_{G, \root}$ will be equal to $\mcu_{\root}$.

Consider $z_1,\ldots,z_n$ as elements of $\bigwedge \CC^m$. We will also represent the vector spaces $V_x$ by elements of $\bigwedge \CC^m$ which we will denote by $F_x(\bz)$. 

As long as $x$ and all of its descendants have the expected dimension, we can use \cref{lem:shuf-wedge-vs-cap-sum} to express $F_x$ algebraically using $\wedge$ and $\shuf$ as follows. We use the notation of \cref{def:generic-for-tree}, and also write $e_1,\ldots,e_m$ for the standard basis of $\mathbb{C}^m.$
\begin{itemize}
\item If $x=i$, $F_x=z_i.$
\item If $x$ is black and $ d_{x}=0$ then $F_x=1,$ which represents the trivial vector space. Otherwise
$F_x= F_{w_1} \shuf \dots \shuf F_{w_p}.$
\item If $x$ is white and $ d_{x}=d_{b_1} + \dots + d_{b_p}$, then $F_x=F_{b_1}\wedge F_{b_2}\wedge\cdots\wedge F_{b_p}.$ Otherwise, $F_x=e_1\wedge\ldots\wedge e_m,$ which represents the vector space $\mathbb{C}^m$. 
\end{itemize}
We will use induction on the number of descendants of $x$. If $x$ has no descendants, then $x=i,$ $ d_{x}=1$ and $V_x = \spn(z_i)$ has expected dimension on the Zariski-open dense subset $\mcu_x$ where $z_i \neq 0$.

Assume $x$ is white. The treatment in black vertices is almost the same, and we omit it.
We assume that the claim holds for $x$'s descendants $b_1,\ldots,b_p$ in the rooted tree, meaning that for each $i$ there is a Zariski open set $\mcu_{b_i}\in\Gr_{m,n}$ in which $\dim V_{y} =d_{y}$ for all descendants of $b_i$.

Assume first $d_x=d_{b_1}+\dots+  d_{b_p}.$
Write \[\mcu_x := \left\{\bz\in\Gr_{m,n}|\bigwedge_{i=1}^pF_{b_i}(\bz)\neq 0\right\}\cap\bigcap_{i=1}^p\mcu_{b_i}.\]
On this set, $\bigwedge_{i=1}^pF_{b_i}(\bz)=F_x$ represents $V_x$ by \cref{lem:shuf-wedge-vs-cap-sum} and $\dim V_x = d_x$. Also, for all descendants $y$ of $x$, $\dim V_y = d_y$.

We now show $\bigwedge_{i=1}^pF_{b_i}(\bz)$ is not identically $0$ as a function from $\Gr_{m,n}$ to $\GC(m)$, hence the complement of its zero locus is Zariski dense and $\mcu_x$ is open and dense. Pick vector spaces $W_1, \dots W_p$ so that $\dim W_i =d_i$ and
\[\dim (W_1+\dots +W_p)=  d_{x}.\]
We can find $\bz\in \Gr_{m,n}$ for which $V_{b_i}=W_i$ for $i=1,\ldots,p.$ First, let $L_i$ denote the leaves which are descendants of $b_i$, and let $\bz_{L_i}$ denote the submatrix using just columns $L_i$. If we act by $g \in GL_m$ just on $\bz_{L_i}$ and leave the remaining columns untouched to obtain $\bz'$, then $F_{b_i}(\bz')$ represents $g V_{b_i}(\bz)$ and $F_{b_j}(\bz')$ represents $V_{b_j}(\bz)$ for $i \neq j$.

Now, pick an arbitrary $\bz\in\bigcap_{i=1}^p\mcu_{b_i}.$ For $i=1,\ldots,p$, pick $g_i\in GL_m$ with $g_iV_{b_i}(\bz)=W_i.$ Then for each $i$, we replace $\bz_{L_i}$ with $g_i \bz_{L_i}$. Call the resulting element $\bz'$.  We have $V_{b_i}(\bz')=W_i$ and $W_i$ is represented by $F_{b_i} (\bz')$. For this $\bz'$, $\bigwedge_{i=1}^pF_{b_i}(\bz')$ is nonzero. If $\bz'$ is of full rank, we are done. Otherwise, note that for an open dense subset of $\mathbf{h}=(h_1,\ldots,h_p)\in GL_m^p$ it holds that $\dim(h_1W_1+\dots+h_p W_p)=d_x$ and $\mathbf{h}\cdot\bz'$, defined by replacing $\bz'_{L_i}$ for $i=1,\ldots,p$ in $h_i\bz'_{L_i}$, is of full rank. Hence the same argument shows that $\bigwedge_{i=1}^pF_{b_i}(\mathbf{h}\cdot\bz')\neq0.
$

The case $d_x =m <d_{b_1}+\dots+  d_{b_p}$ is proven similarly. In this case we put 
\[\mcu_x = \{\bz\in\Gr_{m,n}|F_{b_1} \shuf (F_{b_2} \wedge \dots \wedge F_{b_p}) \neq 0 \}\cap\bigcap_{i=1}^p\mcu_{b_i}.\]
For $\bz \in \mcu_x$, \cref{lem:basic-prop-of-shuf} implies that $V_x=V_{b_1} + \dots + V_{b_p} = \CC^m$ and so $V_x$ has expected dimension. A similar argument as above shows that $F_{b_1} \shuf (F_{b_2} \wedge \dots \wedge F_{b_p})$ is not identically zero on $\Gr_{m,n}$, so $\mcu_x$ is indeed open and dense.
\end{proof}

\begin{corollary}\label{cor:generic-for-tree-open-dense}
    Let $G$ be a bipartite plabic tree with boundary $[n]$ and black leaves. Then there is an open dense subset $\mcu_{G}$ of $\Gr_{m,n}$ such that $\bz \in \mcu_{{G}}$ is generic for $G$.
\end{corollary}
\begin{proof}
Consider the intersection
\[\bigcap_{\root \text{ internal vertex of }G^\ext} \mcu_{G^\ext, \root}\]
where $\mcu_{G^\ext, \root}$ is the open dense set consisting of $\bz$ which are generic for $(G^\ext, \root)$, guaranteed to exist by \cref{prop:generic-for-tree-root-open-dense}. This intersection consists of $\bz$ that are generic for $G$, by definition, and is a finite intersection of open dense subsets, so is open and dense.
\end{proof}

We now turn to $(k,m)$-amplitrees, and show that in this case, the dimensions of the subspaces $V_x^{\rt=\root}$ for generic $\bz$ can be read off of the tree. We will use this to produce VRCs for amplitrees using the Grassmann-Cayley algebra.

\begin{lemma}\label{lem:dim-from-graph-stat}
    Let $G$ be a $(k,m)$-amplitree, $\root$ an internal vertex, and let $\bz \in \Gr_{m,n}$ be generic for $(G, \root)$. For a vertex $x\neq \root$, let $e$ be the edge from $x$ to its parent, and let $G'$ be the connected component of $G \setminus e$ containing $x$. Then
    \begin{equation}\label{eq:dim-from-graph-stat}
    \dim V_x=|\Bb \cap G'| - m \sum_{b \in \Bi \cap G'} (\deg b-2).\end{equation}
    Further, $1 \leq \dim V_x \leq m$ and
    \begin{itemize}
      \item if $x$ is white with children $b_1, \dots, b_p$, then $\dim V_x = \dim V_{b_1} + \dots + \dim V_{b_p}$.
        \item if $x$ is black with children $w_1, \dots, w_p$, then $\dim V_x = \dim V_{w_1} + \dots + \dim V_{w_p} -(p-1)m$.
    \end{itemize}
\end{lemma}

\begin{proof}
We proceed from leaves to root. Equation \eqref{eq:dim-from-graph-stat} is true for boundary vertices $x=i$, since in that case $\dim V_i =1$ and $G'$ consists of just the boundary vertex $i$.

Now, assume \eqref{eq:dim-from-graph-stat} holds for the children $u_1, \dots, u_p$ of $x$. Say that, if you remove the edge from $u_i$ to $x$ from $G$, the connected component $G'_i$ containing $u_i$ has $L_i$ leaves and 
\[\sum_{b \in \Bi \cap G'_i} (\deg b-2 )=:D_i.\]
We have 
\[\dim V_{u_1} + \dots + \dim V_{u_p}= L_1 + \dots + L_p -m (D_1 + \dots + D_p).\]

Suppose $x$ is a white vertex. Then $V_x = V_{u_1} + \dots + V_{u_p}$. Since $\bz$ is generic for $(G, \root)$, we have 
\[\dim V_x = \min(\dim V_{u_1} + \dots + \dim V_{u_p}, m) =\min(L_1 + \dots + L_p -m (D_1 + \dots + D_p), m) .\]
    Now, since $x$ is white, we have 
    \[L_1 + \dots + L_p -m (D_1 + \dots + D_p)=|\Bb \cap G'| - m \sum_{b \in \Bi \cap G'}(\deg b-2) \leq m \]
    where the final inequality is because $G$ is $m$-balanced.
This implies 
\[\dim V_x =L_1 + \dots + L_p -m (D_1 + \dots + D_p) = |\Bb \cap G'| - m \sum_{b \in \Bi \cap G'}(\deg b-2) \]
as desired.

If instead $x$ is black, the argument is very similar. We have $V_x = V_{u_1} \cap \dots \cap V_{u_p}$ and because of genericity, 
\[\dim V_x = \max (0, \dim V_{u_1} + \dots + \dim V_{u_p} - (p-1)m).\]
We have 
\begin{align*}\dim V_{u_1} + \dots + \dim V_{u_p} - (p-1)m &= L_1 + \dots + L_p -m (D_1 + \dots + D_p) - m(\deg x - 2) \\
&= |\Bb \cap G'| - m \sum_{b \in \Bi \cap G'}(\deg b-2) \ge 1 \end{align*}
where the last inequality is because $G$ is $m$-balanced. So again, \eqref{eq:dim-from-graph-stat} holds.

The argument above also shows the bound $1 \leq \dim V_x \leq m$ and the equalities
$ \dim V_x = \sum_{i=1}^p\dim V_{u_i} $ if $x$ is white and 
$ \dim V_x = \sum_{i=1}^p\dim V_{u_i}-(p-1)m$
if $x$ is black.
\end{proof}

\cref{lem:dim-from-graph-stat} gives the following corollary.

\begin{corollary}\label{cor:direct-sum-white-children}
   Let $G$ be a $(k,m)$-amplitree, $\root$ an internal vertex, and let $\bz \in \Gr_{m,n}$ be generic for $(G, \root)$. If $x \neq r$ is white and has children $b_1, \dots, b_p$, then $V_{b_1} + \dots + V_{b_p} = V_{b_1} \oplus \dots \oplus V_{b_p}$.
\end{corollary}

The next lemma shows that for amplitrees, there is only one choice (up to gauge) for the vectors in a VRC.

\begin{lemma} \label{lem:dim-root-subspace}
    Let $G$ be a $(k,m)$-amplitree, let $\root$ be an internal black vertex, and let $\bz \in \Gr_{m,n}$ be generic for $(G, \root)$. Then the subspace $V_\root$ has dimension $1$.
\end{lemma}
\begin{proof}
    Similar to the proof of \cref{lem:dim-from-graph-stat}, say the children of $\root$ are $u_1, \dots, u_p$ and let $L_i, D_i$ be as in that proof. We have that 
    \[\dim V_\root = \max(0, \dim V_{u_1} + \dots + \dim V_{u_p} - (p-1)m).\]
    Using \cref{lem:dim-from-graph-stat}, 
 \begin{align*}\dim V_{u_1} + \dots + \dim V_{u_p} - (r-1)m &= L_1 + \dots + L_p -m (D_1 + \dots + D_p) - m(\deg x - 2 +1) \\
&= |\Bb| - m \sum_{b \in \Bi}(\deg b-2) -m\\
&= km+1 - m(k-1) -m = 1 \end{align*}
where the second-to-last equality uses \cref{lem:facts} (2) and the fact that a $(k,m)$-amplitree by definition has $km+1$ boundary vertices.
\end{proof}

With this information about the subspaces $V_x$, we can give an explicit $m$-VRC on amplitree $G$ with boundary $\bz$. The first step is to define elements of the Grassmann-Cayley algebra which will represent the subspaces $V_x$.

\begin{definition}\label{def:vec-in-VRC-from-GC}
    Let $G$ be a $(k,m)$-amplitree. Choose $\root \in \Bi \cup W$. Orient all edges towards $\root$. Moving from leaves to $\root$, define a function $F^{\rt=\root}_x: \Mat_{m,n} \to \GC(m)$ for each vertex by the following formulas:
    \begin{itemize}
        \item if $x$ is a boundary vertex $i$, $F^{\rt=\root}_x (\bz):=z_i$;
        \item if $x$ is white with children $b_1, \dots, b_p$ reading clockwise from the parent of $x$, $F^{\rt=\root}_x(\bz):= \bigwedge_{i=1}^{p} F_{b_i}(\bz) $;
        \item if $x$ is black with children $w_1, \dots, w_p$ reading clockwise from the parent of $x$, $F^{\rt=\root}_x(\bz):= F_{w_1}(\bz) \shuf \dots \shuf F_{w_p}(\bz)$.
    \end{itemize}
\end{definition}

As usual, we drop ``$\rt = \root$" if the choice of root is clear. We note changing the ordering of the children of $x$ will change the sign of $F_x(\bz)$ but has no other effect.

\begin{lemma}\label{lem:GC-elts-rep-subspaces}
Let $G$ be a $(k,m)$-amplitree and $\root \in W \cup \Bi$ be an internal vertex. If $\bz \in \Mat_{m,n}$ is generic for $(G, \root)$, then $F_x(\bz)$ represents the subspace $V_x$. 

In particular, if $\root\in \Bi$, then $F_\root(\bz)$ is a nonzero vector which can be expressed as a linear combination of $\bz$'s columns, whose coefficients are polynomials in the Pl\"ucker coordinates of $\bz$.
\end{lemma}
\begin{proof}
For the first statement, we move from leaves to root. For $x$ a vertex with children $u_1, \dots, u_p$, assume $F_{u_i}$ represents $V_{u_i}$. \cref{lem:dim-from-graph-stat} shows that the hypotheses of \cref{lem:shuf-wedge-vs-cap-sum} hold for the $F_{u_i}$. Applying \cref{lem:shuf-wedge-vs-cap-sum} shows that $F_x$ represents $V_x.$

The second statement follows from \cref{lem:dim-root-subspace} and expanding out the shuffles and wedges.
\end{proof}

The vectors $F^{\rt=b}_b(\bz)$ will be the vectors of the VRC. To define the coefficients, we need the following lemma.

\begin{lemma}\label{lem:bdry-edge-labels}
Let $\bz \in \Mat_{m,n}$ be generic for $(k,m)$-amplitree $G$ and let $i \in \Bb$. Suppose the adjacent white vertex is $w$, and $w$ has neighbors $i, b_1, \dots, b_p$. Choose $w$ as the root. Then $\bigwedge_{j=1}^p F_{b_j}(\bz) \in \bigwedge^m \CC^m \setminus \{0\}$ and $\lr{\bigwedge_{j=1}^p F_{b_j}(\bz)} \in \CC^*$.
\end{lemma}
\begin{proof}
Consider $G^{\ext}$ and choose as the root the black vertex $b$ placed between $w$ and $i$. Then $F^{\rt = b}_{b_j}(\bz) =F^{\rt = w}_{b_j}(\bz) = F_{b_j}(\bz).$ We also have that $F_w^{\rt=b}(\bz)$ is equal to $\bigwedge_{j=1}^p F_{b_j}(\bz)$. Because $\bz$ is generic, the dimension of the subspace represented by $F_w^{\rt=b}(\bz)$ is given by \eqref{eq:dim-from-graph-stat}, which is equal to $m$ in this case. This implies that $F_w^{\rt=b}(\bz) \in \bigwedge^m \CC^m$ and is nonzero. By definition (see \cref{not:brackets}), $\lr{\bigwedge_{j=1}^p F_{b_j}(\bz)}$ is a scalar, and is nonzero because the wedge is nonzero.
\end{proof}

\begin{definition}\label{def:coeffs-in-VRC-from-GC}
Let $G$ be a $(k,m)$-amplitree and let $\bz \in \Mat_{m,n}$. For an edge $e$ of $G$, define $f_e(\bz):= 1$ if $e$ is not adjacent to a boundary vertex. If $e$ is adjacent to boundary vertex $i$ and white vertex $w$, use $w$ as the root and define 
\[f_e(\bz):= \langle\bigwedge_{j=1}^p F_{b_j}(\bz) \rangle\]
where $b_1, \dots, b_p$ are the other neighbors of $w$, reading clockwise from $i$. If $w$ has degree at least three, then 
\[f_e(\bz)= F_{b_1} \shuf \left(\bigwedge_{j=2}^p F_{b_j}(\bz) \right).\]
\end{definition}

\begin{remark} Let $G$ be a $(k,m)$-amplitree. Suppose $i \in \Bb$ is adjacent to a white vertex $w$, and $\deg(w) \ge 3$.
Let $T_i$ be $G$ with the edge $e=(i,w)$ removed. Then $T_i$ is an $s\ell_m$-\emph{tensor diagram} \cite{CKM, FLL} and the associated tensor invariant is, up to sign, equal to $f_e(\bz)$. If $\deg(w)=2$, then $f_e(\bz)$ factors into a product of tensor invariants, whose corresponding tensor diagrams are smaller subtrees of $G$.
\end{remark}

The next proposition shows that the functions $F^{\rt=b}_b(\bz)$ and $f_e(\bz)$ give us (up to sign) the vectors and relations of a VRC with boundary $\bz$.

\begin{proposition} \label{prop:amplitree-VRC-rep-up-to-sign}
Suppose $\bz = [z_1 \dots z_n] \in \Mat_{m,n}$ is generic for $(k,m)$-amplitree $G$. For $i \in \Bb$, set $v_i := z_i$. For $b \in \Bi$, set $v_b(\bz) := F^{\rt=b}_b(\bz)$. For $e$ an edge of $G$, set $r_e:= f_e(\bz)$. Then there is a choice of signs $\{s_e\} \subset \{\pm 1\}^{E(G)}$ so that $(\{v_b\}, \{s_e r_e\})$ is a non-degenerate $m$-VRC with boundary $\bz$.
\end{proposition}

\begin{proof}
Fix a white vertex $w$ with neighbors $b_1, \dots b_p$ and set $e_i = (w, b_i)$. We check that
\[\sum_{i=1}^p \pm f_{e_i}(\bz) v_{b_i}(\bz) =0.\]
For the remainder of the proof, we suppress the dependence on $\bz$. 

Notice that if $b_i \in \Bi$, then
\[v_{b_i}=F^{\rt=b_i}_{b_i} = (-1)^{c_i} F^{\rt = w}_{b_i} \shuf (F^{\rt=w}_{b_1} \wedge \cdots \wedge \widehat{F^{\rt=w}_{b_i}} \wedge \cdots \wedge F^{\rt=w}_{b_p})  \]
for some integer $c_i$. Indeed, $F^{\rt = w}_{b_i}$ is the shuffle of all $F^{\rt=b_i}_{w'}$ where $w'$ is a neighbor of $b_i$ and is not equal to $w.$ The other term $\bigwedge_{j \neq i} F^{\rt=w}_{b_j}$ is equal to $\pm F^{\rt=b_i}_{w}$. So the expression above is, up to sign, the shuffle of all $F^{\rt=b_i}_{w'}$ where $w'$ is a neighbor of $b_i$. By definition this is equal to $\pm F^{\rt=b_i}_{b_i}.$

Let $d_i$ be the integer such that $F^{\rt=w}_{b_i} \in \bigwedge ^{d_i} \CC^m$. We have
\[d_1 + \dots + d_p= |\Bb| - m \sum_{b \in \Bi} (\deg b-2) = m+1\]
where the first equality is from \eqref{eq:dim-from-graph-stat}, and the second is from \cref{lem:facts} (2). So we may apply \cref{lem:GC-easy-reln} to $F^{\rt=w}_{b_1}, \dots, F^{\rt=w}_{b_p}$ and obtain
\[\sum_{i=1}^p (-1)^{d_i(d_i+ \dots +d_p)} (F^{\rt=w}_{b_1} \wedge \cdots \widehat{F^{\rt=w}_{b_i}} \wedge \cdots \wedge F^{\rt=w}_{b_p}) \shuf F^{\rt=w}_{b_i} =0. \]

The $i$th term in this relation is exactly $(-1)^{d_i(d_i+ \dots +d_p) +c_i} v_{b_i}$ if $b_i \in \Bi$. If instead $b_i \in \Bb$, then $F^{\rt=w}_{b_i} = z_{b_i}$ and
\[(F^{\rt=w}_{b_1} \wedge \cdots \widehat{F^{\rt=w}_{b_i}} \wedge \cdots \wedge F^{\rt=w}_{b_p}) \shuf F^{\rt=w}_{b_i} = (F^{\rt=w}_{b_1} \wedge \cdots \widehat{F^{\rt=w}_{b_i}} \wedge \cdots \wedge F^{\rt=w}_{b_p})^*~z_{b_i} = (-1)^{c_i} f_{e_i} z_{b_i} \]
for some integer $c_i$. In either case, we set $s_{e_i}:=(-1)^{{d_i(d_i+ \dots +d_p) +c_i}}$.

In summary, we have 
\[0= \sum_{i=1}^p (-1)^{d_i(d_i+ \dots +d_p)} (F^{\rt=w}_{b_1} \wedge \cdots \widehat{F^{\rt=w}_{b_i}} \wedge \cdots \wedge F^{\rt=w}_{b_p}) \shuf F^{\rt=w}_{b_i} =\sum_{i=1}^p s_{e_i}f_{e_i}(\bz) v_{b_i}(\bz) =0 \]
as desired.
\end{proof}

The next theorem shows that the $m$-VRC of \cref{prop:amplitree-VRC-rep-up-to-sign} represents the unique element of $\mzVRC_G$.

\begin{theorem}\label{prop:amplitree-unique-VRC} 
    Let $G$ be a $(k,m)$-amplitree and let $\bz \in \Gr_{m,n}$ be generic for $G$. Then there is a unique element of $\mVRC_G$ with boundary $\bz$.
\end{theorem}

\begin{proof}
We fix a representative $[z_1 \dots z_n]$ of $\bz$ which we denote also by $\bz$. By \cref{prop:amplitree-VRC-rep-up-to-sign}, there exists a non-degenerate $m$-VRC $(\bv, \bR)$ with boundary $\bz$. We first show that all other non-degenerate $m$-VRCs with boundary $\bz$ are related to $(\bv, \bR)$ by gauge transformations. We then show that there are no degenerate $m$-VRCs with boundary $\bz$. 

Suppose $(\bv', \bR')$ is a non-degenerate $m$-VRC with boundary $\bz$. By \cref{lem:vrc-vec-in-subspaces}, $v_b'$ and $v_b$ lie in the subspace $V_{b}^{\rt = b}$, which is a line by \cref{lem:dim-root-subspace}. So $v_b' = c_b v_b$ for some nonzero constant $c_b$. Using gauge, we may assume $c_b=1$ for all $b$. Choose a white vertex $w \in W$. Say its neighbors are $b_1, \dots, b_p$, and let $e_i$ denote the edge $(w, b_i)$. We have
\[v_{b_1} = -\frac{1}{r_{e_1}} \left(\sum_{i=2}^p r_{e_i} v_{b_i}\right) =  v_{b_1}'=-\frac{1}{r'_{e_1}} \left(\sum_{i=2}^p r'_{e_i} v'_{b_i}\right) = -\frac{1}{r'_{e_1}} \left(\sum_{i=2}^p r'_{e_i}  v_{b_i}\right) \]
and so 
\[ 0 = \sum_{i=2}^p \left(\frac{r_{e_i}}{r_{e_1}} - \frac{r'_{e_i}}{r'_{e_1}} \right) v_{b_i}.\]
Notice that, by \cref{lem:vrc-vec-in-subspaces}, $v_{b_i} \in V_{b_i}^{\rt = w}$, and by definition, $V_{b_i}^{\rt = w} = V_{b_i}^{\rt = b_1}$. By \cref{cor:direct-sum-white-children}, 
\[V_{b_2}^{\rt = b_1} + \dots +V_{b_p}^{\rt = b_1} = V_{b_2}^{\rt = b_1} \oplus \dots  \oplus V_{b_p}^{\rt = b_1}\]
and so in particular, $v_{b_2}, \dots, v_{b_p}$ are linearly independent. This implies that for $i=2, \dots, p$
\[\frac{r_{e_i}}{r_{e_1}} = \frac{r'_{e_i}}{r'_{e_1}}\]
which implies that the coefficients around $w$ are related by a gauge transformation.

Now suppose $(\bv', \bR')$ is a degenerate $m$-VRC with boundary $\bz$; we will arrive at a contradiction. There is a vector $v_{b_1}$ which is equal to zero. Say it is adjacent to white vertex $w$, whose other neighbors are $b_2, \dots, b_p$. As above, the vector $v_{b_i} \in V_{b_i}^{\rt = b_1}$ and 
\[V_{b_2}^{\rt = b_1} + \dots +V_{b_p}^{\rt = b_1} = V_{b_2}^{\rt = b_1} \oplus \dots  \oplus V_{b_p}^{\rt = b_1}\]
so if they are nonzero, they are linearly independent. Since a linear combination of the $v_{b_i}$ with positive coefficients is equal to $v_{b_1}=0$, the only possibility is that $v_{b_i}=0$ for $i=2, \dots, p$ as well. Continuing this argument, we eventually obtain that some boundary vector is $0$, a contradiction of $\bz$ being generic.
  
\end{proof}

\begin{theorem}\label{prop:tree1}
Let $G$ be a bipartite plabic tree of type $(k,km+1)$.  Then
$\mint(G)=1$
if and only if $G$ is {$m$-balanced}. If $G$ is not $m$-balanced, then $\mint(G)=0$.
\end{theorem}

\begin{proof}[Proof of \cref{prop:tree1}]
Suppose that $G$ is a bipartite plabic tree of type 
$(k,km+1)$ but it does not satisfy the $m$-balanced
property.  Then by \cref{cor:int0}, 
 $\Pi_G$ has $m$-intersection number $0$.

Conversely, suppose that $G$ is a bipartite plabic tree of 
type $(k,km+1)$ which satisfies the $m$-balanced property.
Then by \cref{prop:amplitree-unique-VRC}, 
for generic $\bz \in \Gr_{m,n}$, there is a unique element of $\mVRC_G$ with boundary $\bz$.
But now by \cref{cor:int-num=num-VRC},
$\Pi_G$ has intersection number $1$.
\end{proof}

\cref{prop:tree1} shows that in the case of trees, \cref{conj:IN0} holds: that is, the intersection number is $0$ if and only if we can find a 
subgraph as in \cref{prop:int1criterion}.

\subsection{Enumeration of amplitrees}

Recall from \cref{def:mbalanced}
that a {$(k,m)$-amplitree} is a bipartite plabic tree $G$ 
of type $(k,km+1)$ which is $m$-balanced.
In this section we give some enumerative results about move-equivalence classes of amplitrees.  Since an amplitree is in particular a tree, the only moves that we can apply to an amplitree are the expand and contract moves, as in \cref{fig:expand-contract}.

\cref{tab:amplitrees} gives computational data about the enumeration of amplitrees of type $(k,m)$.

\begin{table}[h] 
\centering 
\begin{tabular}{|c||>{\centering\arraybackslash}p{1.5cm}|>{\centering\arraybackslash}p{1.5cm}|>{\centering\arraybackslash}p{1.5cm}|>{\centering\arraybackslash}p{1.5cm}|>{\centering\arraybackslash}p{1.5cm}|} 
\hline 
\diagbox{$m$}{$k$} & $1$ & $2$ & $3$ & $4$ & $5$  \\
\hline\hline 
$1$ & $1$ & $1$ & $1 $ & $1 $ &  $1 $  \\
\hline
 $2$ & $1$ & $ 5$ & $ 35$ & $285 $ & $2530$  \\
 \hline 
 $3$ & $1$ & $14$ & $ 280$ & $ 6565$ &  $160916$ \\
\hline 
 $4$ & $1$ & $30$ & $1274$ & $63410 $ & $  $ \\
\hline 
$5$ & $1$ & $55$ & $4228$ & $380310$ & $ $  \\
\hline
$6$ & $1$ & $91$ & $11438$ & $ $ & $ $  \\
\hline
\end{tabular}
\vspace{1em}
\caption{The number of (move-equivalence classes of) amplitrees of type $(k,m)$.
}
\label{tab:amplitrees}
\end{table}

We also have the following results.
\begin{theorem}
When $k=2$, the number $c_{k,m}$ of move-equivalence classes of amplitrees of type $(k,m)$ is 
$$c_{2,m}=1^2+2^2+ \dots + m^2=\frac{m(m+1)(2m+1)}{6}.$$  This is the sequence of \emph{square pyramidal numbers}, and it appears as sequence A000330 in \cite{OEIS}.  We can write the generating function as 
\begin{equation}
\label{eq:pyramidal}
\sum_{m\geq 0} c_{2,m}x^m = \frac{1+x}{(1-x)^4}.
\end{equation}
\end{theorem}
\begin{proof}

Each move-equivalence class of a $(2,m)$-amplitree can be represented by a tree with $2m+1$ boundary vertices (leaves), with a single trivalent internal black vertex connected to three internal white vertices, where each white vertex is connected to $\ell$ boundary vertices for some $1 \leq \ell \leq m$.

Let $d_0+d_1+1$ be the number of boundary vertices adjacent to the white vertex which is closest to the boundary vertex $1$; let $d_0$ and $d_1$ denote the number of boundary vertices which are counterclockwise and clockwise of the $1$.  Let $d_2$ and $d_3$ denote the number of boundary vertices which are incident to the other two white vertices.
Thus we must have that $1 \leq d_0+d_1 \leq m$, $1 \leq d_2 \leq m$, and $1\leq d_3 \leq m$, and $1+d_0+d_1+d_2+d_3=2m+1$.

If we let $f_m:=c_{2,m}$ be the number of move-equivalence classes of $(2,m)$-amplitrees, then
we claim that 
\begin{equation}
\label{eq:recurrence}
f_m=5f_{m-1} - 10f_{m-2}+10f_{m-3}-5f_{m-4}+f_{m-5}.
\end{equation}

To prove the claim, note that to get such an amplitree on $2m+1$ leaves from an amplitree on $2m-1$ leaves, we can increase $d_i$ and $d_j$ by 1 where $(i,j) \in \{(0,2), (0, 3), (1,2), (1,3), (2,3)\}$. 
However, any two of these operations commute, so e.g. 
increasing $d_0$ and $d_2$ then increasing $d_1$ and $d_3$ results in the same tree as 
increasing $d_1$ and $d_3$, and then increasing $d_0$ and $d_2$.
Thus by inclusion-exclusion, \eqref{eq:recurrence} holds.

Since \eqref{eq:recurrence} is a  linear recurrence, it follows that $f_m$ is a polynomial in $m$ of degree at most $4$.  Computation of the initial terms results in the formula \eqref{eq:pyramidal}.
\end{proof}

The same style of argument can be used to prove the following result.
\begin{theorem}
When $k=3$, the number $c_{k,m}$ of move-equivalence classes of amplitrees of type $(k,m)$ has generating function given by
$$\sum_{m\geq 0} c_{3,m}x^m = \frac{1+28x+56x^2+14x^3}{(1-x)^7}.$$
\end{theorem}

Similarly, for each fixed $k$ we expect the number of $(k,m)$-amplitrees to satisfy a generating function of the form $p(x) / (1-x)^{3k-2}$ where $p(x)$ is a polynomial of degree at most $3k-2$. This implies that $c_{k,m}= \Theta(m^{3k-3})$ as $m$ tends to infinity, where the implied constants depend on~$k$. If $m \geq 2$ is fixed and $k$ tends to infinity, then the number of $(k,m)$-amplitrees is exponential in~$k$. Indeed, an exponential lower bound can be shown by considering different plane embeddings of trees isomorphic to chain trees, cf \cref{def-chain-tree}, while an exponential upper bound holds for all trees in the plane in general. It remains an intriguing problem to determine how the base of the exponent depends on $m$, or any further information about the joint asymptotic behavior of  $c_{k,m}$ when both $k$ and $m$ are large.  One could also hope to compute an explicit two-variable generating function keeping track of both $k$ and $m$.

\section{Proofs of the quasi-cluster homomorphism property}\label{sec:proofscluster}

In this section, we provide the proofs that the promotions from \cref{sec:promotion-examples} are quasi-cluster homomorphisms.

\subsection{Star promotion}

Recall \cref{pro-upper} of unary star promotion.

\begin{theorem}\label{thm:upper}
Unary star promotion $\aProm_{m}$ is a quasi-homomorphism of cluster algebras from $\mathcal{A}(\Sigma)$ to $\mathcal{A}(\overline{\Sigma})$, where $\Sigma$ is the rectangles seed for ${\Gr}_{4,N'}$ and $\overline{\Sigma}$ is the rectangles seed for ${\Gr}_{4,n}$ with the cluster variables in the leftmost column frozen, as shown in \cref{rectangles}.
\end{theorem}

\begin{notation}
\label{rectangles}
The initial seed $\Sigma$ is the rectangles seed
for $\Gr_{m, \{1,3,\dots,n\}}$ (see \cref{sec:cluster}), shown below for $m=5$:
\vspace{0.5cm}
\begin{center}
\begin{tikzpicture}[scale=1.2, every node/.style={minimum size=0.5cm}, every path/.style={->, thick}, node distance=0.5cm]
\def\cols{7}
\def\rows{5}
\def\ax{1.7}
\node (n00) at (\ax*0,-1) {\framebox{$\langle13456\rangle$}};
\node (n11) at (\ax*1,-1) {$\langle13457\rangle$};
\node (n12) at (\ax*2,-1) {$\langle13458\rangle$};
\node (n13) at (\ax*3,-1) {$\langle13459\rangle$};
\node (n14) at (\ax*4,-1) {$\cdots$};
\node[anchor=west] (n15) at (\ax*4.6,-1) {\framebox{$\langle1345n\rangle$}};
\node (n21) at (\ax*1,-2) {$\langle13467\rangle$};
\node (n22) at (\ax*2,-2) {$\langle13478\rangle$};
\node (n23) at (\ax*3,-2) {$\langle13489\rangle$};
\node (n24) at (\ax*4,-2) {$\cdots$};
\node[anchor=west] (n25) at (\ax*4.6,-2) {\framebox{$\langle134(n{-}1)n\rangle$}};
\node (n31) at (\ax*1,-3) {$\langle13567\rangle$};
\node (n32) at (\ax*2,-3) {$\langle13678\rangle$};
\node (n33) at (\ax*3,-3) {$\langle13789\rangle$};
\node (n34) at (\ax*4,-3) {$\cdots$};
\node[anchor=west] (n35) at (\ax*4.6,-3) {\framebox{$\langle13(n{-}2)(n{-}1)n\rangle$}};
\node (n41) at (\ax*1,-4) {$\langle14567\rangle$};
\node (n42) at (\ax*2,-4) {$\langle15678\rangle$};
\node (n43) at (\ax*3,-4) {$\langle16789\rangle$};
\node (n44) at (\ax*4,-4) {$\cdots$};
\node[anchor=west] (n45) at (\ax*4.6,-4) {\framebox{$\langle1(n{-}3)(n{-}2)(n{-}1)n\rangle$}};
\node (n51) at (\ax*1,-5) {\framebox{$\langle34567\rangle$}};
\node (n52) at (\ax*2,-5) {\framebox{$\langle45678\rangle$}};
\node (n53) at (\ax*3,-5) {\framebox{$\langle56789\rangle$}};
\node (n54) at (\ax*4,-5) {$\cdots$};
\node[anchor=west] (n55) at (\ax*4.6,-5) {\framebox{$\langle(n{-}4)(n{-}3)(n{-}2)(n{-}1)n\rangle$}};
\node[gray,circle,draw] (sigma) at (\ax*6.5,-1.5) {$\Sigma$};
\foreach \i in {1,2,3,4} {
\foreach \j [evaluate=\j as \k using int(\j+1)] in {1,...,4} { 
\draw[->] (n\i\j) -- (n\i\k); }}
\foreach \i [evaluate=\i as \k using int(\i+1)] in {1,...,4} {
\foreach \j in {1,...,3} {
\draw[->] (n\i\j) -- (n\k\j);}}
\foreach \i [evaluate=\i as \k using int(\i+1)] in {1,...,4} {
\foreach \j [evaluate=\j as \l using int(\j+1)] in {1,...,4} {
\draw[->] (n\k\l) -- (n\i\j);}}
\draw[<-] (n11) -- (n00);
\end{tikzpicture}
\end{center}
\vspace{0.25cm}
The seed $\overline{\Sigma}$ is the following rectangles seed for $\Gr_{m,n}$ with the cluster variables in the leftmost column frozen, also shown for $m=5$:
\vspace{0.5cm}
\begin{center}
\begin{tikzpicture}[scale=1.2, every node/.style={minimum size=0.5cm}, every path/.style={->, thick}, node distance=0.5cm]
\def\cols{7}
\def\rows{5}
\def\ax{1.7}
\node (n00) at (\ax*0,0) {\framebox{$\langle12345\rangle$}};
\node (n10) at (\ax*0,-1) {\dbox{$\langle12346\rangle$}};
\node (n11) at (\ax*1,-1) {$\langle12347\rangle$};
\node (n12) at (\ax*2,-1) {$\langle12348\rangle$};
\node (n13) at (\ax*3,-1) {$\langle12349\rangle$};
\node (n14) at (\ax*4,-1) {$\cdots$};
\node[anchor=west] (n15) at (\ax*4.6,-1) {\framebox{$\langle1234n\rangle$}};
\node (n20) at (\ax*0,-2) {\dbox{$\langle12356\rangle$}};
\node (n21) at (\ax*1,-2) {$\langle12367\rangle$};
\node (n22) at (\ax*2,-2) {$\langle12378\rangle$};
\node (n23) at (\ax*3,-2) {$\langle12389\rangle$};
\node (n24) at (\ax*4,-2) {$\cdots$};
\node[anchor=west] (n25) at (\ax*4.6,-2) {\framebox{$\langle123(n{-}1)n\rangle$}};
\node (n30) at (\ax*0,-3) {\dbox{$\langle12456\rangle$}};
\node (n31) at (\ax*1,-3) {$\langle12567\rangle$};
\node (n32) at (\ax*2,-3) {$\langle12678\rangle$};
\node (n33) at (\ax*3,-3) {$\langle12789\rangle$};
\node (n34) at (\ax*4,-3) {$\cdots$};
\node[anchor=west] (n35) at (\ax*4.6,-3) {\framebox{$\langle12(n{-}2)(n{-}1)n\rangle$}};
\node (n40) at (\ax*0,-4) {\dbox{$\langle13456\rangle$}};
\node (n41) at (\ax*1,-4) {$\langle14567\rangle$};
\node (n42) at (\ax*2,-4) {$\langle15678\rangle$};
\node (n43) at (\ax*3,-4) {$\langle16789\rangle$};
\node (n44) at (\ax*4,-4) {$\cdots$};
\node[anchor=west] (n45) at (\ax*4.6,-4) {\framebox{$\langle1(n{-}3)(n{-}2)(n{-}1)n\rangle$}};
\node (n50) at (\ax*0,-5) {\framebox{$\langle23456\rangle$}};
\node (n51) at (\ax*1,-5) {\framebox{$\langle34567\rangle$}};
\node (n52) at (\ax*2,-5) {\framebox{$\langle45678\rangle$}};
\node (n53) at (\ax*3,-5) {\framebox{$\langle56789\rangle$}};
\node (n54) at (\ax*4,-5) {$\cdots$};
\node[anchor=west] (n55) at (\ax*4.6,-5) {\framebox{$\langle(n{-}4)(n{-}3)(n{-}2)(n{-}1)n\rangle$}};
\node[gray,circle,draw] (sigma) at (\ax*6.5,-1.5) {$\overline{\Sigma}$};
\foreach \i in {1,2,3,4} {
\foreach \j [evaluate=\j as \k using int(\j+1)] in {0,...,4} { 
\draw[->] (n\i\j) -- (n\i\k); }}
\foreach \i [evaluate=\i as \k using int(\i+1)] in {1,...,4} {
\foreach \j in {0,...,3} {
\draw[->] (n\i\j) -- (n\k\j);}}
\foreach \i [evaluate=\i as \k using int(\i+1)] in {1,...,4} {
\foreach \j [evaluate=\j as \l using int(\j+1)] in {0,...,4} {
\draw[->] (n\k\l) -- (n\i\j);}}
\draw[<-] (n10) -- (n00);
\end{tikzpicture}
\end{center}
\vspace{0.25cm}
In both seeds, we label the rows $1$ to $m$ from top to bottom, and label the columns $m+1,\dots,n$ from left to right, such that each column is labeled by the largest index that appears in it. 
We let $x_{rc}$, respectively $\overline{x}_{rc}$ denote
the mutable cluster variable of $\Sigma$, respectively $\overline{\Sigma}$ in row $r$ and column $c$.
\end{notation}

We note the following corollary of \cref{thm:upper}.

\begin{corollary}
\label{cor:irreducible-quadratics}
Let $A, B, C$ be disjoint ordered sets in $[n]$, 
i.e. $a<b<c$ for each $a\in A$, $b\in B$, and $c\in C.$  
Suppose moreover that $|C| = m-1$, and that 
$|A|+|B|+|C| = 2m$.  Then the chain polynomial
$\lr{A \shuf B \shuf C}$ is a cluster variable for $\Gr_{m,n}$. 
\end{corollary}
\begin{proof}[Proof of \cref{cor:irreducible-quadratics}]
Let $\overline{\Psi}_m$ be the \emph{reduced unary star promotion map}, defined the same way as 
$\Psi_m$ in \cref{pro-upper} but without a denominator. First let $C=\{n-m+2,n-m+3,\dots,n\}.$  Then it follows from \cref{rem:helpful} below that for $3 \leq j \leq m$, 
\begin{align*}
\overline{\Psi}_m(\lr{j,C}) & \;=\; \lr{1,\dots,\widehat{j{-}1}, \dots,m{+}1}\lr{j{-}1,C} 
- \lr{1,\dots, \widehat{j{-}2},\dots, m{+}1}\lr{j{-}2,C} \\ & \quad\quad\quad\quad + \lr{1,\dots, \widehat{j{-}3},\dots, m{+}1}\lr{j{-}3,C} - \;\dots\; \pm \lr{2,\dots,m{+}1} \lr{1,C} \\
&\;=\; \lr{ (1,2,\dots,j{-}1) \shuf (j,j{+}1,\dots,m{+}1) \shuf C}
\end{align*}
Now we know that ``spreading out indices'' \cite[Lemma 4.20]{even2023cluster} preserves the property of being a cluster variable, so the corollary follows.  
\end{proof}

\begin{remark}
Also applying a cyclic shift \cite[Corollary 4.18]{even2023cluster} preserves the property of being a cluster variable, up to a sign. Hence, for any cyclically ordered disjoint sets $A,B,C$ with $|C|=m-1$ and $|A|+|B|=m+1$, 
we have that $(-1)^{r(m-r)}\lr{A \shuf B \shuf C}$ is a cluster variable, where $r$ is the number of indices in $A,B,C$ which appear to the left of 1.
\end{remark}

\begin{remark}
\label{rem:helpful}
In what follows, we use the fact that 
for $3 \leq j \leq m$, the substitution in \cref{pro-upper} of 
unary star promotion
can also be written via Pl\"ucker relations in the form 
\begin{equation*}
j \;\mapsto\;
\frac{\lr{1,\dots,\widehat{j{-}1}, \dots,m{+}1}}{\lr{1,\dots,\widehat{j}, \dots,m{+}1}} (j-1) + \const (j-2) + \const(j-3) + \dots + \const (1),
\end{equation*}
where $\const$ represents similar rational expressions in Pl\"ucker coordinates.
\end{remark}

\begin{proof}[Proof of \cref{thm:upper}]
For concreteness we explain the proof in detail for $m=5$, then indicate how the formulas generalize for arbitrary $m$.

We first compute the images of the cluster variables in the initial seed under $\Psi_5$. It is not hard to see that $\Psi_5$ fixes the cluster variables in rows $4$ and $5$ of the seed $\Sigma$ in \cref{rectangles}, as well as the cluster variable $\lr{13456}$ in the top-left corner of~$\Sigma$. In order to apply $\Psi_5$ to the remaining cluster variables, we use \cref{rem:helpful}, which in this case 
says that 
\begin{align*}
3 & \;\mapsto \; \frac{\lr{13456}}{\lr{12456}} (2) + \const (1)\\
4 & \;\mapsto \; \frac{\lr{12456}}{\lr{12356}} (3) + \const (2) + \const (1)\\
5 & \;\mapsto \; \frac{\lr{12356}}{\lr{12346}} (4) + \const (3) + \const (2) + \const (1)
\end{align*}
Applying the above formulas to the cluster variables in rows $3$, $2$, and $1$, respectively,
we get
\begin{align*}
&\Psi_5(\lr{1,3,i,i+1,i+2}) \;=\; \frac{ \lr{13456}}{\lr{12456}} \cdot \lr{1,2,i,i+1,i+2} \;=\; F_3 \cdot \lr{1,2,i,i+1,i+2} \\
&\Psi_5(\lr{1,3,4,i,i+1}) \;=\; \frac{ \lr{13456}}{\lr{12356}} \cdot \lr{1,2,3,i,i+1} \;=\; F_2 \cdot \lr{1,2,3,i,i+1} \\
&\Psi_5(\lr{1,3,4,5,i}) \;=\; 
\frac{ \lr{13456}}{\lr{12346}} \cdot \lr{1,2,3,4,i} \;=\; F_1 \cdot \lr{1,2,3,4,i} 
\end{align*}
where 
$$ F_1 \;=\; \frac{ \lr{13456}}{\lr{12346}}
\;\;\;\;\;\;\;\;
F_2 \;=\; \frac{ \lr{13456}}{\lr{12356}}
\;\;\;\;\;\;\;\;
F_3 \;=\; \frac{ \lr{13456}}{\lr{12456}}
\;\;\;\;\;\;\;\;
F_4 \;=\; 1 $$
are the ``frozen factors'' corresponding to the four rows. To summarize, for $x_{rc}$ and $\bar{x}_{rc}$, respectively the cluster variable of $\Sigma$ and $\overline{\Sigma}$ as in \cref{rectangles}, our calculations show that
$\Psi_5(x_{rc}) = F_r \cdot \bar{x}_{rc}$ and hence $\Psi_5(x_{rc}) \propto \bar{x}_{rc}$.

Now we check exchange ratios. The exchange ratio for the mutable cluster variable $x_{1i} = \lr{1345i}$ in row $1$ of $\Sigma$ is
$$
\hat{y}_{\Sigma}(x_{1i}) \;=\; \frac{\lr{1345(i{+}1)}\, \lr{134(i{-}1)i}}{\lr{1345(i{-}1)} \,\lr{134i(i{+}1)}} 
$$
and its image under $\Psi_5$ is
$$
\Psi_5\left(\hat{y}_{\Sigma}(x_{1i})\right)  \;=\; \frac{\lr{1234(i{+}1)}\, \lr{123(i{-}1)i}}{\lr{1234(i{-}1)}\,\lr{123i(i{+}1)}}
$$
But now observe that $\Psi( \hat{y}_{\Sigma}(x_{1i}) )$ is exactly the exchange ratio 
$\hat{y}_{\overline{\Sigma}}(\bar{x}_{1i})$. Similarly, the exchange ratio for the mutable cluster variable 
$x_{2i}=\lr{134(i{-}1)i}$ in row $2$ of $\Sigma$ is
$$ \hat{y}_{\Sigma}(x_{2i}) \;=\;\frac{ \lr{1345(i{-}1)} \,\lr{134i(i{+}1)} \,\lr{13(i{-}2)(i{-}1)i}}{\lr{1345i} \,\lr{134(i{-}2)(i{-}1)}\, \lr{13(i{-}1)i(i{+}1)}} $$
and its image under $\Psi_5$ is
$$ \Psi_5\left(\hat{y}_{\Sigma}(x_{2i})\right) \;=\; \frac{ \lr{1234(i{-}1)} \,\lr{123i(i{+}1)} \,\lr{12(i{-}2)(i{-}1)i}}{\lr{1234i} \,\lr{123(i{-}2)(i{-}1)}\, \lr{12(i{-}1)i(i{+}1)}}
$$
which equals the exchange ratio $\hat{y}_{\overline{\Sigma}}(\bar{x}_{2i})$ just as before. The computations of the exchange ratios and their images for the mutable cluster variables in row $3$ and row $4$ are similar.  
We have now verified that properties (1) and (2) in \cref{def:quasi} are satisfied by our map $\Psi_5$.

The proof in the general case is a straightforward extension of the $m=5$ case.  We 
again use for $\Sigma$ the rectangles seed, where the rows are labeled from $1$ to $m$ from top to bottom, and the columns are labeled from $m+1$ to $n$ from left to right, and the cluster variable $x_{rc}$ in row $r$ and column $c$ is given by 
$$ x_{rc} \;=\;\begin{cases}
\lr{1,3,4,\dots,m{+}1{-}r,\;c{-}r{+}1,\dots,c{-}1,c} & 
\text{ for }1 \leq r \leq m-1\\
\lr{c{-}m{+}1,c{-}m{+}2,\dots, c{-}1,c} & \text{ for } r=m
\end{cases} $$
We again find that $\Psi_m$ fixes the cluster variables in rows $m-1$ and $m$,
and that 
$$ \Psi_m(\lr{1,3,4,\dots,m{+}1{-}r,\;c{-}r{+}1,\dots,c{-}1,c} \;\propto\; \lr{1,2,3,4,\dots,m{-}r,\;c{-}r{+}1,\dots,c{-}1,c} $$
It is straightforward to check that the exchange ratios satisfy 
property (2) of \cref{def:quasi}.
\end{proof}

\begin{remark}
Our unary star promotion map is essentially
the same as the $a=n-k$ ``splicing" map from \cite{splicing}, which was developed independently. The splicing map is also shown there to be a quasi-cluster homomorphism. 
\end{remark}

\subsection{Spurion promotion}\label{sec:spurion_proof}
In this section we show that  the map $\aProm_{sp}$ from \cref{def:upper_spurion_promotion} is a quasi-cluster homomorphism. Recall that $N'=\{1,2,7,8, \dots, n\}$.

\begin{theorem}
\label{thm:spurion-upper}
Unary spurion promotion $\aProm_{sp}$  
is a quasi-cluster homomorphism
from 
$\mathcal{A}(\Sigma)$ to 
$\mathcal{A}(\overline{\Sigma})$, where $\Sigma$ is a seed for ${\Gr}_{4,N'}$ and $\overline{\Sigma}$ is the seed for ${\Gr}_{4,n}$ with some variables frozen, shown in \cref{notation:sigma-bar}.
\end{theorem}

\begin{notation}
\label{notation:sigma-bar}
In the proof of \cref{thm:spurion-upper} we derive the following seed $\overline{\Sigma}$, a freezing of a seed for $\Gr_{4,n}$. The boxed variables are frozen; those in dashed boxes are mutable in the ${\Gr}_{4,n}$ seed. As in \cref{def:upper_spurion_promotion}, the chain polynomial $F_i$ is obtained by formally deleting the index $i$ from the expression $\lr{123 \shuf 456 \shuf 789}$. For indices after $9$ we use $\aa,\bb,\cc,\dots$ as shorthand for $10,11,12,\dots$ respectively.
\nopagebreak
\begin{center}
\begin{tikzpicture}[scale=1.2, every node/.style={minimum size=0.5cm}, every path/.style={->, thick}, node distance=0.5cm]
\def\cols{7}
\def\rows{4}
\def\ax{1.4}
\node (n01) at (\ax*5,0) {\dbox{$\langle1237\rangle$}};
\node (n11) at (\ax*5,-1) {$\langle1239\rangle$};
\node (n12) at (\ax*6,-1) {$\langle123\aa\rangle$};
\node (n13) at (\ax*7,-1) {$\langle123\bb\rangle$};
\node (n14) at (\ax*8,-1) {$\cdots$};
\node[anchor=west] (n15) at (\ax*8.6,-1) {\framebox{$\langle123n\rangle$}};
\node (n21) at (\ax*5,-2) {$\langle1289\rangle$};
\node (n22) at (\ax*6,-2) {$\langle129\aa\rangle$};
\node (n23) at (\ax*7,-2) {$\langle12\aa\bb\rangle$};
\node (n24) at (\ax*8,-2) {$\cdots$};
\node[anchor=west] (n25) at (\ax*8.6,-2) {\framebox{$\langle12(n{-}1)n\rangle$}};
\node (n31) at (\ax*5,-3) {$\langle1789\rangle$};
\node (n32) at (\ax*6,-3) {$\langle189\aa\rangle$};
\node (n33) at (\ax*7,-3) {$\langle19\aa\bb\rangle$};
\node (n34) at (\ax*8,-3) {$\cdots$};
\node[anchor=west] (n35) at (\ax*8.6,-3) {\framebox{$\langle1(n{-}2)(n{-}1)n\rangle$}};
\node (n41) at (\ax*5,-4) {\dbox{$\langle3789\rangle$}};
\node (n42) at (\ax*6,-4) {\framebox{$\langle789\aa\rangle$}};
\node (n43) at (\ax*7,-4) {\framebox{$\langle89\aa\bb\rangle$}};
\node (n44) at (\ax*8,-4) {$\cdots$};
\node[anchor=west] (n45) at (\ax*8.6,-4) {\framebox{$\langle(n{-}3)(n{-}2)(n{-}1)n\rangle$}};
\node (q8) at (\ax*4,-0.5) {\dbox{$F_8$}};
\node (q7) at (\ax*4,-1.5) {\dbox{$F_7$}};
\node (q3) at (\ax*4,-2.5) {\dbox{$F_3$}};
\node (q2) at (\ax*4,-3.5) {\dbox{$F_2$}};
\node (q1) at (\ax*3,-4) {\dbox{$F_1$}};
\node (q6) at (\ax*3,-3) {\dbox{$F_6$}};
\node (q5) at (\ax*3,-2) {\dbox{$F_5$}};
\node (q4) at (\ax*3,-1) {\dbox{$F_4$}};
\node (q9) at (\ax*3,0) {\dbox{$F_9$}};
\node (f1) at (\ax*2,-3) {\framebox{$\lr{1234}$}};
\node (f2) at (\ax*2,-4) {\framebox{$\lr{2345}$}};
\node (f3) at (\ax*3.9,-4.5) {\framebox{$\lr{3456}$}};
\node (f4) at (\ax*3.9,0.5) {\framebox{$\lr{4567}$}};
\node (f5) at (\ax*2,-0) {\framebox{$\lr{5678}$}};
\node (f6) at (\ax*2,-1) {\framebox{$\lr{6789}$}};
\node (mu) at (\ax*2,-2) {\dbox{$\lr{1236}$}};
\node[gray,circle,draw] (sigma) at (\ax*2.5,1) {$\overline{\Sigma}$};
\draw[->] (q8) -- (q7); 
\draw[->] (q7) -- (q3); 
\draw[->] (q3) -- (q2); 
\draw[->] (q2) -- (q1); 
\draw[->] (q1) -- (q6); 
\draw[->] (q6) -- (q5); 
\draw[->] (q5) -- (q4); 
\draw[->] (q4) -- (q9); 
\draw[->] (q9) -- (q8); 
\draw[->] (n11) -- (q8); 
\draw[->] (q7) -- (n11); 
\draw[->] (n21) -- (q7); 
\draw[->] (q3) -- (n21); 
\draw[->] (n31) -- (q3); 
\draw[->] (q2) -- (n31); 
\draw[->] (n41) -- (q2); 
\draw[->] (f3) -- (n41); 
\draw[->] (q1) -- (f3); 
\draw[->] (f2) -- (q1); 
\draw[->] (q6) -- (f2); 
\draw[->] (f1) -- (q6); 
\draw[->] (q5) -- (f1); 
\draw[->] (mu) -- (q5); 
\draw[->] (f6) -- (mu); 
\draw[->] (f1) -- (mu); 
\draw[->] (q4) -- (f6); 
\draw[->] (f5) -- (q4); 
\draw[->] (q9) -- (f5); 
\draw[->] (f4) -- (q9); 
\draw[->] (n01) -- (f4); 
\draw[->] (q8) -- (n01); 
\draw[->] (n01) -- (n11); 
\foreach \i in {1,2,3} {
\foreach \j [evaluate=\j as \k using int(\j+1)] in {1,...,4} { 
\draw[->] (n\i\j) -- (n\i\k); }}
\foreach \i [evaluate=\i as \k using int(\i+1)] in {1,...,3} {
\foreach \j in {1,...,3} {
\draw[->] (n\i\j) -- (n\k\j);}}
\foreach \i [evaluate=\i as \k using int(\i+1)] in {1,...,3} {
\foreach \j [evaluate=\j as \l using int(\j+1)] in {1,...,4} {
\draw[->] (n\k\l) -- (n\i\j);}}
\end{tikzpicture}
\end{center}
\end{notation}

\begin{notation}
\label{notation:sigma-spurion}
[Seed $\Sigma$ for $\Gr_{4,N'}$]
Let $\Sigma$ be the following rectangles seed for ${\Gr}_{4,N'}$. This is similar to the rectangles seeds appearing in proof for the star promotion.
\begin{center}
\begin{tikzpicture}[scale=1.2, every node/.style={minimum size=0.5cm}, every path/.style={->, thick}, node distance=0.5cm]
\def\ax{1.6}
\node (n00) at (\ax*1,0) {\framebox{$\langle1278\rangle$}};
\node (n11) at (\ax*1,-1) {$\langle1279\rangle$};
\node (n12) at (\ax*2,-1) {$\langle127\aa\rangle$};
\node (n13) at (\ax*3,-1) {$\langle127\bb\rangle$};
\node (n14) at (\ax*4,-1) {$\langle127\cc\rangle$};
\node (n15) at (\ax*5,-1) {$\cdots$};
\node[anchor=west] (n16) at (\ax*5.6,-1) {\framebox{$\langle127n\rangle$}};
\node (n21) at (\ax*1,-2) {$\langle1289\rangle$};
\node (n22) at (\ax*2,-2) {$\langle129\aa\rangle$};
\node (n23) at (\ax*3,-2) {$\langle12\aa\bb\rangle$};
\node (n24) at (\ax*4,-2) {$\langle12\bb\cc\rangle$};
\node (n25) at (\ax*5,-2) {$\cdots$};
\node[anchor=west] (n26) at (\ax*5.6,-2) {\framebox{$\langle12(n{-}1)n\rangle$}};
\node (n31) at (\ax*1,-3) {$\langle1789\rangle$};
\node (n32) at (\ax*2,-3) {$\langle189\aa\rangle$};
\node (n33) at (\ax*3,-3) {$\langle19\aa\bb\rangle$};
\node (n34) at (\ax*4,-3) {$\langle1\aa\bb\cc\rangle$};
\node (n35) at (\ax*5,-3) {$\cdots$};
\node[anchor=west] (n36) at (\ax*5.6,-3) {\framebox{$\langle1(n{-}2)(n{-}1)n\rangle$}};
\node (n41) at (\ax*1,-4) {\framebox{$\langle2789\rangle$}};
\node (n42) at (\ax*2,-4) {\framebox{$\langle789\aa\rangle$}};
\node (n43) at (\ax*3,-4) {\framebox{$\langle89\aa\bb\rangle$}};
\node (n44) at (\ax*4,-4) {\framebox{$\langle9\aa\bb\cc\rangle$}};
\node (n45) at (\ax*5,-4) {$\cdots$};
\node[anchor=west] (n46) at (\ax*5.6,-4) {\framebox{$\langle(n{-}3)(n{-}2)(n{-}1)n\rangle$}};
\node[gray,circle,draw] (sigma) at (0,-1.5) {$\Sigma$};
\foreach \i in {1,2,3} {
\foreach \j [evaluate=\j as \k using int(\j+1)] in {1,...,5} { 
\draw[->] (n\i\j) -- (n\i\k); }}
\foreach \i [evaluate=\i as \k using int(\i+1)] in {1,...,3} {
\foreach \j in {1,...,4} {
\draw[->] (n\i\j) -- (n\k\j);}}
\foreach \i [evaluate=\i as \k using int(\i+1)] in {1,...,3} {
\foreach \j [evaluate=\j as \l using int(\j+1)] in {1,...,5} {
\draw[->] (n\k\l) -- (n\i\j);}}
\draw[<-] (n11) -- (n00);
\end{tikzpicture}
\end{center}
\end{notation}

Note that the mutable variables of both $\overline{\Sigma}$ and $\Sigma$ are in a $3 \times (n-9)$ grid. This gives a natural bijection between the mutable variables of the two seeds, which we use in the next proposition.

\begin{lemma}
\label{lem:spurion-promoted-vars}
The mutable variables in $\overline{\Sigma}$ are proportional to the corresponding mutable variables in $\Psi_{sp}(\Sigma)$.
\end{lemma}

\begin{proof}
We apply the map $\Psi_{sp}$ to the variables in $\Sigma$, and substitute the vectors $2$, $7$, and $8$ using the formulas in \cref{rk:intersections}. 

The variables $\lr{19\aa\bb}$, $\lr{9\aa\bb\cc}$ and everything to their right in the third and fourth rows of both seeds are clearly fixed by $\Psi_{sp}$. The variable $\lr{89\aa\bb}$ is also fixed, since $8 \mapsto 8 - \tfrac{F_9}{F_8} 9$ but $\lr{99\aa\bb}=0$, and similarly $\lr{189\aa}$ is fixed as well. Also $\lr{129\aa}$ is fixed since $2 \mapsto 2 - \tfrac{F_1}{F_2}1$ and $\lr{119\aa}=0$, and same for all variables to its right of the form $\lr{12ij}$ in the second row of both seeds. The variable $\lr{1289}$ is also fixed by combining the two previous arguments for $12$ and $89$. Similarly, the variables $\lr{789\aa}$ and $\lr{1789}$ are fixed by combining the substitution rules $7 \mapsto 7-\tfrac{F_8}{F_7} 8+\tfrac{F_9}{F_7} 9 $ and~$8 \mapsto 8 - \tfrac{F_9}{F_8} 9$.

The remaining variables are $\lr{1278}$, $\lr{2789}$, and all variables of the form $\lr{127i}$ for $i \geq 9$ in the first row of~$\Sigma$. We compute the images of these variables under $\Psi_{sp}$ using \cref{rk:intersections}:
\begin{align*}
\Psi_{sp}\big(\lr{127i}\big) \;&=\; \left\langle1\left(2 - \tfrac{F_1}{F_2}1\right)\left(\tfrac{F_1}{F_7}1 - \tfrac{F_2}{F_7}2 + \tfrac{F_3}{F_7}3\right) i \right\rangle \;=\; \frac{F_3}{F_7} \lr{123i} \\
\Psi_{sp}\big(\lr{1278}\big) \;&=\; \left\langle1\left(2 - \tfrac{F_1}{F_2}1\right)\left(\tfrac{F_1}{F_7}1 - \tfrac{F_2}{F_7}2 + \tfrac{F_3}{F_7}3\right) \left(\tfrac{F_7}{F_8}7 - \tfrac{F_3}{F_8}3 + \tfrac{F_2}{F_8}2 - \tfrac{F_1}{F_8}1\right)\right\rangle \;=\; \frac{F_3}{F_8} \lr{1237} \\
\Psi_{sp}\big(\lr{2789}\big) \;&=\; \left\langle\left(\tfrac{F_3}{F_2}3 - \tfrac{F_7}{F_2}7 + \tfrac{F_8}{F_2}8 - \tfrac{F_9}{F_2}9 \right)\left(7-\tfrac{F_8}{F_7} 8+\tfrac{F_9}{F_7} 9\right)\left(8 - \tfrac{F_9}{F_8} 9\right)9\right\rangle \;=\; \frac{F_3}{F_2} \lr{3789}
\end{align*}
\end{proof}

\begin{lemma}
\label{lem:spurion-promoted-ratios}
Applying $\Psi_{sp}$ to the exchange ratio of a mutable variable in $\Sigma$ gives the exchange ratio of the corresponding mutable variable in $\overline{\Sigma}$.
\end{lemma}

\begin{proof}
If $n=9$ then there are no mutable variables in~$\Sigma$, so suppose $n \geq 10$. By the proof of the previous lemma, the images of mutable variables of the form $\lr{127i}$ in~$\Sigma$ are $\Psi_{sp}(\lr{127i}) = \tfrac{F_3}{F_7}\lr{123i} \propto \lr{123i}$ in~$\overline{\Sigma}$, since $F_3$ and $F_7$ are frozen. The other mutable variables in~$\Sigma$ map under~$\Psi_{sp}$ to identical mutable variables in~$\overline{\Sigma}$ without frozen factors. 

We verify that the exchange ratios of mutable variables in~$\Sigma$ map to the exchange ratio of their proportional mutable images. For the three variables in the leftmost column, applying \cref{def:exchage-ratios} and using the computations from the proof of the previous lemma, we obtain:
\begin{align*}
\Psi_{sp}\left(\hat{y}_{\Sigma}\big(\lr{1279}\big)\right) \;&=\; \Psi_{sp} \left(\frac{\lr{127\aa}\lr{1289}}{\lr{1278}\lr{129\aa}}\right) \;=\; \frac{\tfrac{F_3}{F_7}\lr{123\aa}\lr{1289}}{\tfrac{F_3}{F_8}\lr{1237}\lr{129A}} \;=\; \hat{y}_{\overline{\Sigma}}\big(\lr{1239}\big) \\
\Psi_{sp}\left(\hat{y}_{\Sigma}\big(\lr{1289}\big)\right) \;&=\; \Psi_{sp} \left(\frac{\lr{129\aa}\lr{1789}}{\lr{1279}\lr{189\aa}}\right) \;=\; \frac{\lr{129\aa}\lr{1789}}{\tfrac{F_3}{F_7}\lr{1239}\lr{189A}} \;=\; \hat{y}_{\overline{\Sigma}}\big(\lr{1289}\big) \\
\Psi_{sp}\left(\hat{y}_{\Sigma}\big(\lr{1789}\big)\right) \;&=\; \Psi_{sp} \left(\frac{\lr{189\aa}\lr{2789}}{\lr{1289}\lr{789\aa}}\right) \;=\; \frac{\lr{189\aa}\tfrac{F_3}{F_2}\lr{3789}}{\lr{1289}\lr{789A}} \;=\; \hat{y}_{\overline{\Sigma}}\big(\lr{1789}\big)
\end{align*}
For the remaining mutable variables of $\Sigma$ in the first row, $\lr{127i}$ for $9 < i < n$, we similarly verify:
$$ \Psi_{sp}\left(\hat{y}_{\Sigma}\big(\lr{127i}\big)\right) \;=\; \Psi_{sp}\left(\frac{\lr{127(i{+}1)}\lr{12(i{-}1)i}}{\lr{127(i{-}1)}\lr{12i(i{+}1)}}\right) \;=\; \frac{\tfrac{F_3}{F_7}\lr{123(i{+}1)}\lr{12(i{-}1)i}}{\tfrac{F_3}{F_7}\lr{123(i{-}1)}\lr{12i(i{+}1)}} \;=\; \hat{y}_{\overline{\Sigma}}\big(\lr{123i}\big) $$
Also for the remaining mutables in the second row of both quivers, $\lr{12(i{-}1)i}$ for $9 < i < n$:
\begin{align*} \Psi_{sp}\left(\hat{y}_{\Sigma}\big(\lr{12(i{-}1)i}\big)\right) \;&=\; \Psi_{sp}\left(\frac{\lr{12i(i{+}1)}\,\lr{1(i{-}2)(i{-}1)i}\,\lr{127(i{-}1)}}{\lr{127i}\,\lr{12(i{-}2)(i{-}1)}\,\lr{1(i{-}1)i(i{+}1)}}\right) \\
\;&=\; \frac{\lr{12i(i{+}1)}\,\lr{1(i{-}2)(i{-}1)i}\,\tfrac{F_3}{F_7}\lr{123(i{-}1)}}{\tfrac{F_3}{F_7}\lr{123i}\,\lr{12(i{-}2)(i{-}1)}\,\lr{1(i{-}1)i(i{+}1)}} \;=\; \hat{y}_{\overline{\Sigma}}\big(\lr{12(i{-}1)i}\big)
\end{align*}
Finally, the verification for mutable variables of the form $\lr{1(i{-}2)(i{-}1)i}$ for $9 < i < n$ appearing in the third row of both quivers is omitted, because these variables and their six neighbors are all fixed by~$\Psi_{sp}$, as shown in the proof of \cref{lem:spurion-promoted-vars}.
\end{proof}

\begin{lemma}
\label{lem:spurion-promotion-mutations} Let $\overline{\overline{\Sigma}}$ be the seed obtained from $\overline{\Sigma}$ by unfreezing the twelve dashed variables in \cref{notation:sigma-bar}. Then $\overline{\overline{\Sigma}}$ is a seed for ${\Gr}_{4,n}$.
\end{lemma}

\begin{proof}
We present a sequence of mutations that goes from a rectangles seed for $\Gr_{4,n}$ to $\overline{\overline{\Sigma}}$. We start from the following rectangles seed, similar the ones used above:
\nopagebreak
\begin{center}
\begin{tikzpicture}[scale=1.2, every node/.style={minimum size=0.5cm}, every path/.style={->, thick}, node distance=0.5cm]
\def\cols{7}
\def\rows{4}
\def\ax{1.44}
\node (n00) at (\ax*1,0) {\framebox{$\langle1234\rangle$}};
\node (n11) at (\ax*1,-1) {$\langle1235\rangle$};
\node (n12) at (\ax*2,-1) {$\langle1236\rangle$};
\node (n13) at (\ax*3,-1) {$\langle1237\rangle$};
\node (n14) at (\ax*4,-1) {$\langle1238\rangle$};
\node (n15) at (\ax*5,-1) {$\langle1239\rangle$};
\node (n16) at (\ax*6,-1) {$\cdots$};
\node[anchor=west] (n17) at (\ax*6.6,-1) {\framebox{$\langle123n\rangle$}};
\node (n21) at (\ax*1,-2) {$\langle1245\rangle$};
\node (n22) at (\ax*2,-2) {$\langle1256\rangle$};
\node (n23) at (\ax*3,-2) {$\langle1267\rangle$};
\node (n24) at (\ax*4,-2) {$\langle1278\rangle$};
\node (n25) at (\ax*5,-2) {$\langle1289\rangle$};
\node (n26) at (\ax*6,-2) {$\cdots$};
\node[anchor=west] (n27) at (\ax*6.6,-2) {\framebox{$\langle12(n{-}1)n\rangle$}};
\node (n31) at (\ax*1,-3) {$\langle1345\rangle$};
\node (n32) at (\ax*2,-3) {$\langle1456\rangle$};
\node (n33) at (\ax*3,-3) {$\langle1567\rangle$};
\node (n34) at (\ax*4,-3) {$\langle1678\rangle$};
\node (n35) at (\ax*5,-3) {$\langle1789\rangle$};
\node (n36) at (\ax*6,-3) {$\cdots$};
\node[anchor=west] (n37) at (\ax*6.6,-3) {\framebox{$\langle1(n{-}2)(n{-}1)n\rangle$}};
\node (n41) at (\ax*1,-4) {\framebox{$\langle2345\rangle$}};
\node (n42) at (\ax*2,-4) {\framebox{$\langle3456\rangle$}};
\node (n43) at (\ax*3,-4) {\framebox{$\langle4567\rangle$}};
\node (n44) at (\ax*4,-4) {\framebox{$\langle5678\rangle$}};
\node (n45) at (\ax*5,-4) {\framebox{$\langle6789\rangle$}};
\node (n46) at (\ax*6,-4) {$\cdots$};
\node[anchor=west] (n47) at (\ax*6.6,-4) {\framebox{$\langle(n{-}3)(n{-}2)(n{-}1)n\rangle$}};
\node[gray,circle,draw] (sigma) at (0,-1.5) {$\Tilde\Sigma$};
\foreach \i in {1,2,3} {
\foreach \j [evaluate=\j as \k using int(\j+1)] in {1,...,6} { 
\draw[->] (n\i\j) -- (n\i\k); }}
\foreach \i [evaluate=\i as \k using int(\i+1)] in {1,...,3} {
\foreach \j in {1,...,5} {
\draw[->] (n\i\j) -- (n\k\j);}}
\foreach \i [evaluate=\i as \k using int(\i+1)] in {1,...,3} {
\foreach \j [evaluate=\j as \l using int(\j+1)] in {1,...,6} {
\draw[->] (n\k\l) -- (n\i\j);}}
\draw[<-] (n11) -- (n00);
\end{tikzpicture}
\end{center}
We then apply the following sequence of $11$ mutations. For each mutation in this list, we give the variable before mutation, the expression for the variable after mutation using the exchange relation in the seed, and a polynomial-in-Pl\"uckers expression for the variable after mutation.   
\begin{enumerate}[topsep=1em, itemsep=1em, leftmargin=2em, labelsep=1em, label=(\alph*)]
\item 
$\displaystyle\lr{1678} \;\mapsto\;  \frac{\lr{1278}\lr{1567}\lr{6789}+\lr{1267}\lr{1789}\lr{5678}}{\lr{1678}} \;=\; \lr{12 \shuf 567 \shuf 789} $
\item 
$\displaystyle \lr{1567} \;\mapsto\; \frac{\lr{1456}\lr{12 \shuf 567 \shuf 789}+\lr{1256}\lr{1789}\lr{4567}}{\lr{1567}} \;=\; \lr{12 \shuf 456 \shuf 789} \;=\; F_3$
\item 
$\displaystyle \lr{1267} \;\mapsto\; \frac{\lr{1237}\lr{1256}\lr{6789}+\lr{1236}\lr{12 \shuf 567 \shuf 789}}{\lr{1267}} \;=\; \lr{123 \shuf 56 \shuf 789} \;=\; F_4$
\item 
$\displaystyle \lr{1256} \;\mapsto\;   \frac{\lr{1245}\lr{123 \shuf 56 \shuf 789}+\lr{1235}\lr{12 \shuf 456 \shuf 789}}{\lr{1256}} \;=\; \lr{123 \shuf 45 \shuf 789} \;=\; F_6 $
\item 
$\displaystyle \lr{1245} \;\mapsto\;  \frac{\lr{123 \shuf 45 \shuf 789}\lr{1456}+\lr{1345}\lr{12 \shuf 456 \shuf 789}}{\lr{1245}} \;=\; \lr{13 \shuf 456 \shuf 789} \;=\; F_2$
\item 
$\displaystyle \lr{1456} \;\mapsto\;  \frac{\lr{13 \shuf 456 \shuf 789}+\lr{1789}\lr{3456}}{\lr{1456}} \;=\; \lr{3789}$
\item 
$\displaystyle \lr{1345} \;\mapsto\;  \frac{\lr{123 \shuf 45 \shuf 789}\lr{3456}+\lr{13 \shuf 456 \shuf 789}\lr{2345}}{\lr{1345}} \;=\; \lr{23 \shuf 456 \shuf 789} \;=\; F_1$
\item 
$\displaystyle \lr{12 \shuf 567 \shuf 789} \;\mapsto\; \frac{\lr{1237}F_3\lr{5678}+F_4\lr{1278}\lr{4567}}{\lr{12 \shuf 567 \shuf 789}} \;=\; \lr{123 \shuf 456 \shuf 78} \;=\; F_9$
\item 
$\displaystyle \lr{1278} \;\mapsto\;  \;=\; \frac{\lr{1238}\lr{12 \shuf 456 \shuf 789}+\lr{1289}\lr{123 \shuf 456 \shuf 78}}{\lr{1278}} \;=\; \lr{123 \shuf 456 \shuf 89} \;=\; F_7$
\item 
$\displaystyle \lr{1238} \;\mapsto\;  \frac{\lr{1237}\lr{123 \shuf 456 \shuf 89}+\lr{1239}\lr{123 \shuf 456 \shuf 78}}{\lr{1238}} \;=\; \lr{123 \shuf 456 \shuf 79} \;=\; F_8$
\item 
$\displaystyle \lr{1235} \;\mapsto\;  \frac{\lr{123 \shuf 56 \shuf 789}\lr{1234}+\lr{1236}\lr{123 \shuf 45 \shuf 789}}{\lr{1235}} \;=\; \lr{123 \shuf 46 \shuf 789} \;=\; F_5$
\end{enumerate}
After applying these mutations, we end up with $\overline{\overline{\Sigma}}$.
\end{proof}

\begin{proof}[Proof of \cref{thm:spurion-upper}]

\cref{lem:spurion-promotion-mutations} verifies that $\overline{\Sigma}$ is a freezing of a seed for $\Gr_{4,n}$.
We verify that $\Psi_{sp} : \mathcal{A}(\Sigma) \to \mathcal{A}(\overline{\Sigma})$ is a quasi-cluster homomorphism as in \cref{def:quasi}. Our calculations in \cref{lem:spurion-promoted-vars} show that the image $\Psi_{sp}(x)$ of every mutable cluster variable $x$ in $\Sigma$ is proportional to a mutable variable $\bar{x}$ in $\overline{\Sigma}$, up to a Laurent monomial in the frozen variables $\{F_2,F_3,F_7,F_8\}$. This verifies condition (1): $\Psi_{sp}(x) \propto \bar{x}$. Our calculations in \cref{lem:spurion-promoted-ratios} verify condition (2), that $\Psi_{sp}$ maps the exchange ratio of $x$ to the exchange ratio of~$\bar{x}$, as required.
\end{proof}

\subsection{Chain-tree promotion} We show that $\Psi_{ch}: \CC(\Gr_{4,N'}) \;\rightarrow\; \CC(\Gr_{4,n})$ from \cref{def:chain_promotion} is a quasi-cluster homomorphism. Here $N'=\{1,2,3,\bb,\cc,\dd,\ldots,n\}$.

\begin{theorem}\label{thm:chain-tree-prom}
Chain-tree promotion $\Psi_{ch}$ is a quasi-cluster homomorphism
of cluster algebras from 
$\mathcal{A}(\Sigma)$ to 
$\mathcal{A}(\overline{\Sigma})$, where $\Sigma$ is a seed for ${\Gr}_{4,N'}$ and $\overline{\Sigma}$ is the seed for ${\Gr}_{4,n}$ with some variables frozen, shown in \cref{not:chain-sigma-bar}.
\end{theorem}

\begin{notation}
\label{not:chain-sigma-bar}
Let $\overline{\Sigma}$ be the following seed, which will be shown in the proof to be the image under $\Psi_{ch}$ of a seed for $\Gr_{4,N'}$ on the one hand, and obtained from a seed for $\Gr_{4,n}$ by freezing on the other hand. As before, the variables in dashed boxes are mutable in the $\Gr_{4,n}$ seed and frozen in $\overline{\Sigma}$. Arrows between such frozen variables are omitted, so many of them are isolated, because the full $\Gr_{4,n}$ quiver is complicated and nonplanar.
\begin{center}
\begin{tikzpicture}[scale=1.2, every node/.style={minimum size=0.5cm}, every path/.style={->, thick}, node distance=0.5cm]
\def\cols{7}
\def\rows{4}
\def\ax{1.6}
\node (n01) at (\ax*5,0) {\dbox{$\langle1237\rangle$}};
\node (n11) at (\ax*5,-1) {$X_1$};
\node (n12) at (\ax*6,-1) {$\langle123\dd\rangle$};
\node (n13) at (\ax*7,-1) {$\langle123\ee\rangle$};
\node (n14) at (\ax*8,-1) {$\cdots$};
\node[anchor=west] (n15) at (\ax*8.6,-1) {\framebox{$\langle123n\rangle$}};
\node (n21) at (\ax*5,-2) {$X_2$};
\node (n22) at (\ax*6,-2) {$\langle12\cc\dd\rangle$};
\node (n23) at (\ax*7,-2) {$\langle12\dd\ee\rangle$};
\node (n24) at (\ax*8,-2) {$\cdots$};
\node[anchor=west] (n25) at (\ax*8.6,-2) {\framebox{$\langle12(n{-}1)n\rangle$}};
\node (n31) at (\ax*5,-3) {$X_3$};
\node (n32) at (\ax*6,-3) {$\langle1\bb\cc\dd\rangle$};
\node (n33) at (\ax*7,-3) {$\langle1\cc\dd\ee\rangle$};
\node (n34) at (\ax*8,-3) {$\cdots$};
\node[anchor=west] (n35) at (\ax*8.6,-3) {\framebox{$\langle1(n{-}2)(n{-}1)n\rangle$}};
\node (n41) at (\ax*5,-4) {\dbox{$X_4$}};
\node (n42) at (\ax*6,-4) {\framebox{$\langle\aa\bb\cc\dd\rangle$}};
\node (n43) at (\ax*7,-4) {\framebox{$\langle\bb\cc\dd\ee\rangle$}};
\node (n44) at (\ax*8,-4) {$\cdots$};
\node[anchor=west] (n45) at (\ax*8.6,-4) {\framebox{$\langle(n{-}3)(n{-}2)(n{-}1)n\rangle$}};
\node (qc) at (\ax*6,0) {\dbox{$F_\cc$}};
\node (qb) at (\ax*4,-0.5) {\dbox{$F_\bb$}};
\node (q7) at (\ax*4,-1.5) {\dbox{$F_7$}};
\node (q3) at (\ax*4,-2.5) {\dbox{$F_3$}};
\node (q2) at (\ax*4,-3.5) {\dbox{$F_2$}};
\node[gray,circle,draw] (sigma) at (\ax*3,-1) {$\overline{\Sigma}$};
\draw[->] (n11) -- (qb); 
\draw[->] (q7) -- (n11); 
\draw[->] (n12) -- (qc); 
\draw[->] (qc) -- (n11); 
\draw[->] (n21) -- (q7); 
\draw[->] (q3) -- (n21); 
\draw[->] (n31) -- (q3); 
\draw[->] (q2) -- (n31); 
\draw[->] (n01) -- (n11); 
\foreach \i in {1,2,3} {
\foreach \j [evaluate=\j as \k using int(\j+1)] in {1,...,4} { 
\draw[->] (n\i\j) -- (n\i\k); }}
\foreach \i [evaluate=\i as \k using int(\i+1)] in {1,...,3} {
\foreach \j in {1,...,3} {
\draw[->] (n\i\j) -- (n\k\j);}}
\foreach \i [evaluate=\i as \k using int(\i+1)] in {1,...,3} {
\foreach \j [evaluate=\j as \l using int(\j+1)] in {1,...,4} {
\draw[->] (n\k\l) -- (n\i\j);}}
\end{tikzpicture}
\\ ~ \\
\begin{tikzpicture}[scale=1.2, every node/.style={minimum size=0.5cm}, every path/.style={->, thick}, node distance=0.5cm]
\def\ax{1.6}
\node at (\ax*3,-5) {\dbox{$F_1$}};
\node at (\ax*4,-5) {\dbox{$F_4$}};
\node at (\ax*5,-5) {\dbox{$F_5$}};
\node at (\ax*6,-5) {\dbox{$F_6$}};
\node at (\ax*7,-5) {\dbox{$F_8$}};
\node at (\ax*8,-5) {\dbox{$F_9$}};
\node at (\ax*9,-5) {\dbox{$F_\aa$}};
\node at (\ax*10,-5) {\dbox{$F_\dd$}};
\node at (\ax*3,-6) {\dbox{$\lr{1236}$}};
\node at (\ax*4,-6) {\dbox{$\lr{7\bb\cc\dd}$}};
\node at (\ax*5,-6) {\dbox{$\lr{8\bb\cc\dd}$}};
\node at (\ax*6.5,-6) {\dbox{$\lr{123 \shuf 456 \shuf 78}$}};
\node at (\ax*8.25,-6) {\dbox{$\lr{123 \shuf 456 \shuf 7\bb}$}};
\node at (\ax*10,-6) {\dbox{$\lr{123 \shuf 56 \shuf 789}$}};
\node at (\ax*3.0,-7) {\fbox{$\lr{1234}$}};
\node at (\ax*3.9,-7) {\fbox{$\lr{2345}$}};
\node at (\ax*4.8,-7) {\fbox{$\lr{3456}$}};
\node at (\ax*5.7,-7) {\fbox{$\lr{4567}$}};
\node at (\ax*6.6,-7) {\fbox{$\lr{5678}$}};
\node at (\ax*7.5,-7) {\fbox{$\lr{6789}$}};
\node at (\ax*8.4,-7) {\fbox{$\lr{789\aa}$}};
\node at (\ax*9.3,-7) {\fbox{$\lr{89\aa\bb}$}};
\node at (\ax*10.2,-7) {\fbox{$\lr{9\aa\bb\cc}$}};
\end{tikzpicture}
\end{center}
Recall from \cref{def:chain_promotion} that the $13$ variables $F_i$ are obtained by erasing the label $i$ from the expression $\lr{(123 \shuf 456) 7 (89\aa \shuf \bb\cc\dd)}$. Hence $F_i$ are cubic polynomials in Pl\"ucker coordinates. In addition $\overline{\Sigma}$ contains the following four quartic variables:
\begin{align*}
X_1 &\;=\; \lr{123B} F_B-\lr{1237}F_7 &&
X_2 \;=\; \lr{127B} F_7-\lr{123B}F_3 \\
X_3 &\;=\; \lr{137B} F_3-\lr{127B}F_2 &&
X_4 \;=\; \lr{237B} F_2-\lr{137B}F_1 \\
\end{align*}
\end{notation}

\begin{notation}
\label{not:chain-sigma}
Let $\Sigma$ be the following rectangles seed for ${\Gr}_{4,N'}$.
\begin{center}
\begin{tikzpicture}[scale=1.2, every node/.style={minimum size=0.5cm}, every path/.style={->, thick}, node distance=0.5cm]
\def\ax{1.6}
\node (n00) at (\ax*1,0) {\framebox{$\langle123\bb\rangle$}};
\node (n11) at (\ax*1,-1) {$\langle123\cc\rangle$};
\node (n12) at (\ax*2,-1) {$\langle123\dd\rangle$};
\node (n13) at (\ax*3,-1) {$\langle123\ee\rangle$};
\node (n14) at (\ax*4,-1) {$\cdots$};
\node[anchor=west] (n15) at (\ax*4.6,-1) {\framebox{$\langle123n\rangle$}};
\node (n21) at (\ax*1,-2) {$\langle12\bb\cc\rangle$};
\node (n22) at (\ax*2,-2) {$\langle12\cc\dd\rangle$};
\node (n23) at (\ax*3,-2) {$\langle12\dd\ee\rangle$};
\node (n24) at (\ax*4,-2) {$\cdots$};
\node[anchor=west] (n25) at (\ax*4.6,-2) {\framebox{$\langle12(n{-}1)n\rangle$}};
\node (n31) at (\ax*1,-3) {$\langle13\bb\cc\rangle$};
\node (n32) at (\ax*2,-3) {$\langle1\bb\cc\dd\rangle$};
\node (n33) at (\ax*3,-3) {$\langle1\cc\dd\ee\rangle$};
\node (n34) at (\ax*4,-3) {$\cdots$};
\node[anchor=west] (n35) at (\ax*4.6,-3) {\framebox{$\langle1(n{-}2)(n{-}1)n\rangle$}};
\node (n41) at (\ax*1,-4) {\framebox{$\langle23\bb\cc\rangle$}};
\node (n42) at (\ax*2,-4) {\framebox{$\langle3\bb\cc\dd\rangle$}};
\node (n43) at (\ax*3,-4) {\framebox{$\langle\bb\cc\dd\ee\rangle$}};
\node (n44) at (\ax*4,-4) {$\cdots$};
\node[anchor=west] (n45) at (\ax*4.6,-4) {\framebox{$\langle(n{-}3)(n{-}2)(n{-}1)n\rangle$}};
\node[gray,circle,draw] (sigma) at (0,-1.5) {$\Sigma$};
\foreach \i in {1,2,3} {
\foreach \j [evaluate=\j as \k using int(\j+1)] in {1,...,4} { 
\draw[->] (n\i\j) -- (n\i\k); }}
\foreach \i [evaluate=\i as \k using int(\i+1)] in {1,...,3} {
\foreach \j in {1,...,3} {
\draw[->] (n\i\j) -- (n\k\j);}}
\foreach \i [evaluate=\i as \k using int(\i+1)] in {1,...,3} {
\foreach \j [evaluate=\j as \l using int(\j+1)] in {1,...,4} {
\draw[->] (n\k\l) -- (n\i\j);}}
\draw[<-] (n11) -- (n00);
\end{tikzpicture}
\end{center}
\end{notation}

Note that the mutable variables of $\Sigma$ and $\overline{\Sigma}$ are arranged in grids of the same size, and thus there is a natural bijection between them. We use this bijection throughout.

\begin{lemma}
\label{lem:chain-promoted-vars}
The mutable variables in $\overline{\Sigma}$ are proportional to the corresponding mutable variables in $\Psi_{ch}(\Sigma)$.
\end{lemma}

\begin{proof}
We apply the map $\Psi_{ch}$ to the variables in $\Sigma$, and substitute the vectors $2,3,\bb$, and $\cc$ using the formulas in \cref{def:chain_promotion}.

The variables $\lr{123 i}$, with $i>\cc=12$ in the first row are fixed by $\Psi_{ch}$ as
$$\Psi_{ch}\big(\lr{123i}\big) \;=\; \left\langle1\left(2 - \tfrac{F_1}{F_2}1\right)\left(3 - \tfrac{F_1}{F_3}1+\tfrac{F_2}{F_3}2\right) i \right\rangle=\lr{123i}.$$
Analogously, the variables $\lr{12\cc\dd},\lr{1\bb\cc\dd}$, and $\lr{\bb\cc\dd\ee}$ and all the variables to their right in the respective rows are fixed by $\Psi_{sp}$. 
The variables $\lr{123\bb}$ and $\lr{3\bb\cc\dd}$ map under $\Psi_{sp}$ to:
\begin{align*}
&\Psi_{ch}\big(\lr{123\bb}\big) \;=\; \left\langle1\left(2 - \tfrac{F_1}{F_2}1\right)\left(3 - \tfrac{F_1}{F_3}1+\tfrac{F_2}{F_3}2\right) \left(\tfrac{F_7}{F_\bb}7-\tfrac{F_1}{F_\bb}1+\tfrac{F_2}{F_\bb}2-\tfrac{F_3}{F_\bb}3\right) \right\rangle \;=\; \frac{F_7}{F_\bb} \lr{1237} \\
&\Psi_{ch}\big(\lr{3\bb\cc\dd}\big) \;=\; \left\langle \left(\tfrac{F_7}{F_3}7-\tfrac{F_\bb}{F_3}\bb+\tfrac{F_\cc}{F_3}\cc-\tfrac{F_\dd}{F_3}\dd\right) \left(\bb-\tfrac{F_\cc}{F_\bb}\cc+\tfrac{F_\dd}{F_\bb}\dd\right)  \left(\cc-\tfrac{F_\dd}{F_\cc}\dd\right) \dd\right\rangle \;=\; \frac{F_7}{F_3} \lr{7\bb\cc\dd}
\end{align*}
The images of the variables $\lr{123C}, \lr{12BC}, \lr{13BC}$ and $\lr{23BC}$ under $\Psi_{ch}$ involve quartics cluster variables and are:
\begin{align*}
\Psi_{ch}\big(\lr{123\cc}\big) \;&=\; \left\langle1\left(2 - \tfrac{F_1}{F_2}1\right)\left(3 - \tfrac{F_1}{F_3}1+\tfrac{F_2}{F_3}2\right) \left(\tfrac{F_\bb}{F_C}\bb-\tfrac{F_7}{F_C}7+\tfrac{F_1}{F_C}1-\tfrac{F_2}{F_C}2+\tfrac{F_3}{F_C}3\right) \right\rangle \;=\; \\
 \;&=\ \left\langle 123 \left(\tfrac{F_\bb}{F_C}\bb-\tfrac{F_7}{F_C}7\right) \right\rangle=\frac{\lr{123\bb} F_\bb-\lr{1237}F_7}{F_C}=\frac{X_1}{F_C} \\
\Psi_{ch}\big(\lr{12\bb\cc}\big) \;&=\; \left\langle 1\left(2 - \tfrac{F_1}{F_2}1\right) \left(\tfrac{F_7}{F_\bb}7 - 1\tfrac{F_1}{F_\bb} + \tfrac{F_2}{F_\bb}2 -\tfrac{F_3}{F_\bb}3\right)  \left(\tfrac{F_\bb}{F_C}\bb-\tfrac{F_7}{F_C}7 + 1\tfrac{F_1}{F_C} - \tfrac{F_2}{F_C}2 +\tfrac{F_3}{F_C}3\right) \right\rangle \;=\; \\
\;&=\; \left\langle 12 \left(\tfrac{F_7}{F_\bb}7 -\tfrac{F_3}{F_\bb}3\right)  \left(\tfrac{F_\bb}{F_C}\bb\right)\right\rangle=\frac{\lr{127\bb} F_7-\lr{123\bb}F_3}{F_C}=\frac{X_2}{F_C}\\
\Psi_{ch}\big(\lr{13\bb\cc}\big) \;&=\; \left\langle 1\left(3 - \tfrac{F_1}{F_3}1 + \tfrac{F_2}{F_3}2\right) \left(\tfrac{F_7}{F_\bb}7 - 1\tfrac{F_1}{F_\bb} + \tfrac{F_2}{F_\bb}2 -\tfrac{F_3}{F_\bb}3\right)  \left(\tfrac{F_\bb}{F_C}\bb-\tfrac{F_7}{F_C}7 + \dots\right) \right\rangle \;=\; \\
\;&=\; \left\langle 1 \left(3+\tfrac{F_2}{F_3}2\right)  \left(\tfrac{F_7}{F_\bb}7\right) \left(\tfrac{F_\bb}{F_C}\bb\right)\right\rangle=\frac{\lr{137\bb} F_3+\lr{127\bb}F_2}{F_3F_C}F_7=X_3\frac{F_7}{F_3 F_C}\\
\Psi_{ch}\big(\lr{23\bb\cc}\big) \;&=\; \left\langle \left(2 - \tfrac{F_1}{F_2}1\right)\left(3 - \tfrac{F_1}{F_3}1 + \tfrac{F_2}{F_3}2\right) \left(\tfrac{F_7}{F_\bb}7 - 1\tfrac{F_1}{F_\bb} + \tfrac{F_2}{F_\bb}2 -\tfrac{F_3}{F_\bb}3\right)  \left(\tfrac{F_\bb}{F_C}\bb-\tfrac{F_7}{F_C}7 + \dots \right) \right\rangle \;=\; \\
\;&=\; \left\langle \left(2 - \tfrac{F_1}{F_2}1\right) 3 \left(\tfrac{F_7}{F_\bb}7\right) \left(\tfrac{F_\bb}{F_C}\bb\right)\right\rangle=\frac{\lr{237\bb} F_2-\lr{137\bb}F_1}{F_2 F_C}F_7=X_4\frac{F_7}{F_2 F_C}
\end{align*}

\end{proof}

\begin{lemma}
\label{lem:chain-tree-promoted-ratios}
Applying $\Psi_{ch}$ to the exchange ratio of a mutable variable in $\Sigma$ gives the exchange ratio of the corresponding mutable variable in $\overline{\Sigma}$.
\end{lemma}
\begin{proof}
 We verify that the exchange ratios of mutable variables in~$\Sigma$ map to the exchange ratio of their proportional mutable images. For the three mutable variables in the leftmost column, applying \cref{def:exchage-ratios} and using the computations from the proof of the previous lemma, we obtain:
\begin{align*}
&\Psi_{ch}\left(\hat{y}_{\Sigma}\big(\lr{123\cc}\big)\right) \;=\; \Psi_{ch} \left(\frac{\lr{123\bb}\lr{12\cc\dd}}{\lr{12\bb\cc}\lr{123\dd}}\right) \;=\; \frac{\tfrac{F_7}{F_\bb}\lr{1237}\lr{12\cc\dd}}{\tfrac{1}{F_\cc}X_2 \lr{123\dd}} \;=\; \hat{y}_{\overline{\Sigma}}\big(X_1\big) \\
&\Psi_{ch}\left(\hat{y}_{\Sigma}\big(\lr{12\bb\cc}\big)\right) \;=\; \Psi_{ch} \left(\frac{\lr{123\cc}\lr{1\bb\cc\dd}}{\lr{12\cc\dd}\lr{13\bb\cc}}\right) \;=\; \frac{X_1 \tfrac{1}{F_\cc}\lr{1\bb\cc\dd}}{\lr{12\cc\dd}X_3 \tfrac{F_7}{F_3 F_\cc}} \;=\; \hat{y}_{\overline{\Sigma}}\big(X_2\big) \\
&\Psi_{ch}\left(\hat{y}_{\Sigma}\big(\lr{13\bb\cc}\big)\right) \;=\; \Psi_{ch} \left(\frac{\lr{12\bb\cc}\lr{3\bb\cc\dd}}{\lr{1\bb\cc\dd}\lr{23\bb\cc}}\right) \;=\; \frac{X_2 \tfrac{1}{F_\cc} \lr{7\bb\cc\dd} \tfrac{F_7}{F_3}}{\lr{1\bb\cc\dd}X_4 \tfrac{F_7}{F_2 F_\cc}} \;=\; \hat{y}_{\overline{\Sigma}}\big(X_3\big)
\end{align*}
For the three mutable variables in the second column from the left:
\begin{align*}
&\Psi_{ch}\left(\hat{y}_{\Sigma}\big(\lr{123\dd}\big)\right) \;=\; \Psi_{ch} \left(\frac{\lr{123\cc}\lr{12\dd\ee}}{\lr{12\cc\dd}\lr{123\ee}}\right) \;=\; \frac{\tfrac{1}{F_\cc}X_1\lr{12\dd\ee}}{\lr{12\cc\dd} \lr{123\ee}} \;=\; \hat{y}_{\overline{\Sigma}}\big(\lr{123\dd}\big) \\
&\Psi_{ch}\left(\hat{y}_{\Sigma}\big(\lr{12\cc\dd}\big)\right) \;=\; \Psi_{ch} \left(\frac{\lr{12\bb\cc}\lr{123\dd}\lr{1\cc\dd\ee}}{\lr{123\cc}\lr{12\dd\ee} \lr{1\bb\cc\dd}}\right) \;=\; \frac{X_2 \tfrac{1}{F_\cc}\lr{123\dd} \lr{1\cc\dd\ee}}{X_1 \tfrac{1}{F_\cc} \lr{12\dd\ee} \lr{1\bb\cc\dd}} \;=\; \hat{y}_{\overline{\Sigma}}\big(\lr{12\cc\dd}\big) \\
&\Psi_{ch}\left(\hat{y}_{\Sigma}\big(\lr{1\bb\cc\dd}\big)\right) \;=\; \Psi_{ch} \left(\frac{\lr{12\cc\dd}\lr{13\bb\cc}\lr{\bb\cc\dd\ee}}{\lr{12\bb\cc}\lr{1\cc\dd\ee} \lr{3\bb\cc\dd}}\right) =\frac{\lr{12\cc\dd} X_3 \tfrac{F_7}{F_3 F_\cc} \lr{\bb\cc\dd\ee}}{X_2 \tfrac{1}{F_\cc} \lr{1\cc\dd\ee} \lr{7\bb\cc\dd} \tfrac{F_7}{F_3}} = \hat{y}_{\overline{\Sigma}}\big(\lr{1\bb\cc\dd}\big)
\end{align*}
Finally, the verification for the remaining mutable variables is omitted, because these variables and their six neighbors are all fixed by~$\Psi_{ch}$.
\end{proof}

\begin{lemma}\label{lem:chain-tree-promotion-mutations} The seed $\overline{\Sigma}$ is obtained from a seed for $\Gr_{4,n}$ by freezing some variables and deleting arrows between frozen vertices.
\end{lemma}

\begin{proof}
Similar to the proof of \cref{lem:spurion-promotion-mutations}, this is shown by an explicit sequence of mutations, starting from the same standard rectangles seed $\Tilde{\Sigma}$ for $\Gr_{4,n}$, and ending with a seed $\overline{\overline{\Sigma}}$, which gives rise to $\overline{\Sigma}$ by freezing. We only specify the sequence of mutated variables, which is sufficient to determine the required seed. As in \cref{sec:upper} above, we denote the rows of the rectangles seed $\tilde{\Sigma}$ of $\Gr_{4,n}$ by $1,\dots,4$, the columns by $5,\dots,n$, and the variable in row $r$ and column $c$ by~$x_{rc}$. For example $x_{28} = \lr{1278}$. The variable obtained from $x$ by a mutation is denoted $\bar{x}$. Apply a sequence of $50$ mutations to the seed~$\tilde{\Sigma}$, at the following $50$ variables:
\begin{align*}
& x_{3\cc}, x_{3\bb}, x_{3\aa}, x_{39}, x_{38}, x_{37}, x_{36}, x_{2\cc}, x_{2\bb}, x_{2\aa}, x_{29}, x_{28}, x_{27}, x_{26}, \bar x_{39}, \bar x_{38}, \bar x_{3\aa}, \bar{\bar{x}}_{39}, \bar x_{28}, \bar x_{29}, \\& x_{18}, x_{19}, x_{1\aa}, \bar x_{19}, \bar x_{2\aa}, \bar x_{18}, \bar{\bar{x}}_{19}, \bar x_{3\cc}, \bar x_{2\cc}, \bar x_{2\bb}, x_{1\cc}, \bar{\bar{x}}_{2\cc}, x_{1\bb}, \bar x_{1\cc}, \bar{\bar{\bar{x}}}_{39}, \bar{\bar{x}}_{38}, \bar{\bar{x}}_{29}, \bar{\bar{\bar{x}}}_{38}, \bar x_{27}, x_{15}, \\& \bar x_{26}, x_{25}, \bar x_{36}, x_{35}, \bar x_{25}, \bar x_{1\bb}, \bar{\bar{x}}_{3\cc}, \bar x_{3\bb}, \bar x_{37}, \bar{\bar{x}}_{25}
\end{align*}
A~tedious but straightforward verification shows that the resulting seed is an unfreezing of $\overline{\Sigma}$ as required. 
\end{proof}

\begin{proof}[Proof of \cref{thm:chain-tree-prom}]
\cref{lem:chain-tree-promotion-mutations} verifies that $\overline{\Sigma}$ is a freezing of a seed for $\Gr_{4,n}$.
We verify that $\Psi_{ch} : \mathcal{A}(\Sigma) \to \mathcal{A}(\overline{\Sigma})$ is a quasi-cluster homomorphism as in \cref{def:quasi}. Our calculations in \cref{lem:chain-promoted-vars} show that the image $\Psi_{ch}(x)$ of every mutable cluster variable $x$ in $\Sigma$ is proportional to a mutable variable $\bar{x}$ in $\overline{\Sigma}$. This verifies condition (1): $\Psi_{ch}(x) \propto \bar{x}$. Our calculations in \cref{lem:chain-tree-promoted-ratios} verify condition (2), that $\Psi_{ch}$ maps the exchange ratio of $x$ to the exchange ratio of~$\bar{x}$, as required.
\end{proof}

\subsection{Forest promotion}
Recall \cref{def:forest_promotion} of the map $\Psi_{fo}$.

\begin{theorem}
\label{thm:forest-prom}
Forest promotion $\Psi_{fo}$ is a quasi-cluster homomorphism
of cluster algebras from 
$\mathcal{A}(\Sigma)$ to 
$\mathcal{A}(\overline{\Sigma})$, where $\Sigma$ is a seed for ${\Gr}_{3,N'}$ and $\overline{\Sigma}$ is the seed for ${\Gr}_{3,n}$ with some variables frozen, shown in \cref{notation:sigma-forest}.
\end{theorem}

\begin{notation}
\label{notation:sigma-forest}
Let $\Sigma$ be the following rectangles seed for ${\Gr}_{3,N'}$ where $N' = \{1,\dots,n\} \setminus \{3,c\}$. Here and below, we denote $e:=d+1$, $f:=d+2$, $a':=a-1$, $a'':=a-2$, and similarly $n':=n-1$, $n'':=n-2$. 
\vspace{0.5cm}
\begin{center}
\begin{tikzpicture}[scale=1.2, every node/.style={minimum size=0.5cm}, every path/.style={->, thick}, node distance=0.5cm]
\def\cols{7}
\def\rows{4}
\def\ax{1.3}
\node (n00) at (\ax*1,0) {\framebox{$\langle124\rangle$}};
\node (n11) at (\ax*1,-1) {$\langle125\rangle$};
\node (n12) at (\ax*2,-1) {$\langle126\rangle$};
\node (n13) at (\ax*3,-1) {$\cdots$};
\node (n14) at (\ax*4,-1) {$\langle12a\rangle$};
\node (n15) at (\ax*5,-1) {$\langle12b\rangle$};
\node (n16) at (\ax*6,-1) {$\langle12d\rangle$};
\node (n17) at (\ax*7,-1) {$\langle12e\rangle$};
\node (n18) at (\ax*8,-1) {$\cdots$};
\node[anchor=west] (n19) at (\ax*8.6,-1) {\framebox{$\langle12n\rangle$}};
\node (n21) at (\ax*1,-2) {$\langle145\rangle$};
\node (n22) at (\ax*2,-2) {$\langle156\rangle$};
\node (n23) at (\ax*3,-2) {$\cdots$};
\node (n24) at (\ax*4,-2) {$\langle1a'a\rangle$};
\node (n25) at (\ax*5,-2) {$\langle1ab\rangle$};
\node (n26) at (\ax*6,-2) {$\langle1bd\rangle$};
\node (n27) at (\ax*7,-2) {$\langle1de\rangle$};
\node (n28) at (\ax*8,-2) {$\cdots$};
\node[anchor=west] (n29) at (\ax*8.6,-2) {\framebox{$\langle1n'n\rangle$}};
\node (n31) at (\ax*1,-3) {\framebox{$\langle245\rangle$}};
\node (n32) at (\ax*2,-3) {\framebox{$\langle456\rangle$}};
\node (n33) at (\ax*3,-3) {$\cdots$};
\node (n34) at (\ax*4,-3) {\framebox{$\langle a''a'a\rangle$}};
\node (n35) at (\ax*5,-3) {\framebox{$\langle a'ab\rangle$}};
\node (n36) at (\ax*6,-3) {\framebox{$\langle abd \rangle$}};
\node (n37) at (\ax*7,-3) {\framebox{$\langle bde \rangle$}};
\node (n38) at (\ax*8,-3) {$\cdots$};
\node[anchor=west] (n39) at (\ax*8.6,-3) {\framebox{$\langle n''n'n\rangle$}};
\node[gray,circle,draw] (sigma) at (\ax*9,0) {$\Sigma$};
\foreach \i in {1,2} {
\foreach \j [evaluate=\j as \k using int(\j+1)] in {1,...,8} { 
\draw[->] (n\i\j) -- (n\i\k); }}
\foreach \i [evaluate=\i as \k using int(\i+1)] in {1,...,2} {
\foreach \j in {1,...,8} {
\draw[->] (n\i\j) -- (n\k\j);}}
\foreach \i [evaluate=\i as \k using int(\i+1)] in {1,...,2} {
\foreach \j [evaluate=\j as \l using int(\j+1)] in {1,...,8} {
\draw[->] (n\k\l) -- (n\i\j);}}
\draw[<-] (n11) -- (n00);
\end{tikzpicture}
\end{center}
Consider also the rectangles seed $\tilde{\Sigma}$ for $\Gr_{3,n}$:
\vspace{0.1cm}
\begin{center}
\begin{tikzpicture}[scale=1.2, every node/.style={minimum size=0.5cm}, every path/.style={->, thick}, node distance=0.5cm]
\def\cols{7}
\def\rows{4}
\def\ax{1.3}
\node (n00) at (\ax*1,0) {\framebox{$\langle123\rangle$}};
\node (n11) at (\ax*1,-1) {$\langle124\rangle$};
\node (n12) at (\ax*2,-1) {$\langle125\rangle$};
\node (n13) at (\ax*3,-1) {$\cdots$};
\node (n14) at (\ax*4,-1) {$\langle12a\rangle$};
\node (n15) at (\ax*5,-1) {$\langle12b\rangle$};
\node (n16) at (\ax*6,-1) {$\langle12c\rangle$};
\node (n17) at (\ax*7,-1) {$\langle12d\rangle$};
\node (n18) at (\ax*8,-1) {$\langle12e\rangle$};
\node (n19) at (\ax*9,-1) {$\cdots$};
\node[anchor=west] (n110) at (\ax*9.6,-1) {\framebox{$\langle12n\rangle$}};
\node (n21) at (\ax*1,-2) {$\langle134\rangle$};
\node (n22) at (\ax*2,-2) {$\langle145\rangle$};
\node (n23) at (\ax*3,-2) {$\cdots$};
\node (n24) at (\ax*4,-2) {$\langle1a'a\rangle$};
\node (n25) at (\ax*5,-2) {$\langle1ab\rangle$};
\node (n26) at (\ax*6,-2) {$\langle1bc\rangle$};
\node (n27) at (\ax*7,-2) {$\langle1cd\rangle$};
\node (n28) at (\ax*8,-2) {$\langle1de\rangle$};
\node (n29) at (\ax*9,-2) {$\cdots$};
\node[anchor=west] (n210) at (\ax*9.6,-2) {\framebox{$\langle1n'n\rangle$}};
\node (n31) at (\ax*1,-3) {\framebox{$\langle234\rangle$}};
\node (n32) at (\ax*2,-3) {\framebox{$\langle345\rangle$}};
\node (n33) at (\ax*3,-3) {$\cdots$};
\node (n34) at (\ax*4,-3) {\framebox{$\langle a''a'a\rangle$}};
\node (n35) at (\ax*5,-3) {\framebox{$\langle a'ab\rangle$}};
\node (n36) at (\ax*6,-3) {\framebox{$\langle abc \rangle$}};
\node (n37) at (\ax*7,-3) {\framebox{$\langle bcd \rangle$}};
\node (n38) at (\ax*8,-3) {\framebox{$\langle cde \rangle$}};
\node (n39) at (\ax*9,-3) {$\cdots$};
\node[anchor=west] (n310) at (\ax*9.55,-3) {\framebox{$\langle n''n'n\rangle$}};
\node[gray,circle,draw] (sigma) at (\ax*9.5,0) {$\tilde\Sigma$};
\foreach \i in {1,2} {
\foreach \j [evaluate=\j as \k using int(\j+1)] in {1,...,9} { 
\draw[->] (n\i\j) -- (n\i\k); }}
\foreach \i [evaluate=\i as \k using int(\i+1)] in {1,...,2} {
\foreach \j in {1,...,9} {
\draw[->] (n\i\j) -- (n\k\j);}}
\foreach \i [evaluate=\i as \k using int(\i+1)] in {1,...,2} {
\foreach \j [evaluate=\j as \l using int(\j+1)] in {1,...,9} {
\draw[->] (n\k\l) -- (n\i\j);}}
\draw[<-] (n11) -- (n00);
\end{tikzpicture}
\end{center}
By applying mutations to $\tilde{\Sigma}$ at $\lr{1bc}$, $\lr{12b}$, and $\lr{12c}$, we obtain the following seed, and we then freeze four variables: $F_2=\lr{134}$, $F_3=\lr{124}$, $F_b=\lr{acd}$, and $F_c=\lr{abd}$.
\vspace{0.3cm}
\begin{center}
\begin{tikzpicture}[scale=1.2, every node/.style={minimum size=0.5cm}, every path/.style={->, thick}, node distance=0.5cm]
\def\cols{7}
\def\rows{4}
\def\ax{1.3}
\node (n00) at (\ax*1,0) {\framebox{$\langle123\rangle$}};
\node (n11) at (\ax*1,-1) {\dbox{$\langle124\rangle$}};
\node (n12) at (\ax*2,-1) {$\langle125\rangle$};
\node (n13) at (\ax*3,-1) {$\cdots$};
\node (n14) at (\ax*4,-1) {$\langle12a\rangle$};
\node (n15) at (\ax*5.5,-1) {$\langle 12 \shuf ab \shuf cd \rangle$};
\node (n16) at (\ax*5,0) {\dbox{$\langle acd\rangle$}};
\node (n17) at (\ax*6,0) {\dbox{$\langle abd\rangle$}};
\node (n18) at (\ax*8,-1) {$\langle12e\rangle$};
\node (n19) at (\ax*9,-1) {$\cdots$};
\node[anchor=west] (n110) at (\ax*9.6,-1) {\framebox{$\langle12n\rangle$}};
\node (n21) at (\ax*1,-2) {\dbox{$\langle134\rangle$}};
\node (n22) at (\ax*2,-2) {$\langle145\rangle$};
\node (n23) at (\ax*3,-2) {$\cdots$};
\node (n24) at (\ax*4,-2) {$\langle1a'a\rangle$};
\node (n25) at (\ax*5,-2) {$\langle1ab\rangle$};
\node (n26) at (\ax*7,-1) {$\lr{12d}$};
\node (n27) at (\ax*7,-2) {$\langle1cd\rangle$};
\node (n28) at (\ax*8,-2) {$\langle1de\rangle$};
\node (n29) at (\ax*9,-2) {$\cdots$};
\node[anchor=west] (n210) at (\ax*9.6,-2) {\framebox{$\langle1n'n\rangle$}};
\node (n31) at (\ax*1,-3) {\framebox{$\langle234\rangle$}};
\node (n32) at (\ax*2,-3) {\framebox{$\langle345\rangle$}};
\node (n33) at (\ax*3,-3) {$\cdots$};
\node (n34) at (\ax*4,-3) {\framebox{$\langle a''a'a\rangle$}};
\node (n35) at (\ax*5,-3) {\framebox{$\langle a'ab\rangle$}};
\node (n36) at (\ax*7.2,0) {\framebox{$\langle abc \rangle$}};
\node (n37) at (\ax*3.8,0) {\framebox{$\langle bcd \rangle$}};
\node (n38) at (\ax*8,-3) {\framebox{$\langle cde \rangle$}};
\node (n39) at (\ax*9,-3) {$\cdots$};
\node[anchor=west] (n310) at (\ax*9.55,-3) {\framebox{$\langle n''n'n\rangle$}};
\node[gray,circle,draw] (sigma) at (\ax*9.5,0) {$\overline{\Sigma}$};
\foreach \i in {1,2} {
\foreach \j [evaluate=\j as \k using int(\j+1)] in {1,2,3,4,8,9} { 
\draw[->] (n\i\j) -- (n\i\k); }}
\foreach \i [evaluate=\i as \k using int(\i+1)] in {1,...,2} {
\foreach \j in {1,2,3,4,5,8,9} {
\draw[->] (n\i\j) -- (n\k\j);}}
\foreach \i [evaluate=\i as \k using int(\i+1)] in {1,...,2} {
\foreach \j [evaluate=\j as \l using int(\j+1)] in {1,2,3,4,8,9} {
\draw[->] (n\k\l) -- (n\i\j);}}
\draw[<-] (n11) -- (n00);
\draw[->] (n15) -- (n26);
\draw[->] (n26) -- (n18);
\draw[->] (n25) -- (n27);
\draw[->] (n27) -- (n28);
\draw[->] (n28) -- (n26);
\draw[->] (n26) -- (n27);
\draw[->] (n27) -- (n15);
\draw[->] (n26) -- (n17);
\draw[->] (n17) -- (n15);
\draw[->] (n15) -- (n16);
\draw[->] (n16) -- (n14);
\draw[->] (n38) -- (n27);
\draw[->] (n38) -- (n27);
\draw[->] (n16) -- (n37);
\draw[->] (n36) -- (n17);
\end{tikzpicture}
\end{center}
\end{notation}

\begin{lemma}
\label{lem:forest-promoted-vars}
The images under $\Psi_{fo}$ of the mutable variables of $\Sigma$ are proportional to the mutable variables of $\overline{\Sigma}$.
\end{lemma}

\begin{proof}
We apply the map $\Psi_{fo}$ to the variables in $\Sigma$, and substitute the vectors $2$ and $b$, using the formulas in \cref{def:forest_promotion}. The variables $\lr{12 i}$ and $\lr{ab j}$, with $i \not = b$ and $j \not = 2$ are fixed by~$\Psi_{fo}$ since
$$\Psi_{fo}\big(\lr{12i}\big) \;=\; \left\langle1\left(2 - \tfrac{F_1}{F_2}1\right) i \right\rangle\;=\;\lr{12i}, \;\;\text{ and }\;\; \Psi_{fo}\big(\lr{abj}\big) \;=\; \left\langle a\left(b - \tfrac{F_a}{F_b}a\right) j \right\rangle\;=\;\lr{abj}.$$
Here we have denoted by $F_i$ the Pl\"ucker with labels $\{1,2,3,4\} \setminus \{i\}$ and by $F_x$ the Pl\"ucker with labels $\{a,b,c,d\} \setminus \{x\}$.
The only variables that change under $\Psi_{fo}$ are the following ones:
\begin{align*}
&\Psi_{fo}\big(\lr{12b}\big) \;=\; \left\langle1\left(2 - \tfrac{F_1}{F_2}1\right)\left(b - \tfrac{F_a}{F_b}a\right) \right\rangle \;=\; \frac{\lr{12b}F_b-\lr{12a}F_a}{F_b} \;=\; \frac{\langle 12 \shuf ab \shuf cd \rangle}{F_b} \\
&\Psi_{fo}\big(\lr{1bd}\big) \;=\; \left\langle 1 \left(\tfrac{F_c}{F_b}c-\tfrac{F_d}{F_b}d\right) d\right\rangle \;=\; \frac{F_c}{F_b} \lr{1cd} \\ &\Psi_{fo}\big(\lr{bde}\big) \;=\; \left\langle \left(\tfrac{F_c}{F_b}c-\tfrac{F_d}{F_b}d\right)de\right\rangle \;=\; \frac{F_c}{F_b} \lr{cde}\\
&\Psi_{fo}\big(\lr{245}\big) \;=\; \left\langle \left(\tfrac{F_3}{F_2}3-\tfrac{F_4}{F_2}4\right) 45\right\rangle \;=\; \frac{F_3}{F_2} \lr{345}
\end{align*}
Thus, we obtain frozen variables $F_2,F_3,F_b,F_c$, the quadratic $\langle 12 \shuf ab \shuf cd\rangle$ and the new Pl\"uckers $\lr{1cd},\lr{cde}$ and $\lr{345}$ in~$\overline{\Sigma}$, as required. Note that the mutable variables in $\Sigma$ are in bijection with mutable factors in their images in~$\overline{\Sigma}$.
\end{proof}

\begin{lemma}
\label{lem:forest-promoted-ratios}
The exchange ratios of the mutable variables in $\Sigma$ map under $\Psi_{fo}$ to the exchange ratios of the corresponding mutable variables in $\overline{\Sigma}$.
\end{lemma}

\begin{proof}
 We verify that the exchange ratios of mutable variables in~$\Sigma$ map to the exchange ratio of their proportional mutable images. Applying \cref{def:exchage-ratios} and using the computations from the proof of the previous lemma, we obtain for the following mutable variables of the first row:
\begin{align*}
\Psi_{fo}\left(\hat{y}_{\Sigma}\big(\lr{12a}\big)\right) \;&=\; \Psi_{fo} \left(\frac{\lr{12a'}\lr{1ab}}{\lr{12b}\lr{1a'a}}\right) \;=\; \frac{\lr{12a'} \lr{1ab}}{\tfrac{1}{F_b}\langle 12 \shuf ab \shuf cd \rangle \lr{1a'a}} \;=\; \hat{y}_{\overline{\Sigma}}\big(\lr{12a}\big) \\
\Psi_{fo}\left(\hat{y}_{\Sigma}\big(\lr{12b}\big)\right) \;&=\; \Psi_{fo} \left(\frac{\lr{12a}\lr{1bd}}{\lr{12d}\lr{1ab}}\right) \;=\; \frac{\lr{12a} \tfrac{F_c}{F_b}\lr{1cd}}{\lr{12d} \lr{1ab}} \;=\; \hat{y}_{\overline{\Sigma}}\big(\langle 12 \shuf ab \shuf cd \rangle\big) \\
\Psi_{fo}\left(\hat{y}_{\Sigma}\big(\lr{12d}\big)\right) \;&=\; \Psi_{fo} \left(\frac{\lr{12b}\lr{1de}}{\lr{1bd}\lr{12e}}\right) \;=\; \frac{\tfrac{1}{F_b}\langle 12 \shuf ab \shuf cd \rangle \lr{1de}}{\tfrac{F_c}{F_b}\lr{1cd}\lr{12e}} \;=\; \hat{y}_{\overline{\Sigma}}\big(\lr{12d}\big)
\end{align*}
For the following mutable variables of the second row we obtain:
\begin{align*}
\Psi_{fo}\left(\hat{y}_{\Sigma}\big(\lr{145}\big)\right) \;&=\; \Psi_{fo} \left(\frac{\lr{125}\lr{456}}{\lr{156}\lr{245}}\right) \;=\; \frac{\lr{125}\lr{456}}{\lr{156}\tfrac{F_3}{F_2}\lr{345}} \;=\; \hat{y}_{\overline{\Sigma}}\big(\lr{145}\big) \\
\Psi_{fo}\left(\hat{y}_{\Sigma}\big(\lr{1ab}\big)\right) \;&=\; \Psi_{fo} \left(\frac{\lr{12b}\lr{1a'a}\lr{abd}}{\lr{1bd}\lr{12a}\lr{a'ab}}\right) \;=\; \frac{\tfrac{1}{F_b}\langle 12 \shuf ab \shuf cd\rangle \lr{1a'a}\lr{abd}}{\tfrac{F_c}{F_b}\lr{1cd}\lr{12a}\lr{a'ab}} \;=\; 
\hat{y}_{\overline{\Sigma}}\big(\lr{1ab}\big) \\
\Psi_{fo}\left(\hat{y}_{\Sigma}\big(\lr{1bd}\big)\right) \;&=\; \Psi_{fo} \left(\frac{\lr{12d}\lr{bde}\lr{1ab}}{\lr{12b}\lr{1de}\lr{abd}}\right) \;=\; \frac{\lr{12d}\tfrac{F_c}{F_b}\lr{cde}\lr{1ab}}{\tfrac{1}{F_b}\langle 12 \shuf ab \shuf cd\rangle \lr{1de}\lr{abd}} \;=\; 
\hat{y}_{\overline{\Sigma}}\big(\lr{1cd}\big)\\
\Psi_{fo}\left(\hat{y}_{\Sigma}\big(\lr{1de}\big)\right) \;&=\; \Psi_{fo} \left(\frac{\lr{12e}\lr{1bd} \lr{def}}{\lr{12d}\lr{1ef} \lr{bde}}\right) \;=\; \frac{\lr{12e}\tfrac{F_c}{F_b}\lr{1cd} \lr{def}}{\lr{12d}\lr{1ef} \tfrac{F_c}{F_b}\lr{cde}} \;=\; \hat{y}_{\overline{\Sigma}}\big(\lr{12e}\big) 
\end{align*}
Finally, the verification for the remaining mutable variables is omitted, because these variables and their neighbors are all fixed by~$\Psi_{fo}$.
\end{proof}

\begin{proof}[Proof of \cref{thm:forest-prom}] \cref{notation:sigma-forest} shows that $\overline{\Sigma}$ is a seed for $\Gr_{4,n}$ with certain variables frozen. The properties of quasi-cluster homomorphism follow from Lemmas~\ref{lem:forest-promoted-vars} and~\ref{lem:forest-promoted-ratios}.
\end{proof}

\section{The (colored) operad of plabic tangles}\label{sec:categorical_pov}
In this section we describe the operad structure which underlies the relation between plabic tangles and promotion maps.

\subsection{The plabic operad}

\begin{definition}\label{def:plabic_operad} For $n \in \Z_{\ge 0}$ and $ \bd=(d_1, \dots, d_l) \in \Z_{\ge 0}^l$, we denote by $\mathcal{P}(n;\bd)$ the set of plabic tangles $(G, \bD=(D^{(i)})_{i \in [l]})$ where $G$ ranges over all possible cores with $n$ boundary vertices, and $\bD$ ranges over all possible blobs where $|D^{(i)}|=d_i$. We consider $\mathcal{P}(n;\bd)$ up to  local moves which do not touch the boundary vertices of the blobs.

The \emph{(colored) operad of plabic tangles} (or simply the \emph{plabic operad}) $\mathcal{P}$ consists of the following:
\begin{enumerate}
\item  the collection $(\mathcal{P}(n;\bd))_{n, \bd}$, where we refer to the values of $n$ and $d_i$ as the \emph{colors} of $\mathcal{P}$; 

\item a distinguished element $1=1_n$ in 
$\mathcal{P}(n;n)$ called the \emph{unit}, which consists of $n$ non-intersecting segments connecting the boundary vertices of the inner and outer disks;

\item \label{item:comp} for all positive integers
$n, \bd=(d_1,\dots,d_l),\bd'=( d'_1,\dots, d'_{l'})$ and $1\leq i \leq l$, a composition map 
\begin{align*}
\circ_i: 
&\mathcal{P}(n;\bd)\times \mathcal{P}(d_i;\bd')\to
\mathcal{P}(n;d_1,\ldots,d_{i-1},d'_1,\ldots,d'_{l'},d_{i+1},\dots, d_l),\\
&(G,\bD),(G',\bE)  \mapsto (G'',(D^{(1)},\ldots,{D^{(i-1)}},E^{(1)},\ldots,E^{(l')},{D^{(i+1)}},\ldots,D^{(l)}))
\end{align*}
obtained by inserting $(G',\bE)$ into the disk $D^{(i)}$ and gluing the attaching segments of $D^{(i)}$ to boundary vertices of $G'$;
\item \label{item:perm} for all positive integers $n, (d_1,\dots,d_l)$ and permutation $\sigma\in S_l$, a map
\[\sigma^*:\mathcal{P}(n;d_1,\ldots,d_l)\to  \mathcal{P}(n;d_{\sigma(1)},\ldots,d_{\sigma(l)})\] given by renumbering the inner disks.
\end{enumerate}
\end{definition}

\begin{figure}
\includegraphics[width=0.5\textwidth]{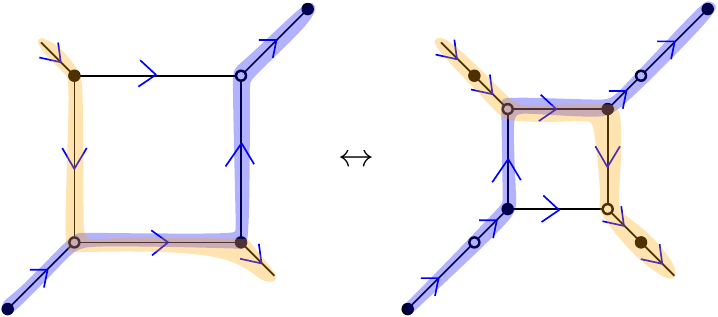}
\caption{We allow square moves on plabic graphs with a brushing only if the perfect orientation is a rotation of the orientation here and the paths either do not interact with the square or go between diagonally opposite corners. }
\label{fig:brushed-square-moves}
\end{figure}

\begin{remark}\label{rem:brushed}
Similarly, we can define the \emph{(colored) operad of brushed plabic tangles} $\widehat{P}=(\widehat{\mathcal{P}}(n;\bd)_{n,\bd})$ as the operad whose underlying set is the set of brushed plabic tangles. We consider brushed tangles up to expand-contract moves and \emph{brushed} square moves (shown in \cref{fig:brushed-square-moves}) that do not touch the boundary vertices of the blobs. The unit $\widehat{1}_n$ is the tangle $1_n$ with brushing given by the trivial paths each consisting of one boundary vertex.
Composition is given by inserting the brushed tangle $(G',\bE)$ into the blob $D=D^{(i)}$ of the brushed tangle $(G,\bD)$, where the brushing is obtained as follows.

Let $\OO=\OO^D$ be the  reverse perfect orientation of $G$ associated to the blob $D$, let $E$ 
be a blob of $(G',\bE)$, and let $\OO'$ be the reverse perfect orientation of $G'$ associated to $E$. Let $I \subset D$ 
be the set of sinks of $\OO'$.  
The orientation $\OO^E$ of the composed brushing is obtained by orienting $G'$ according to $\OO'$, reversing the orientation of all edges in the paths $\{P_u\}_{u \in I}$ in $\OO$, and orienting all other edges of $G$ according to $\OO$. 
For $u \in E$, the path $P_u$ of the composed brushing is the concatenation of the path $P'_u$ in $\OO'$ with the path $P_{v}$ in $\OO$ which ends at the starting point of $P'_u$. For any other blob $D' \in \bD$, we do not change the paths or the orientation. The symmetric group action is the natural lift of the action on the underlying tangle.
\end{remark}

Being an operad means that the following coherence axioms hold; the proof is straightforward.

\begin{observation}\label{obs:being_operad}
Let $(G,\bD)\in\mathcal{P}(N,\bd)$ and $(G',\mathbf{E})\in\mathcal{P}(d_i,\bd'),$ with $\bd=(d_1,\ldots,d_\ell)$ and $\bd'=(d'_1,\ldots,d'_{\ell'})$. Then we have the following.
\begin{itemize}
\item \emph{identity:} $(G,\bD)\circ_i 1_{d_i}=(G,\bD).$
\item \emph{commutativity}:
For $j\neq i\leq \ell$ and $(G'',\mathbf{H})\in\mathcal{P}(d_j,\bd'')$ 
\[[(G,\bD) \circ_i(G',\mathbf{E})]\circ_j(G'',\mathbf{H})=
[(G,\bD)\circ_j(G'',\mathbf{H})]\circ_i(G',\mathbf{E}).\]
\item \emph{associativity}: 
For $i\leq \ell,j\leq \ell',$  and $(G'',\mathbf{H})\in\mathcal{P}(d'_j,\bd'')$ 
\[[(G,\bD)\circ_i(G',\mathbf{E})]\circ_j(G'',\mathbf{H})
=(G,\bD)\circ_i[(G',\mathbf{E})\circ_j(G'',\mathbf{H})].\]
\item \emph{group action:} The maps $\sigma^\star$ for $\sigma\in S_\ell$ act on $\bigsqcup_{\bd=(d_1,\ldots,d_\ell)}\mathcal{P}(N,\bd).$
\item \emph{equivariance of compositions}: For $i\leq \ell$ and $\sigma\in S_\ell,$ 
\[
(G,(D^{(\sigma(1))},\ldots,D^{(\sigma(\ell))})\circ_{\sigma(i)}(G',\bE)
= 
\hat{\sigma}^*((G,\bD) \circ_i(G',\bE)),
\] where $\hat\sigma$ is the permutation on $([\ell]\setminus\{i\})\sqcup[\ell']$, identified with the total order obtained from $\{\sigma(1),\dots,\sigma(\ell)\}$ by 
removing $i$ and inserting the total order on $\{1,2,\dots,\ell'\}$ in its place.
An analogous statement holds for $\sigma\in S_{\ell'}$ acting on $\bE$. 
\end{itemize}

These properties also hold for the operad of brushed plabic tangles.
\end{observation}

A \emph{suboperad} is a subset of an operad which contains the units and is closed under composition and permutation maps $\sigma^*$.

\begin{proposition}\label{prop:sub_operad}
The collection of dominant ($m$-generically) solvable plabic tangles forms a suboperad $\Pds_m$ of $\mathcal{P}.$ In particular the composition of dominant solvable tangles is dominant solvable. 
The same holds for the suboperad $\widehat{\Pds_m}$ of brushed dominant solvable tangles. 
\end{proposition}
\begin{proof}
It is straightforward that $1_n$ is solvable and dominant, and it is also immediate that the action of the symmetric group restricts to an action on $\Pds_m.$ We are left to show that compositions of dominant solvable tangles are dominant  solvable tangles.

Consider the composed tangle ${(G, \bD)\circ_i (G',\bE)}$ where both constituent tangles are dominant solvable. The proof of \cref{cor:gluing} shows that the core of the composed tangle is solvable. The fact that the composed tangle is dominant follows from \cref{prop:dominant_solvable} and
\cref{rem:brushed}.

The proof for brushed tangles is a straightforward extension of the above argument.
\end{proof}
\subsection{Representations of plabic operads}
In this section we explain that the promotion maps from  \cref{sec:promotion} form \emph{representations} of the operad $\Pds_m$ of
dominant solvable tangles. 
\begin{definition}\label{def:in1_tangles_algebra_general}
An \emph{algebra $\mathcal{R}$ over an operad of plabic tangles} consists of: 
\begin{itemize}
\item a suboperad $\mathcal{P}'$ of $\mathcal{P}$;
\item a collection of $\CC$-algebras $\mathcal{R}(n)$ for every color which may appear in $\mathcal{P}'$, where 
$\mathcal{R}(0)=\CC$; 
\item  for every $(G,\bD)\in\mathcal{P}'$, a \emph{promotion map}
\begin{equation}\label{eq:promotion_tangle_alg_general}
\mathcal{R}(G,\bD): \bigotimes_{D \in \bD}\mathcal{R}(|D|)
    \to \mathcal{R}(n)\end{equation}
The spaces $\mathcal{R}(|D|)$ for $D\in\bD$ are the \emph{input spaces} or \emph{input components} of $\mathcal{R}(G,\bD),$ while $\mathcal{R}(n)$ is the output space.
\end{itemize}
This data is required to satisfy:
\begin{itemize}
\item \emph{normalization}: when $\ell=0$ the domain of the map \eqref{eq:promotion_tangle_alg_general} is $\CC$, and the map is the natural embedding $\mathbb{C}\to\mathcal{R}(n)$ which sends $1\mapsto 1$.
\item \emph{identity}: $\mathcal{R}(1_n)$ is the identity map $\mathcal{R}(n)\to\mathcal{R}(n).$
\item \emph{composition}:
$\mathcal{R}\left((G, \bD)\circ_i (G',\bE)\right)=\mathcal{R}(G,\bD)\circ_i\mathcal{R}(G',\bE)$, where the $\circ_i$ on the right corresponds to substituting the output of $\mathcal{R}(G',\bE)$ into the $i$th component of the domain of $\mathcal{R}(G,\bD)$.
\item \emph{symmetric group action}: for $\sigma\in S_{|\bD|}$, we have
$\sigma\cdot\mathcal{R}(G,\bD)=\mathcal{R}(\sigma^*(G,\bD))$, 
where $\sigma\cdot$ permutes the components of the domain $\bigotimes_{D \in \bD}\mathcal{R}(|D|)$ of $\mathcal{R}(G,\bD)$ according to the permutation $\sigma.$ \end{itemize}
 \end{definition}
 
 One can similarly define an algebra $\widehat{\mathcal{R}}$ over a suboperad of brushed plabic tangles.
 An important example is the following.
\begin{definition}\label{def:in1_tangles_algebra}
The \emph{algebra $\Ads_m$ of dominant 
($m$-generically) solvable plabic tangles}  
is the collection $\mathcal{R}(n)=\mathbb{C}(\Conf^\circ_{m,n}),$ for $n\geq m$, together with the promotion maps 
\eqref{eq:promotion_tangle_conf} 
\begin{equation*}\label{eq:promotion_tangle_alg_conf}\gProm=\gProm_{(G,\bD)}:\bigotimes_{D \in \bD}\CC(\Conf^\circ_{m, D})
    \to \CC(\Conf^\circ_{m, n})\end{equation*}
associated to each dominant ($m$-generically) solvable plabic tangle $(G,\bD).$

The \emph{algebra $\Abds$ of dominant 
($m$-generically) solvable brushed plabic tangles}  
is the collection $\widehat{\mathcal{R}(n)}=\mathbb{C}(\Gr_{m,n}),$ together with a promotion map \eqref{eq:promotion_tangle_alg} associated to each brushed dominant solvable 
plabic tangle $(G,\bD).$ 
\end{definition}
 
\begin{observation}\label{obs:being_operad_rep}
The algebra $\Ads_m$ 
is an algebra over the suboperad  $\Pds_m$. The algebra $\Abds$ 
is an algebra over the suboperad  $\widehat{\Pds_m}$. 
The only non-trivial property to verify is the behaviour under composition, which follows from \cref{prop:sub_operad}.
\end{observation}

We now sketch how to extend the above definitions to the operad of rank-$m$ regular tangles (which we do not require to be dominant or solvable).
It will be more convenient to focus on geometric promotion rather than algebraic promotion.

\begin{definition}
For a rank$-m$ regular  plabic tangle $(G,\bD)$ with
$\bD=(D^{(1)},\dots,D^{(\ell)})$
we define \emph{boundary $m$-VRCs} for $(G,\bD)$ as collections of elements 
$$(\bz,\bw^{(1)},\ldots,\bw^{(\ell)})=(z_1,\ldots,z_n;w^{(1)}_1,\ldots,w^{(\ell)}_{|D^{(\ell)}|})\in\Conf^\circ_{m, n} \times \Conf^\circ_{m, D^{(1)}} \times \dots \times \Conf^\circ_{m,D^{(\ell)}}$$
for which  there exists an $m$-VRC $[\bv, \bR] \in \mVRC_G$ which restricts to 
$(\bz,\bw^{(1)},\ldots,\bw^{(\ell)})$ on the boundaries of the inner and outer disks.

We similarly define \emph{boundary $m$-VRCs}  for a brushed tangle as collections
$$(\bz,\bw^{(1)},\ldots,\bw^{(\ell)})
\in\Gr_{m, n} \times \Gr_{m, D^{(1)}} \times \dots \times \Gr_{m,D^{(\ell)}}$$
where we use the brushing to lift the points to the Grassmannian.
\end{definition}

\begin{definition}
We associate to the tangle $(G,\bD)=(G, (D^{(1)},\dots,D^{(\ell)}))$ the space 
\[F_{(G,\bD)} \hookrightarrow \Conf^\circ_{m, n} \times \Conf^\circ_{m, D^{(1)}} \times \dots \times \Conf^\circ_{m,D^{(\ell)}},\] of all boundary $m$-VRCs for $(G,\bD).$ 
We similarly define ${F}_{\widehat{(G,\bD)}}$ for a brushed tangle by replacing everywhere $\Conf^\circ$ by $\Gr.$
\end{definition}
Note that $F_{1_n}$ is just the diagonal in $\Conf^\circ_{m,n}\times\Conf^\circ_{m,n}.$

When the tangle is generically solvable, $F_{(G,\bD)} $ is the graph of the 
geometric promotion map from \eqref{eq:promotion_tangle_geo}.
However, $F_{(G,\bD)}$ makes sense more generally, e.g. when we have
intersection number greater than $1$.  In \cref{sec:4mass} we will give an example of a promotion map corresponding to the \emph{$4$-mass box}, which has  intersection number $2$. Such a promotion can be thought of as a multivalued function or a function from some Galois cover of the Grassmannian.

In this more general setting we can define an analogue of the algebraic promotion map from \cref{eq:promotion_tangle_alg}, where the domain 
is the coordinate ring of $F_{(G,\bD)}$, which
can usually be written in the form 
\[R_{(G,\bD)}:=\left(\CC(\Conf^\circ_{m, D^{(1)}}) \otimes_{\CC} \dots \otimes_{\CC} 
    \CC(\Conf^\circ_{m, D^{(\ell)}})
    \otimes_{\CC} \CC(\Conf^\circ_{m, n})\right)/\mathcal{I}_{(G,\bD)}\]where $\mathcal{I}_{(G,\bD)}$ is the ideal defining (the closure of) $F_{(G,\bD)}.$

In this language, the composition of tangles corresponds geometrically to a projection of the fibered product.
Let the \emph{fibered product} $F_{(G,\bD)}\times_{\Conf^\circ_{m,D^{(i)}}}F_{(G',\bE)}$ denote the space of pairs of boundary $m$-VRCs of the form 
\[\left((\bz;\bw^{(1)},\ldots, \bw^{(i)},\ldots,\bw^{(\ell)}), (\bw^{(i)}; \bbu^{(1)},\dots, \bbu^{(r)})\right)\in F_{(G,\bD)}\times F_{(G',\bE)}.\]
Let 
$\mathrm{Pr}_{\Conf^\circ_{m,n}\times\prod_{j=1,j\neq i}^\ell\Conf^\circ_{m,D^{(j)}}\times\prod_{h=1}^\ell\Conf^\circ_{m,E^{(h)}}}$ be the projection map which maps this element to
\[(\bz; \bw^{(1)},\ldots, \widehat{\bw^{(i)}},\bbu^{(1)}, \ldots \bbu^{(r)},\ldots,\bw^{(\ell)}).\]
The following observation generalizes the composability part of \cref{prop:sub_operad}, and \cref{obs:being_operad_rep}.
\begin{observation}\label{obs:composition_fiber_prod} We have
\begin{equation*}
F_{(G,\bD)\circ_i (G',\bE)} = \mathrm{Pr}_{\Conf^\circ_{m,n}\times\prod_{j=1,j\neq i}^\ell\Conf^\circ_{m,D^{(j)}}\times\prod_{h=1}^\ell\Conf^\circ_{m,E^{(h)}}}\left(F_{(G,\bD)}\times_{\Conf^\circ_{m,D^{(i)}}}F_{(G',\bE)}\right).
\end{equation*}
(The statement for brushed tangles is analogous, but with Grassmannians replacing  configuration spaces everywhere.)
\end{observation}

The symmetric group $S_\ell$ acts on the $\ell$ input spaces in the fibered product by permuting them. It holds that for all $\bd=(d_1,\ldots,d_\ell)$, $n$, $(G,\bD)\in\mathcal{P}(n,\bd),$ and $\sigma\in S_\ell$, 
\[F_{\sigma^\star(G,\bD)}=\sigma\cdot F_{(G,\bD)},
\]where on the right hand side, 
 $\sigma$ acts by permuting the corresponding elements of the product. 
This leads to the following definition. 

\begin{definition}\label{def:general_operad_algebra}
Let $\mathcal{P}'$ be a suboperad of the plabic operad $\mathcal{P}$ or the brushed plabic operad $\widehat{\mathcal{P}}$.
A \emph{generalized topological representation} of $\mathcal{P}'$ is a collection of topological spaces $\mathcal{F}(n)$ for $n\geq 0$, as well as 
a topological space $\mathcal{F}{(G,\bD)} \subseteq \mathcal{F}(n)\times\bigtimes_{i=1}^\ell\mathcal{F}(|D^{(i)}|)$ for every $(G,\bD)=(G,(D^{(1)},\dots,D^{(\ell)}))\in\mathcal{P}',$\footnote{When $\ell=0$, we have  $\mathcal{F}(G)\subseteq\mathcal{F}(n).$} satisfying the following properties: 
\begin{itemize}
\item \emph{identity}: $\mathcal{F}(1_n)$ is the diagonal of $\mathcal{F}(n)\times\mathcal{F}(n)$;
\item \emph{composition}:
$$\mathcal{F}\Big((G,\bD)\circ_i (G',\bE)\Big) =\mathrm{Pr}_{\mathcal{F}(n)\times\prod_{j=1,j\neq i}^\ell\mathcal{F}(|D^{(j)}|)\times\prod_{h=1}^\ell\mathcal{F}(|E^{(h)}|)}\mathcal{F}{(G,\bD)}\times_{\mathcal{F}(|D^{(i)}|)}\mathcal{F}{(G',\bE)}.$$
\item \emph{symmetric group action}: for $\sigma\in S_\ell$
\[\sigma\cdot\mathcal{F}(G,\bD)=\mathcal{F}(\sigma^*(G,\bD)),\]
where $\sigma\cdot$ permutes the components $\mathcal{F}(|D^{(i)}|)$
and affects the projection accordingly. 
\end{itemize}

The topological representation obtained by setting $\mathcal{F}(n)=\Conf^\circ_{m,n}$ and $\mathcal{F}{(G,\bD)}={F}_{(G,\bD)}$ is called the 
$\Amp_m$-topological 
representation of the plabic operad. The topological representation obtained by setting ${\mathcal{F}}(n)=\Gr_{m,n}$ and ${\mathcal{F}}{\widehat{(G,\bD)}}={F}_{\widehat{(G,\bD)}}$ is called the \emph{$\Amp_m$-topological representation of the brushed plabic operad}. 
\end{definition}
We can obtain an \emph{algebraic representation} from the $\Amp_m-$topological representation of the brushed plabic operad one by moving to coordinate rings.

For fixed $m$, an interesting suboperad is the suboperad of (brushed) dominant tangles whose core $G$ has rank $k$ and dimension $km$.  One can show that the intersection number is multiplicative under composition. In \cref{sec:4mass} we discuss the induced map in the special case of the $4$-mass box.

The following conjecture generalizes the first item of \cref{conj:cluster1} to higher intersection numbers.
\begin{conjecture}\label{conj:poshigher}
Let $(G, \bD)$ be a plabic tangle of dimension at most $km$ which admits a brushing. Then there exists a brushing $\mcb$ and a choice of signs $\bsig$ such that the following holds for the associated space $F_{\widehat{(G,\bD)}}$:
\[F_{\widehat{(G,\bD)}}\cap\Gr^{\ge 0}_{m, n} \times \Gr_{m, D^{(1)}} \times \dots \times \Gr_{m,D^{(\ell)}}\subseteq \Gr^{\ge 0}_{m, n} \times \Gr^{\ge 0}_{m, D^{(1)}} \times \dots \times \Gr^{\ge 0}_{m,D^{(\ell)}}\]
\end{conjecture}

Similar operadic frameworks can be associated to the 
\emph{momentum amplituhedron} \cite{Damgaard:2019ztj} and the \emph{ABJM amplituhedra} \cite{He:2023rou}.  In fact the \emph{arc moves} of \cite{perlstein2025bcfw} are a special case of promotions in the ABJM amplituhedron setting.

\begin{remark} The plabic operad has additional structures, such as boundary maps which are obtained by erasing an edge of the core $G$ to obtain a new core $G' \subset G$. These boundary maps are compatible with representations of the operad: the space $F_{(G',\bD)}$ is a subspace of the closure $\overline{F_{(G,\bD)}}$ of $F_{(G,\bD)}$ in 
the Hausdorff topology.
We leave the study of this object  to a future work.
\end{remark}

\subsection{Modules}
We now briefly remark on modules associated to operads of plabic tangles. 
\begin{definition}
Let $\mathcal{R}$ be an algebra over a suboperad $\mathcal{P}'$ plabic tangles. A \emph{module} over $\mathcal{R}$ is an assignment of an $\mathcal{R}(n)$-module 
 $\mathcal{M}(n)$ for every color of $\mathcal{P}'$, as well as maps
 \[\mathcal{M}(G,\bD):\bigotimes_{D \in \bD}\mathcal{M}(|D|)
    \to \mathcal{M}(n).\]
    These maps satisfy analogous normalization, identity, composition and symmetric group action conditions  as in \cref{def:in1_tangles_algebra_general}, and in addition, 
    for all $f_i\in\mathcal{R}(|D^{(i)}|)$ and $m_i\in\mathcal{M}(|D^{(i)}|),$ with $i=1,\ldots,\ell,$
    \begin{equation}\label{eq:module_operad}
\mathcal{M}(G,\bD)(f_1m_1\otimes\cdots\otimes f_\ell m_\ell)
    = \mathcal{R}(G,\bD)(f_1\otimes\cdots\otimes f_\ell) \ \mathcal{M}(G,\bD)(m_1\otimes \cdots\otimes m_\ell),
    \end{equation}
 compatible with \eqref{eq:promotion_tangle_alg_general}. We can similarly define a brushed version.
\end{definition}

The next definition is motivated by the \emph{canonical form} of amplituhedron cells \cite{Arkani-Hamed:2012zlh,lam2015totally}.

\begin{definition}
Let $\mathcal{M}(n)$ to be the module of meromorphic forms on $\Gr_{m,n}.$
The \emph{canonical form} of a plabic graph is the meromorphic form $\bigwedge_{i=1}^d \frac{dx_i}{x_i},$ where $(x_i)_{i=1}^d$ is any coordinate system associated to a plabic graph, such as the edge weights (modulo gauge). Let $\widehat{(G,\bD)}$ be a brushed $m$-generically solvable plabic tangle, where $G$ is of type $(k,n)$.
We define 
        \begin{align*}
    \mathcal{M}(\widehat{G,\bD}): \mathcal{M}(|D^{(1)}|) \otimes \dots \otimes 
   &\mathcal{M}(|D^{(\ell)}|)
    \to \mathcal{M}(n)\\
     \text{ by }&(\omega_1\otimes\cdots\otimes\omega_\ell)\mapsto \Omega_G\wedge\gProm^*\omega_1\wedge\cdots\wedge\gProm^*\omega_\ell,\end{align*}
    where $\Omega_G$ is the push-forward to $\Gr_{m,n}$ of the canonical form of $G$ via the map $\comp_{Z,G}$ from \eqref{eq:Upsilon}  which associates a point of $\Pi_G \subset \Gr_{k,n}$ to a point of $\Gr_{m,n}$, and $\gProm^*=\gProm^*_{\widehat{(G,\bD)}}$ is the pull-back of the forms by the geometric promotion defined in \cref{def:brushed}.

    \end{definition}

    If $\ell=0$ then $\mathcal{M}(\widehat{G,\emptyset})$ is a meromorphic form on $\Gr_{m,n}$ which agrees with the canonical form of the positroid cell $S_G$ \cite{Arkani-Hamed:2012zlh,lam2015totally}. This definition can be extended beyond the dominant solvable case using the $\Amp_m$-topological
    representation of the brushed plabic operad.

    \begin{example}
    Consider the BCFW tangle described in \cref{def:bcfw_promotion}. The map it induces is
    \begin{align*}\omega_L\otimes\omega_R\mapsto\frac{d\lr{bcdn}}{d\lr{abcd}}\wedge\frac{d\lr{acdn}}{d\lr{abcd}}\wedge\frac{d\lr{abdn}}{d\lr{abcd}}\wedge\frac{d\lr{abcn}}{d\lr{abcd}}\wedge\gProm^*_{BCFW}\omega_L\wedge\gProm^*_{BCFW}\omega_R,\end{align*}
    where $\gProm^*_{BCFW}$ is the pullback along the map $\gProm_{BCFW}$ associated to the map $\aProm_{BCFW}$ of \cref{def:bcfw_promotion}. This map is well known from the study of the BCFW recursion \cite{Bai:2014cna}, where if $\omega_L,\omega_R$ are the canonical forms of the $\tZ$-images of the left and right blobs, then the left hand side of the above equation is the canonical form of the $\tZ$-image of the composed cell.
\end{example}

The above definitions extend  to e.g. the momentum and ABJM amplituhedron.

\section{Beyond intersection number one and cluster algebras}\label{sec:4mass}
While so far we have considered mainly the case of intersection number $1$ positroid varieties, one can also define promotions for more general plabic tangles, as was mentioned in \cref{sec:categorical_pov}. In this section we  provide an example of a plabic tangle whose core is the $4$-mass box (see \cref{fig:4mass_upper_spurion}) which indexes an intersection number $2$ positroid variety. While the promotion will not yield a quasi-cluster homomorphism between the coordinate rings of the corresponding Grassmannians, we will prove that this promotion has remarkable positivity properties (\cref{th:4mb_pos}).

\subsection{$4$-Mass Box Promotion}

We describe unary promotion by a $4$-mass box plabic graph as in \cref{fig:4mass_upper_spurion}. The core is a $4-$mass box cell with inner disks. Since the cell has $m$-intersection number $2$, where $m=4$, there are generically $2$ solutions for the vector configuration of the core.

\begin{remark}\label{rk:schubert_problem}
Note that vector-relation configurations for the $4$-mass box plabic graph, with $m=4$, encode the following classical Schubert problem: 
\emph{how many lines in $\PPP^3$ meet four general lines?}  Here e.g. the vectors at $1$ and $2$ span a line, as do the vectors at $3$ and $4$, the vectors at $5$ and $6$, and the vectors at $7$ and $8$. 
\end{remark}

In the following proposition, we describe how these $2$ solutions give rise to two promotion maps.

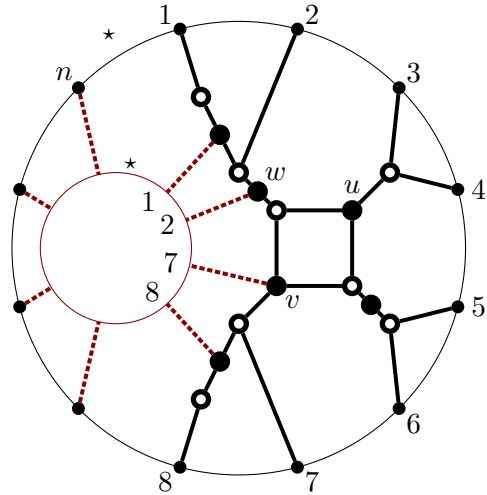
\begin{figure}[htbp]
\centering
\begin{tikzpicture}[scale=1,  
every node/.style={line width=2pt, circle, minimum size=2mm, draw, inner sep=0pt},
boundary/.style={line width=2pt, circle, minimum size=1mm, draw, inner sep=0pt, fill=black},
core/.style={line width=1.5pt},
blob/.style={line width=1.5pt,dash pattern=on 2pt off 1pt},
]
\node[boundary, label=above left:1] (b1) at (105:3) {};
\node[boundary, label=above right:2] (b2) at (75:3) {};
\node[boundary, label=above right:3] (b3) at (45:3) {};
\node[boundary, label=right:4] (b4) at (15:3) {};
\node[boundary, label=right:5] (b5) at (-15:3) {};
\node[boundary, label=below right:6] (b6) at (-45:3) {};
\node[boundary, label=below right:7] (b7) at (-75:3) {};
\node[boundary, label=below left:8] (b8) at (-105:3) {};
\node[boundary] (b10) at (-135:3) {};
\node[boundary] (b11) at (-165:3) {};
\node[boundary] (b12) at (165:3) {};
\node[boundary, label=above left:$n$] (bn) at (135:3) {};
\node[draw=none,label=above left:$\star$] (star) at (120:3) {};
\node (blob) at (180:1.625) {};
\node[fill=white] (ne) at (0.5,0.5) {};
\node[fill=black, label=above right:$w$] (ne2) at (0.25,0.75) {};
\node[fill=white] (ne1) at (0,1) {};
\node[fill=black] (ne3) at (-0.25,1.5) {};
\node[fill=white] (ne4) at (-0.5,2) {};
\node[fill=black, label=above:$u$] (nw) at (1.5,0.5) {};
\node[fill=white] (nw1) at (2,1) {};
\node[fill=black, label=below right:$v$] (se) at (0.5,-0.5) {};
\node[fill=white] (se1) at (0,-1) {};
\node[fill=black] (se3) at (-0.25,-1.5) {};
\node[fill=white] (se4) at (-0.5,-2) {};
\node[fill=white] (sw) at (1.5,-0.5) {};
\node[fill=black] (sw2) at (1.75,-0.75) {};
\node[fill=white] (sw1) at (2,-1) {};
\draw[core] (b3)--(nw1)--(b4) (nw1)--(nw)--(ne)--(se)--(sw)--(nw) (sw)--(sw2)--(sw1) (b5)--(sw1)--(b6) (ne)--(ne2)--(ne1) (b2)--(ne1)--(ne3)--(ne4)--(b1) (se)--(se1) (b7)--(se1)--(se3)--(se4)--(b8);
\draw[darkred,blob] (ne3)--(blob) (ne2)--(blob) (se)--(blob) (se3)--(blob) (b10)--(blob) (b11)--(blob) (b12)--(blob) (bn)--(blob);
\draw[darkred,fill=white] (blob) circle[radius=1];
\draw[black] (0,0) circle[radius=3];
\node[draw=none] at ($(blob) + (55:0.75)$) {$1$};
\node[draw=none] at ($(blob) + (25:0.75)$) {$2$};
\node[draw=none] at ($(blob) + (-15:0.75)$) {$7$};
\node[draw=none] at ($(blob) + (-50:0.75)$) {$8$};
\node[draw=none] (star2) at ($(blob) + (80:1.125)$) {$\star$};
\end{tikzpicture}
    \caption{A brushed plabic tangle whose core is the 4-mass box plabic graph, which has $4$-intersection number 2. As usual, the brushing is indicated by the labels of the blob's vertices.}
    \label{fig:4mass_upper_spurion}
\end{figure}

\begin{proposition}\label{prop:4mass-promotion}
The plabic tangle from \cref{fig:4mass_upper_spurion} gives rise to two corresponding promotion maps $\Psi_{\pm}:\C(\widehat\Gr_{4,N'})\to \C(\widehat\Gr_{4,n})[\sqrt{\Delta}],$
where $N'=\{1,2,7,8,\ldots,n\}$ and
\begin{equation}\label{eq:4mb_promotion}
    z_2 \mapsto \frac{12 \shuf 56 X_{\pm}}{\langle 1 56 X_{\pm}\rangle  }, \quad z_7 \mapsto X_{\pm} := z_7 + \alpha_{\pm} z_8,
\end{equation}
\begin{equation}\label{eq:alpha}
    \alpha_{\pm}= \frac{-B \pm \sqrt{\Delta}}{2 A}, \quad \Delta:=B^2-4AC
\end{equation}
\begin{equation}\label{eq:ABC}
    A=Y_{88}, \, \, B= Y_{78}+Y_{87}, \, \, C= Y_{77},
\end{equation}
where $Y_{ij}:=\langle 12i \shuf 43 \shuf 56j\rangle$ are cluster variables for $\Gr_{4,n}$ when $i,j \in \{7,\ldots,n\}$.
\end{proposition}

\begin{remark}\label{rmk:branching_loc}
Geometrically, referring to the terminology of \cref{sec:categorical_pov}, if we denote the brushed tangle inducing this promotion by $\widehat{(G,D)},$ then ${F}_{\widehat{(G,D)}}$ can be identified with the graph of a map from the Galois double cover of $\Gr_{4,n}$ defined by extending $\CC(\widehat{\Gr_{4,n}})$ by $\sqrt{\Delta},$ to $\Gr_{4,N'},$ and we may view the promotion maps representing that graph.
 Interestingly, this map from the Galois cover can be presented as different global branches, $\aProm_\pm.$ We conjecture that this phenomenon of different well-defined global branches on the positive Grassmannian, holds for all cells of intersection number greater than one, which is equivalent to saying that certain expressions which describe branching loci of the promotion map from the Galois cover are positive on the positive Grassmannian. 
\end{remark}

\begin{proof}[Proof of \cref{prop:4mass-promotion}]
The white vertex attached to $7,8$ enforces, up to gauge, a linear relation $v=z_7+\alpha z_8$, with $\alpha$ a coefficient to be determined. Analogously, we have $u=z_3+\beta z_4$. The vector $w$ associated to the black vertex between the two white vertices close to $1,2$ has to belong to $(12)$ and $(uv)$. These two subspaces intersect trivially if they are generic, therefore they need to satisfy: $\langle 12uv \rangle=0$. Analogously, considering the vector $w$ on the black vertex between the two white vertices close to $5,6$, we need $\langle 56uv \rangle=0$. The two conditions give the following constraints on $\alpha, \beta$:
\begin{equation}
    \langle ij37 \rangle +\alpha \langle ij38 \rangle + \beta \langle ij47 \rangle+ \alpha \beta \langle ij48 \rangle =0, \quad \text{ for }\{i,j\} \in \{\{1,2\},\{5,6\}\}.
\end{equation}
If we solve  the equation where $\{i,j\}=\{1,2\}$ for $\beta$ and substitute it in the other equation we get the following quadratic equation for $\alpha$:
\begin{equation}
    A \alpha^2 +B \alpha +C=0,
\end{equation}
where $A,B,C$ are as in \cref{eq:ABC}. Hence there are $2$ solutions for $v$, which we denote as $X_{\pm}=z_7+\alpha_{\pm}z_8$, with $\alpha_{\pm}$ as in \cref{eq:alpha}. 

Since $\langle 12uv \rangle=0$, the vectors  satisfy a linear relation of the kind: $\gamma z_1 + z_2 +\delta u + \epsilon v=0$, where the coefficients are non-zero. Moreover, a vector $w$ in the intersection of $(12)$ and $(uv)$ is $w=\gamma z_1 + z_2=-\delta u - \epsilon v$. By wedging the linear relation with $5 \wedge 6 \wedge v$ we get:
\begin{equation}
    0=\gamma \langle 1 56 v \rangle + \langle 2 56 v \rangle +\delta \langle u 56 v \rangle+\epsilon \langle v 56 v \rangle=\gamma \langle 1 56 v \rangle + \langle 2 56 v \rangle,
\end{equation}
because $\langle 56 uv \rangle=0$ and $v \wedge v=0$. Therefore, we can solve for $\gamma$ and we can express $w$ as:
\begin{equation}
    w=-\frac{\langle 2 56 v \rangle}{\langle 1 56 v \rangle}z_1+z_2=\frac{12 \shuf 56v}{\langle 1 56v \rangle}.
\end{equation}
\end{proof}
\begin{proposition}[\cite{Arkani-Hamed:2019rds}]\label{prop:delta_pos}
   The function $\Delta$ in \cref{eq:alpha} is positive on $\Gr^{>0}_{4,n}$.
\end{proposition}

\begin{remark}
The positivity of $\Delta$ in \cref{prop:delta_pos} is related to the existence of real solutions to the Schubert problem in \cref{rk:schubert_problem}, which have been the subject of recent study, see e.g. \cite{karprealSchubert}.
\end{remark}

\begin{remark}
$\Delta$ is \emph{nontrivially positive}, according to the definition in \cite{Arkani-Hamed:2019rds}. Whereas, all the following proofs will rely on sums or products of functions which are  \emph{Laurent positive}.
\end{remark}

\begin{lemma}\label{lem:positivity_expr}
Let $i \in \{9,\ldots,n\}$. Then $2A\lr{127i}-B\lr{128i}$ is positive on $\Gr^{>0}_{4,n}$, where $A$ and $B$ are the functions in \cref{eq:ABC}.
\end{lemma}
\begin{proof}
Let us write $2A\lr{127i}-B\lr{128i}$ as
\begin{align*}
& \big(\lr{1283}\lr{4568}-\lr{1284}\lr{3568}\big)\lr{127i} + \big(\lr{1283}\lr{4568}-\lr{1284}\lr{3568}\big)\lr{127i}-\\
& \big(\lr{1273}\lr{4568}-\lr{1274}\lr{3568}+\lr{1283}\lr{4567}-\lr{1284}\lr{3567}\big)\lr{128i}=\\
& \lr{4568}\big(\lr{1283}\lr{127i}-\lr{1273}\lr{128i}\big)+\lr{1283}\big( \lr{4568}\lr{127i}-\lr{4567}\lr{128i}\big)+\\
& \lr{3568}\big(\lr{1274}\lr{128i}-\lr{1284}\lr{127i}\big)+\lr{1284}\big(\lr{3567}\lr{128i}-\lr{3568}\lr{127i}\big)=\\
&  -\lr{4568}\lr{1278}\lr{123i}+\lr{1283}\langle456 \shuf 78 \shuf 12i \rangle+\lr{3568}\lr{1278}\lr{124i}+\lr{1284}\langle356 \shuf 87 \shuf 12i\rangle=\\
&  \lr{1278}\langle 12i \shuf 43 \shuf 568\rangle+
  \langle (128\shuf 34)65(78\shuf12i) \rangle
\end{align*}
The claim follows from the fact that $\langle 12i\shuf 43 \shuf 568\rangle$ and $\langle (128\shuf 34)65(78\shuf12i) \rangle$ are positive on $\Gr^{>0}_{4,n}$.
\end{proof}

\begin{proposition}\label{conj:4mb_positivity}
Let $X_{\pm}$ be as in \cref{eq:4mb_promotion}. Then on the positive Grassmannian $\Gr^{>0}_{4,n}$, the following holds:
\begin{itemize}
\item[i)] $\lr{12X_{\pm}i}$ is positive, for all $i \in \{9,\ldots,n\}$;
\item[ii)] $\langle 789 \shuf 21 \shuf 56X_{s} \rangle$ and $\lr{156X_s}$ are both positive (negative), if $s=+$ ($s=-$). 
\end{itemize}
\end{proposition}

\begin{proof}
\textbf{Proof of i).}
\begin{align*}
\lr{12X_{\pm}i}=\lr{127i}+\alpha_\pm \lr{128i}=\lr{127i}+\frac{-B \pm \sqrt{\Delta}}{2A} \lr{128i}=\frac{1}{2A} \left(2A\lr{127i}+(-B \pm \sqrt{\Delta}) \lr{128i}\right)
\end{align*}
since $A>0$ and $\lr{12X_{+}i}>\lr{12X_{-}i}$, it is enough to show that $2A \lr{12X_{-}i}>0 $:
\begin{align*}
2A\lr{127i}-(B+\sqrt{\Delta}) \lr{128i}>0 \Leftrightarrow 2A\lr{127i}-B\lr{128i}>\sqrt{\Delta}\lr{128i}
\end{align*}
By \cref{lem:positivity_expr}  $2A\lr{127i}-B\lr{128i}>0$, we can square both sides and we obtain:
\begin{align*}
& 4A^2\lr{127i}^2+B^2\lr{128i}^2-4AB\lr{127i}\lr{128i}>(B^2-4AC)\lr{128i}^2 \Leftrightarrow \\
& A\lr{127i}^2-B \lr{127i}\lr{128i}+C\lr{128i}^2>0 \\
& A\lr{127i}^2-B \lr{127i}\lr{128i}+C\lr{128i}^2=\big(\lr{1283}\lr{4568}-\lr{1284}\lr{3568}\big)\lr{127i}^2- \\
& \big(\lr{1273}\lr{4568}-\lr{1274}\lr{3568}+\lr{1283}\lr{4567}-\lr{1284}\lr{3567}\big)\lr{127i}\lr{128i}+\\
& \big(\lr{1273}\lr{4567}-\lr{1274}\lr{3567} \big)\lr{128i}^2=\\
& \lr{3568}\lr{127i}\big(\lr{1274}\lr{128i}-\lr{1284}\lr{127i}\big)+\lr{3567}\lr{128i} \big(\lr{1284}\lr{127i}-\lr{1274}\lr{128i} \big)+\\
& \lr{1283}\lr{127i} \big( \lr{4568}\lr{127i}-\lr{4567}\lr{128i}\big)+\lr{1273}\lr{128i} \big( \lr{4567}\lr{128i}-\lr{4568}\lr{127i}\big)=\\
& \lr{1278}\lr{124i}\big(\lr{3568}\lr{127i}-\lr{3567}\lr{128i}\big)+\langle 456 \shuf 87 \shuf 12i\rangle \big(\lr{1283}\lr{127i}- \lr{1273}\lr{128i}\big)=\\
& \lr{1278}\lr{124i} \langle 356\shuf 78 \shuf 12i\rangle-\langle 456 \shuf 87 \shuf 12i\rangle \lr{1278}\lr{123i}=\\
&  \lr{1278} \langle (12i\shuf 34)65(78\shuf12i) \rangle
\end{align*}
Since $\langle (12i\shuf 34)65(78\shuf12i) \rangle$ is positive on $\Gr^{>0}_{4,n}$, this completes the proof. 

\textbf{Proof of ii)}
We now prove that $\lr{156X_s}$ is positive if $s=+$ and negative if $s=-$. The two cases, once multiplied by $2A>0$ and rearranged read:
\begin{align*}
\sqrt{\Delta} \lr{1568}>y \quad \mbox{and} \quad \sqrt{\Delta} \lr{1568}<-y, \qquad y:=B \lr{1568}-2A \lr{1567}.
\end{align*}
Therefore we can square both sides to obtain: 
\begin{align*}
\Delta \lr{1568}^2>y^2 \Leftrightarrow A \lr{1567}^2-B\lr{1567}\lr{1568}+C\lr{1568}^2<0.
\end{align*}
By using steps similar to the ones in the proof of i) we can show that 
\begin{align*}
A \lr{1567}^2-B\lr{1567}\lr{1568}+C\lr{1568}^2=-\lr{5678} 
\lr{(156\shuf34)21(78\shuf156)}
\end{align*}
where $\langle (156\shuf34)21(78\shuf156)\rangle$ can be checked to be positive on $\Gr^{>0}_{4,n}$.

We now conclude the proof by showing that $\langle 789 \shuf 21 \shuf 56 X_s \rangle$ is positive for $s=+$ and negative for $s=-$. Let us define $Y_i:=\langle 789 \shuf 21 \shuf 56i \rangle$, then $\langle 789 \shuf 21 \shuf 56 X_s \rangle=Y_7+\alpha_s Y_8$. Following the same reasoning for the proof above for $\langle 156 X_s \rangle$, we arrive that the statement to prove is equivalent to:
\begin{align*}
A Y_7^2-BY_7Y_8+CY_8^2<0.
\end{align*}
By using similar steps as in the previous proof and identities of the type
\begin{align*}
\lr{127i}Y_8-\lr{128i}Y_7=\lr{1278}\langle 12i \shuf 65 \shuf 789 \rangle, \quad \mbox{and} \quad  \lr{567i}Y_8-\lr{568i}Y_7=\lr{5678}Y_i,
\end{align*}
one can show that
\begin{align*}
A Y_7^2-BY_7Y_8+CY_8^2=-\lr{1278}\lr{5678} 
\lr{(789 \shuf 12) 56 \shuf 34 \shuf 12 (56 \shuf 789)}
\end{align*}
where the quartic polynomial $\lr{(789 \shuf 12) 56 \shuf 34 \shuf 12 (56 \shuf 789)}$ is positive on $\Gr^{>0}_{4,n}$.
\end{proof}

\begin{lemma}\label{lem:4bm_pos}
Consider the rectangle seed $\Sigma$ for $\Gr_{4,N'}$ in \cref{notation:sigma-spurion}. Let $x$ be a cluster variable in $\Sigma$, then its image $\Psi_{\pm}(x)$ under the $4$-mass box promotion is a positive function on~$\Gr^{>0}_{4,n}$.
\end{lemma}
\begin{proof}
Let us denote $F_2:=\lr{156X_\pm}$ and $F_1:=\lr{256X_\pm}$. Then:
\begin{align*}
\Psi_{\pm}\big(\lr{1278}\big) \;&=\; \left\langle1\left(2 - \tfrac{F_1}{F_2}1\right)\left(7+\alpha_{\pm} 8\right) 8\right\rangle \;=\; \lr{1278}\\
\Psi_{\pm}\big(\lr{127i}\big) \;&=\; \left\langle 1 2\left(7+\alpha_{\pm} 8\right)i \right\rangle \;=\; \lr{12X_\pm i}, \quad i \in \{9,\ldots,n\}  \\
\Psi_{\pm}\big(\lr{78ij}\big) \;&=\; \left\langle \left(7+\alpha_{\pm} 8\right)8ij\right\rangle \;=\; \lr{78ij}, \quad 2 \not \in \{i,j\} \\
\Psi_{\pm}\big(\lr{2789}\big) \;&=\; \left\langle\left(2 - \tfrac{F_1}{F_2}1\right)\left(7+\alpha_{\pm} 8\right) 89\right\rangle \;=\; \frac{F_2\lr{2789}-F_1\lr{1789}}{F_2}=\frac{\langle 789 \shuf 21 \shuf 56X_\pm\rangle}{\lr{156X_\pm}}
\end{align*}
The rest of the variables in $\Sigma$ are fixed by $\Psi_\pm$. It then follows from \cref{conj:4mb_positivity} that $\Psi_\pm(x)$ is positive on $\Gr^{>0}_{4,n}$ for every variable $x$ in $\Sigma$.
\end{proof}

\begin{theorem}[$4$-Mass Box Cluster Positivity]\label{th:4mb_pos}
The $4$-mass box promotion map $\Psi_{\pm}$ preserves positivity of cluster variables for $\Gr_{4,N'}$, i.e. if $x$ is a cluster variable for $\Gr_{4,N'}$, then $\Psi_{\pm}(x)$ is a positive function on $\Gr^{>0}_{4,n}$.
\end{theorem}
\begin{proof}
If $x$ is a cluster variable for $\Gr_{4,N'}$, then by Laurent positivity of cluster algebras, $x$ can be expressed as a Laurent polynomial of the cluster variables in the initial seed $\Sigma$ with positive coefficients. By \cref{lem:4bm_pos}, the images under $\Psi_\pm$ of all such  cluster variables in the initial seed $\Sigma$ will be positive functions on $\Gr^{>0}_{4,n}$, hence $\Psi_\pm(x)$ will also be a positive function on $\Gr^{>0}_{4,n}$. 
\end{proof}

\begin{remark} Positivity is non-trivial already for Pl\"ucker coordinates. If $x=\lr{127n}$, then 
$$ \Psi_{\pm}(x)=\lr{127n}+\alpha_{\pm}\lr{128n},$$
is not a sum of positive terms, as $\lr{127n}$ and $\lr{128n}$ are positive and $\alpha_{\pm}$ is negative on the positive Grassmannian.
\end{remark}

We obtain the following as a corollary of \cref{th:4mb_pos}.
\begin{corollary}
Let us consider the geometric promotion $\gProm_\pm: \Gr_{4,n}\dashedrightarrow \Gr_{4,N'}$, of which $\aProm_\pm$ is the pullback. Then $\gProm_\pm(\Gr^{> 0}_{4,n}) \subseteq \Gr^{> 0}_{4,N'}$.
\end{corollary}

In general, we make the following conjecture for positivity properties of higher intersection number tangles.

\begin{conjecture}
Recall \cref{conj:poshigher} and \cref{rmk:branching_loc}. Suppose $(G, \bD)$ is a tangle which admits a brushing, and suppose that $\Pi_G$ has dimension $km$ and $\mint(G)>0$. Then we can find a brushing and signs for which \cref{conj:poshigher} holds.

Moreover, $F_{\widehat{(G,\bD)}}\cap\Gr^{>0}_{m, n} \times \Gr^{>0}_{m, D^{(1)}} \times \dots \times \Gr^{>0}_{m,D^{(\ell)}}$ is the disjoint union of a fixed number graphs of functions from $\Gr_{m,n}^{>0}$ to $\prod_{D \in \bD}\Gr^{>0}_{m, D}$. 
\end{conjecture}

\begin{remark}[Positivity Phenomenon in Scattering Amplitudes]
\label{rk:physics_scattering_2}
Scattering amplitudes in planar $\mathcal{N}=4$ SYM are conjectured to be regular on $\Gr^{>0}_{4,n}$. In other words, they have singularities on hypersurfaces $f(\textbf{z})=0$, where $f(\textbf{z})>0$ when $\textbf{z} \in \Gr^{>0}_{4,n}$. We call this the `positivity phenomenon' in planar $\mathcal{N}=4$ SYM. When $f(\textbf{z})$ is a rational function, $f(\textbf{z})$ is conjectured to be a cluster variable for $\Gr_{4,n}$, hence 
it is positive on $\Gr^{>0}_{4,n}$.
In our framework, we expect $f(\textbf{z})$ to be the image of a cluster promotion map, and hence a cluster variable and positive on the positive Grassmannian. 
This was proved when $f(\textbf{z})$ is a singularity of a Yangian invariant $\mathcal{Y}_G$ associated to a BCFW cell $S_G$ \cite{even2023cluster} in relation to \emph{cluster adjacency conjecture}.

When $f(\textbf{z})$ is an algebraic function, there is no general mechanism that would explain positivity. In special cases, there are some results in the frameworks of \emph{limiting rays} and \emph{cluster algebraic functions} \cite{CanakciSchiffler, Drummond:2019cxm, Arkani-Hamed:2019rds}. We expect that algebraic singularities of Yangian invariants $\mathcal{Y}_G$, with $\mbox{IN}_4(G)>1$, are \emph{also} images of promotion maps. In the case of the four mass box promotion, \cref{th:4mb_pos} shows that the map preserves positivity. We expect this to hold in general for promotion maps coming from a plabic graph $G$ with $\mbox{IN}_m(G)>1$ (\cref{conj:poshigher}). Hence promotion maps should give a new mechanism to generate  algebraic functions that are positive on $\Gr^{\geq 0}_{m,n}$ and explain positivity of algebraic singularities of Yangian invariants. Noting that all algebraic singularities of scattering amplitudes in planar $\mathcal{N}=4$ SYM that have been found so far are singularities of Yangian invariants, understanding promotion maps should be a stepping-stone towards proving the positivity phenomenon. 
    
\end{remark}

\begin{remark}
While the promotion in \cref{prop:4mass-promotion} is not a quasi-cluster homomorphism between Grassmannian cluster algebras, it does take relations between cluster variables, e.g. the 3-term exchange relations, into highly non trivial relations between functions involving $\sqrt\Delta.$ 
This hints at a more general algebraic structure governing the singularities of scattering amplitudes, beyond cluster algebras, where algebraic promotions would serve as morphisms of the associated category. 
\end{remark}

\appendix
\section{Background on plabic graphs}\label{app:plabic}

In \cite{postnikov}, 
Postnikov 
classified the cells of the positive Grassmannian, 
using equivalence classes of \emph{reduced plabic graphs}, 
and \emph{decorated (or affine) permutations}. 
Here we review some of this technology, following \cite{FW7},
\cite{postnikov}, and \cite{PSWv1}.

\begin{definition}[Plabic graphs]
   \label{def:plabic}
A {\it planar bicolored graph} (or ``plabic graph'')
is a planar graph $G$ properly embedded into a closed disk, such that 
		each internal vertex is colored black or white;
		each internal vertex is connected by 
		a path to some boundary vertex;
		there are  vertices lying on the 
		boundary of the disk labeled $1,\dots, n$
		for some positive $n$;
and each of the boundary vertices is incident to a single 
		edge.
See Figure \ref{G25} for an example.
\end{definition}

We note that segments of the boundary of the disk between boundary vertices of $G$ are not edges of $G$. The faces of $G$ are the connected components of the complement of $G$ in the disk. We use the notation $E(G)$ for the edges of $G$ and $F(G)$ for the faces.

\begin{figure}[h]
\centering
\includegraphics[height=1in]{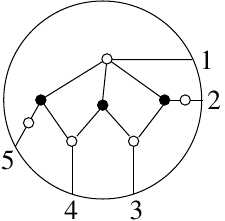}
\caption{A plabic graph. }
\label{G25}
\end{figure}

If $G$ has an internal leaf which is incident to a boundary 
vertex, 
we call this a  
\emph{lollipop}.  
\emph{We will require that our plabic graphs are \emph{leafless}, i.e. that 
they have no internal leaves
except for lollipops.}

There is a natural set of local transformations (moves) of plabic graphs:

(M1) \emph{Square move} (or \emph{urban renewal}).  If a plabic graph has a square formed by
four trivalent vertices whose colors alternate,
then we can switch the
colors of these four vertices.

(M2) \emph{Contracting/expanding a vertex}.
Two adjacent internal vertices of the same color can be merged or unmerged.

(M3) \emph{Middle vertex insertion/removal}.
We can remove/add degree $2$ vertices.

See \cref{fig:M1} for depictions of these three moves.

\begin{figure}[h]
\centering
\includegraphics[height=.5in]{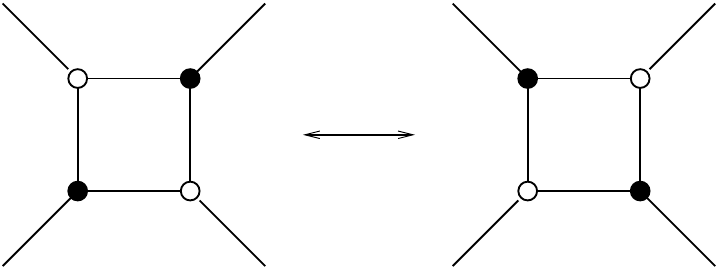}
\hspace{.3in}
\raisebox{6pt}{\includegraphics[height=.3in]{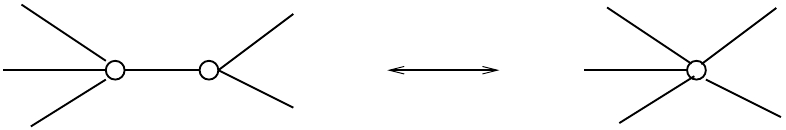}}
\hspace{.3in}
\raisebox{16pt}{\includegraphics[height=.07in]{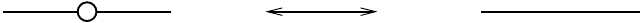}}
\caption{
	Local moves (M1), (M2), (M3) on plabic graphs. }
\label{fig:M1}
\end{figure}
\begin{definition}\label{def:move}
Two plabic graphs are  \emph{move-equivalent} if they can be obtained
from each other by moves (M1)-(M3).  The \emph{move-equivalence class}
of a given plabic graph $G$ is the set of all plabic graphs which are move-equivalent
to $G$.
A plabic graph 
is called \emph{reduced} if there is no graph in its move-equivalence
	in which there is a \emph{bubble}, that is, 
	two adjacent vertices $u$ and $v$ 
	which are connected by 
		more than one edge.  
\end{definition}

\begin{remark}\label{rem:bipartite}
It is sometimes convenient to regard the boundary vertices of a plabic graph as being black. Note that 
given any plabic graph $G$, with boundary vertices regarded as black vertices, we can always apply moves (M1)-(M3) to find a move-equivalent graph $G'$ which is 
bipartite.  
\end{remark}

\begin{definition}[Perfect orientation]\label{def:orientation}
A {\it perfect orientation\/} (respectively, a \emph{reverse perfect orientation}) $\OO$ of a plabic graph $G$ is a
choice of orientation of each edge such that each
internal black (respectively, white) vertex $u$ is incident to exactly one edge
directed away from $u$; and each internal white (respectively, black) vertex $v$ is incident
to exactly one edge directed towards $v$.
A plabic graph 
is called {\it perfectly orientable\/} if it has a perfect orientation.
Let $G_\OO$ denote the directed graph associated with a perfect orientation $\OO$ of $G$. The {\it source set\/} $I_\OO \subset [n]$ of a perfect orientation $\OO$ is the set of boundary vertices $i$ for which $i$
	is a source of the directed graph $G_\OO$. A \emph{flow} $F$ from $I_\OO$ to $J$ is a collection of vertex-disjoint directed paths in $G_\OO$ whose starting points are $I_\OO$ and whose ending points are $J$.
\end{definition}

    If $G$ has $n$ boundary vertices, and $k:=|I_\OO|$ is the size of a (equivalently, any) perfect orientation, then we say that $G$ \emph{has type $(k,n)$.}

\begin{proposition}\cite[Proposition 11.7, Lemma 11.10]{postnikov}\label{prop:positroid}
Let $G$ be a plabic graph. Then there is a matroid 
$\pos_G$ called a \emph{positroid} whose bases
are precisely the subsets
$$\{I \ \vert \ I = I_{\OO} \text{ for some perfect 
orientation }\OO\text{ of }G\}.$$

\end{proposition}

\begin{lemma} \cite[Lemma 3.2]{PSWv1} \label{acycliclemma}
Each reduced plabic graph $G$ on $[n]$ has an acyclic perfect orientation. Moreover,
for each total order $<_i$ on $[n]$ defined by 
$i <_i i+1 <_i \dots <_i n <_i 1 <_i \dots <_i i-1$,
$G$ has an acyclic perfect orientation $\OO$ whose source set $I_{\OO}$
is the lexicographically minimal basis with respect to $<_i$.
\end{lemma}

Given a perfect orientation $\OO$ of a plabic graph $G$ of type $(k,n)$,
one can describe explicitly an associated cell $S_G$ of $\Grk$; we therefore say that 
$G$ and $S_G$ have \emph{rank $k$}.  It turns out that the cell depends only 
on the move-equivalence class of $G$, not on $G$ or the choice of perfect orientation. 
 The description of the cell is particularly simple 
when $\OO$ is acyclic.

\begin{definition}[Path matrix and boundary measurement map]\label{def:pathmatrix}
Let $G$ be a
plabic graph of type $(k,n)$
with an acyclic perfect orientation $\OO$ of $G$.  
Let us associate a variable $x_e$ with each edge $e\in E(G)$ of $G$.
 For $i\in I_\OO$ and $j\notin I_\OO$,
define the {\it boundary measurement\/} $M_{ij}$ as the following expression:
$$
	M_{ij}:=\sum_{P} \left( \prod_{e\in P} x_e / \prod_{e'\in P} x_{e'} \right),
$$
where the sum is over all directed paths in $G_\OO$ that start 
and end at the boundary vertices $i$ and $j$,
the first product 
is over all edges $e$ in $P$ which are oriented white-to-black, and the second product is over all edges $e'$ in $P$ which are oriented black-to-white.
The \emph{path matrix} $A = A(G,\OO) = (a_{ij})$
is the unique $k \times n$ matrix with rows indexed in increasing order by
$I_{\OO}$ such that 
\begin{itemize}
\item The $k \times k$ submatrix of $A$ in the column set $I_\OO$ is the identity matrix.
\item For any 
	$i\in I_\OO$ and $j\notin I_\OO$, the entry $a_{ij}$ equals 
    $\pm M_{ij}$.
\item All Pl\"ucker coordinates of $A$ are subtraction-free expressions in the variables $x_e$.
\end{itemize}
\end{definition}

\begin{proposition}\label{prop:bdry-meas-pluckers}
Let $G, \OO$ be as above. The path matrix $A(G, \OO)$ has Pl\"ucker coordinates 
\[\lr{J}_A= \sum_{F} \left(\prod_{e \in F} x_e/ \prod_{e' \in F} x_{e'} \right) \]
where the sum is over flows $F$ from the source set $I_\OO$ to $J$, the first product 
is over all edges $e$ in $F$ which are oriented white-to-black, and the second product is over all edges $e'$ in $F$ which are oriented black-to-white. 
\end{proposition}

We can perform certain operations, called \emph{gauge transformations}, 
which preserve the boundary measurements $M_{ij}$.  The following lemma is 
easy to verify.

\begin{lemma}[Gauge transformations]\label{lem:gauge} Let $G, \OO$ be as in \cref{def:pathmatrix} and assign to each edge of $G$ a variable $x_e$. For each internal vertex $v$, choose a number $t_v \in \CC^*$, and for each edge $e=\{u,v\}$, let $x_e'= x_e t_u t_v$. Then the boundary measurements $M_{ij}$ obtained from edge weights $x_e$ are the same as those obtained from edge weights $x_e'$.
\end{lemma}

\begin{definition}\label{def:bdry-meas-map} Let $G, \OO$ be as in \cref{def:pathmatrix}.
    The map $$\mathbb{B}_{G}: (\CC^*)^{E(G)}/gauge \to \Gr_{k,n}$$ sending 
$\{x_e\} \mapsto A$ is the \emph{boundary measurement map}.
\end{definition}

We note that the boundary measurement map can also be defined in terms of a plabic graph and a perfect orientation containing cycles; however, the definition is a little more complicated. One can also equivalently define the boundary measurement map in terms of matchings, see \cite[Corollary 4.7]{PSWv1}. The matching formulation shows that while the path matrix $A(G, \OO)$ depends on the choice of perfect orientation $\OO$, the resulting point in the Grassmannian does not. This justifies the absence of $\OO$ from the notation $\mathbb{B}_G$ for the boundary measurement map.

\begin{theorem}[{\cite[Theorem 7.1]{MullerSpeyerTwist}}]
    Let $G$ be a reduced plabic graph. The boundary measurement map $\mathbb{B}_G$ is an isomorphism onto its image $T_G$. Moreover, $T_G$ is a Zariski-open and dense subset of $\Pi_G$. 
\end{theorem}

We call the image $T_G$ of $\mathbb{B}_G$ the \emph{boundary measurement torus}. Its dimension is $\dim T_G = |F(G)|-1$, where $F(G)$ is the set of faces of $G$.

It will sometimes be useful to ``use up" the gauge transformations by making some edge weights equal to 1.

\begin{definition} \label{def:gauge-fix}
A \emph{gauge-fix} of a plabic graph $G$ is a partition $E(G)=E_1 \sqcup E_{\neq 1}$ so that, if one uses gauge transformations to set the weight of all edges in $E_1$ to 1, boundary measurement becomes an isomorphism $\mathbb{B}_G:(\CC^*)^{E_{\neq 1}} \to T_G.$
\end{definition}

We have the following characterization of gauge-fixes.

\begin{proposition}[{\cite[Lemma 13.1]{lam2015totally}}] \label{prop:gauge-fix-characterization} A partition $E(G)=E_1 \sqcup E_{\neq 1}$ is a gauge-fix if and only if $E_1$ is a spanning forest of $G$ where each connected component contains exactly one boundary vertex.
\end{proposition}

We will need the following lemma relating boundary measurement tori for $G$ and $G^{op}$, where $G^{op}$ is the graph obtained from $G$ by switching the color of every vertex.

\begin{lemma}\label{lem:bdry-meas-G-vs-Gop}
Let $G$ be a plabic graph and let $N$ be the $n \times n$ diagonal matrix with diagonal $1, -1, \dots, (-1)^{n-1}$. Choose $\mathbf{x} \in (\CC^*)^{E(G)}/ gauge$ and set $V:= \mathbb{B}_G (\mathbf{x}) $ and $W := \mathbb{B}_{G^{op}} (\mathbf{x})$. Then $W=(VN)^{\perp} = V^{\perp}N$.
\end{lemma}

\begin{proof}
Consider an acyclic perfect orientation $\OO$ for $G$ with source set $I_\OO$. We use this orientation to compute $\mathbb{B}_G$. For $\mathbb{B}_{G^{op}}$, we use the perfect orientation $\OO'$ obtained by reversing the orientation of each edge. It has sources $I_\OO^c$.

Let $V = \mathbb{B}_G(\{x_e\})$. Then by \cref{prop:bdry-meas-pluckers}, $\lr{J}_V$ is the generating function for flows in $\OO$ from $I_\OO$ to $J$. On the other hand, let $W=\mathbb{B}_{G^{op}}(\{x_e\})$. The Pl\"ucker coordinate $\lr{J^c}_W$ is the generating function for flows in $\OO'$ from $I_\OO^c$ to $J^c$, and is equal to $\lr{J}_V$. Now \cite[Lemma 1.11(2)]{Karp} implies that $W=V^{\perp} N$. The equality $V^\perp N = (VN)^\perp$ is straightforward.
\end{proof}

\begin{theorem}\cite{postnikov}\label{thm:positroidcell}
Let $G$ and $\OO$ be as in \cref{def:pathmatrix}.
Then for any positive real values of the edge variables $x_e$,
the realizable matroid associated to the path matrix $A(G,\OO)$ is the 
{positroid}
$\mathcal{M}_G$ from 
\cref{prop:positroid}.
Let $S_G \subset \Grk$ denote
the set of all $k$-planes in $\R^n$ spanned by the
path matrices $A = A(G,\OO)$, as each edge variable $x_e$ ranges over $\R_{>0}$.
Then $S_G$ is homeomorphic to an open ball, called a \emph{positroid cell}.
If $G$ is reduced, then $S_G$ has dimension 
$r(G)-1$, where $r(G)$ is the number of regions of $G$.
We have that $\Grk$ is a disjoint union of positroid cells.
\end{theorem}

If $G$ is a perfectly orientable graph which fails to be reduced,
then the dimension of $S_G$ will be less than $r(G)$.

\begin{theorem}\cite[Theorem 18.5]{postnikov}\label{thm:closure}
If $S_G$ 
is a cell associated to plabic graph $G$,
then 
every cell in the closure of $S_G$ comes from a plabic graph 
obtained from $G$ by deleting some edges, and 
conversely,
if we delete some edges of $G$, obtaining a perfectly orientable graph $H$,
it corresponds 
 to a cell $S_H$ in the 
closure of $S_G$.
\end{theorem}

 \bibliographystyle{alpha}
        \bibliography{ClusterTilesPromotion}

\end{document}